\documentclass[12pt]{amsart}

\usepackage{amsmath,amsthm,amssymb,epsf}
\usepackage[hidelinks]{hyperref}
\usepackage{graphicx,epstopdf}
\usepackage[margin=1in]{geometry}
\usepackage[all]{xy}

\newtheorem{theorem}{Theorem}[section]
\newtheorem{lemma}[theorem]{Lemma}
\newtheorem{proposition}[theorem]{Proposition}
\newtheorem{corollary}[theorem]{Corollary}
\newtheorem{conjecture}[theorem]{Conjecture}

\theoremstyle{definition} 
\newtheorem{definition}[theorem]{Definition}
\newtheorem{warning}[theorem]{Warning}
\newtheorem{convention}[theorem]{Convention}
\newtheorem{remark}[theorem]{Remark}

\DeclareMathOperator{\Alex}{Alex}
\DeclareMathOperator{\Cat}{Cat}
\DeclareMathOperator{\ext}{ext}
\DeclareMathOperator{\external}{external}
\DeclareMathOperator{\glob}{global}
\DeclareMathOperator{\HFS}{HFS}
\DeclareMathOperator{\HFK}{HFK}
\DeclareMathOperator{\Hom}{Hom}
\DeclareMathOperator{\id}{id}
\DeclareMathOperator{\Ind}{Ind}
\DeclareMathOperator{\loc}{loc}
\DeclareMathOperator{\local}{local}
\DeclareMathOperator{\moving}{moving}
\DeclareMathOperator{\op}{op}
\DeclareMathOperator{\refi}{ref}
\DeclareMathOperator{\Rest}{Rest}
\DeclareMathOperator{\un}{un}
\DeclareMathOperator{\unmoving}{unmoving}

\newcommand{\A}{\mathcal{A}}
\newcommand{\B}{\mathcal{B}}
\newcommand{\co}{\colon}
\newcommand{\F}{\mathbb{F}}
\newcommand{\gl}{\mathfrak{gl}}
\newcommand{\I}{\mathcal{I}}
\newcommand{\Ib}{\mathbf{I}}
\newcommand{\K}{\mathcal{K}}
\newcommand{\m}{\mathfrak{m}}
\newcommand{\Rc}{\mathcal{R}}
\newcommand{\Sc}{\mathcal{S}}
\newcommand{\U}{\mathcal{U}}
\newcommand{\Uq}{\U_q(\gl(1|1))}
\newcommand{\X}{\mathcal{X}}
\newcommand{\Xt}{\widetilde{\X}}
\newcommand{\x}{\mathbf{x}}
\newcommand{\y}{\mathbf{y}}
\newcommand{\Z}{\mathbb{Z}}

\title[Singular crossings and the Kauffman-states functor]{Singular crossings and Ozsv{\'a}th-Szab{\'o}'s Kauffman-states functor}
\author[Andrew Manion]{Andrew Manion}
\thanks {This research was supported by an NSF MSPRF fellowship, grant number DMS-1502686, as well as an NSF FRG grant number DMS-1664240.}
\address{Department of Mathematics, USC, 3620 S. Vermont Ave., Los Angeles, CA 90089}
\email{amanion@usc.edu}

\begin{document}

\begin{abstract} Recently, Ozsv{\'a}th and Szab{\'o} introduced some algebraic constructions computing knot Floer homology in the spirit of bordered Floer homology, including a family of algebras $\B(n)$ and, for a generator of the braid group on $n$ strands, a certain type of bimodule over $\B(n)$. We define analogous bimodules for singular crossings. Our bimodules are motivated by counting holomorphic disks in a bordered sutured version of a Heegaard diagram considered previously by Ozsv{\'a}th, Stipsicz, and Szab{\'o}.
\end{abstract}

\maketitle

\section{Introduction}

Heegaard Floer homology, introduced by Ozsv{\'a}th and Szab{\'o} \cite{HFOrig,PropsApps}, is part of a relatively small family of topological invariants that are well-suited for distinguishing homeomorphic but non-diffeomorphic smooth $4$-manifolds. Physically, these invariants stem from $4$-dimensional topological quantum field theories (TQFTs). Along with the Chern--Simons TQFTs in $3$ dimensions used by Witten to interpret the Jones polynomial \cite{WittenChernSimonsJones}, 4d TQFTs distinguishing exotic $4$-manifolds (especially Donaldson theory) were a primary motivation for Atiyah's mathematical axiomatization of TQFTs in \cite{AtiyahTQFT}.

Since the mid-1990s, much interest has focused on ``extended'' TQFTs, which have extra structure beyond what Atiyah proposed. Many interesting TQFTs can be at least partially extended, and impressive classification results have been proved for ``fully extended'' TQFTs (see \cite{BDCobordism, Lurie, AyalaFrancis}). In Heegaard Floer homology, important steps toward a once-extended TQFT structure were taken by Lipshitz--Ozsv{\'a}th--Thurston under the name of bordered Floer homology \cite{LOTBorderedOrig}, which is now an active research program.

In \cite{OSzNew, OSzNewer, OSzHolo, OSzPong}, Ozsv{\'a}th--Szab{\'o} adapt the methods of bordered Floer homology to obtain an efficient algebraic description of knot Floer homology ($\HFK$), a Heegaard Floer invariant for knots and links defined originally in \cite{HFKOrig, RasmussenThesis}. In this paper we will refer to their theory as the ``Kauffman-states functor'' because it gives a functorial tangle invariant involving (partial) Kauffman states of a tangle projection, analogous to the states defined in \cite{FKT}. Indeed, their theory arises from holomorphic disk counts in a local variant of the Heegaard diagram from \cite{AltKnots}, shown in Figure~\ref{fig:NonsingularDiag}, whose Heegaard Floer generators correspond to Kauffman states. A computer program \cite{HFKCalc} based on the Kauffman-states functor is impressively fast and can compute $\HFK$ for most knots of up to $40$-$50$ crossings.

\begin{figure}
	\includegraphics[scale=0.625]{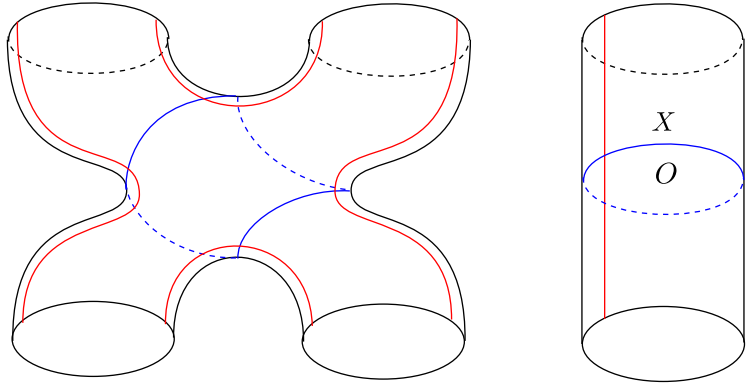}
	\caption{Local pieces of the Kauffman-states Heegaard diagram from \cite{AltKnots}.}
	\label{fig:NonsingularDiag}
\end{figure}

While $\HFK$ has a wealth of topological applications, it is also interesting when asking about the relationship between the 4d and 3d examples motivating Atiyah's axioms. Heegaard Floer homology began its life on the 4d side, but it has surprising connections with the 3d Chern--Simons theory associated to the Lie superalgebra $\mathfrak{gl}(1|1)$. For example, the Euler characteristic of $\HFK$ recovers the Alexander polynomial of a knot, which can be viewed as arising from $\mathfrak{gl}(1|1)$ similarly to how the Jones polynomial arises from $\mathfrak{sl}(2)$. More generally, one expects that Heegaard Floer homology (including its extended TQFT aspects) ``categorifies'' the $\mathfrak{gl}(1|1)$ Chern--Simons TQFT; we are interested in making this vague statement as precise and complete as possible.

A reasonable once-extended version of the fact that $\chi(HFK)$ recovers the Alexander polynomial says that to a tangle, bordered Floer homology should associate a bimodule categorifying the $\Uq$-linear map associated to the tangle. Out of various possible ways to define such bimodules (see e.g. \cite{PV, EPV}), Ozsv{\'a}th--Szab{\'o}'s Kauffman-states functor has the advantage of a certain minimality property: the computations of \cite{ManionDecat} imply that generators of their bimodule for a crossing are in bijection with nonzero matrix entries in the corresponding $\Uq$-linear map, with no cancellation upon taking the Euler characteristic.

The algebraic methods typically used to study categorified Chern--Simons theories define bimodules for crossings as mapping cones on morphisms between singularized and resolved crossings. Arguments of Ozsv{\'a}th--Szab{\'o} \cite{OSzCube} and Manolescu \cite{ManolescuCube} imply that the mapping cone relationship holds up to homotopy equivalence for the knot Floer complexes of a closed knot. This relationship is a natural way to study connections between $\HFK$ and HOMFLY-PT homology; an unresolved conjecture of Dunfield--Gukov--Rasmussen \cite{DGR} posits a spectral sequence from HOMFLY-PT homology to $\HFK$. 

To localize the construction of \cite{OSzCube,ManolescuCube} to tangles using bordered Floer homology, one imagines cutting the ``planar'' Heegaard diagram of \cite{OSzCube, ManolescuCube} into pieces to which bordered Floer invariants can be assigned by counting holomorphic disks. Such a localization could be of interest in categorification as well as in Heegaard Floer homology, especially given work in preparation of Rapha{\"e}l Rouquier and the author \cite{ManionRouquier} situating Khovanov's categorification of $\U^+_q(\gl(1|1))$ \cite{KhovOneHalf} and tensor products of its higher representations in the flexible framework of bordered and cornered Floer homology. However, the diagrams obtained by naively decomposing the planar diagram are not easy to analyze using bordered Floer techniques; see \cite{ManionDiagrams} for an alternate decomposition that may have better properties.

Alishahi--Dowlin \cite{AlishahiDowlin} recently introduced differential graded bimodules for singular crossings over Ozsv{\'a}th--Szab{\'o}'s algebras from \cite{OSzNew}. Using these bimodules and a mapping cone construction, they define bimodules for nonsingular crossings which appear to be nontrivially related to Ozsv{\'a}th--Szab{\'o}'s Kauffman-states bimodules (Dowlin, private communication). Alishahi--Dowlin's bimodules were motivated by holomorphic disk counts in ``global'' Heegaard diagrams for closed singular knots, rather than ``local'' diagrams as in bordered Floer homology, and it is not clear whether there is a local Heegaard diagram giving rise to Alishahi--Dowlin's bimodule for a singular crossing. 

We work from the other direction in this paper, starting from a Heegaard diagram for a singular crossing introduced by Ozsv{\'a}th--Stipsicz--Szab{\'o} in \cite{OSSz}. One local version of this diagram is shown in Figure~\ref{fig:SingularDiag}. A stabilized version of the diagram was given in \cite{ManionDiagrams} (based on ideas of Ozsv{\'a}th--Szab{\'o}); we slightly modify the stabilized diagram by adding corners to view it as a bordered sutured Heegaard diagram as in \cite{Zarev}. See Figure~\ref{fig:StabilizedSingularDiag} below for an illustration.

\begin{figure}
	\includegraphics[scale=0.625]{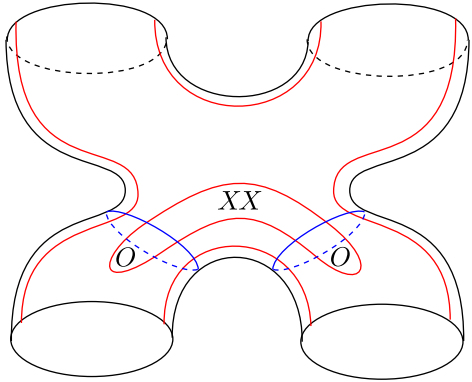}
	\caption{The local piece of the Kauffman-states Heegaard diagram for a singular crossing from \cite{OSSz}.}
	\label{fig:SingularDiag}
\end{figure}

Motivated by the local disk-counting techniques used by Ozsv{\'a}th--Szab{\'o} to define their bimodules, we define a bimodule $\X^{DA}$ for a singular crossing between two strands. More specifically, $\X^{DA}$ is a type of $\A_{\infty}$ bimodule known as a $DA$ bimodule in bordered Floer homology. The right $\A_{\infty}$ actions on $\X^{DA}$ are quite elaborate, with nonzero $m_3$, $m_4$, and $m_5$ terms appearing. Our first result is that $\X^{DA}$ satisfies the appropriate structure relation.

\begin{theorem}\label{thm:IntroDAWellDefined}
The $DA$ bimodule $\X^{DA}$ shown graphically in Figures~\ref{fig:UnsimplifiedDABimod12} and \ref{fig:UnsimplifiedDABimod3} is a valid $DA$ bimodule.
\end{theorem}

Theorem~\ref{thm:IntroDAWellDefined} appears to be very restrictive; much of the complicated structure of $\X^{DA}$ is forced by well-definedness and a few simple holomorphic disk counts.

Unlike Ozsv{\'a}th--Szab{\'o}'s bimodules for nonsingular crossings, the Heegaard diagram used in this paper cannot quite produce a minimal categorification in the sense mentioned above. One reason is given by the empty rectangle in Figure~\ref{fig:SingularDiag}; algebraically, a few generators in $\X^{DA}$ can be cancelled to produce a simplified bimodule $\Xt^{DA}$. Our next result describes the result of this simplification.

\begin{theorem}
The $DA$ bimodule $\Xt^{DA}$ shown graphically in Figures~\ref{fig:SimplifiedDABimod12} and \ref{fig:SimplifiedDABimod3} is a valid $DA$ bimodule and is homotopy equivalent to $\X^{DA}$.
\end{theorem}

We also define a bimodule $\X_i^{DD}$ for a singular crossing between strands $i$ and $i+1$ out of $n$ strands. This bimodule is of a simpler type, called a $DD$ bimodule (it has a differential but no higher $\A_{\infty}$ actions). While the lack of higher $\A_{\infty}$ actions makes $DD$ bimodules easier to work with, the compatibility between internal and external components of the differential on $\X_i^{DD}$ is intricate.

\begin{theorem}
The $DA$ bimodule $\X_i^{DD}$ constructed in Section~\ref{sec:GlobalDD} is a valid $DD$ bimodule.
\end{theorem}

As described below in Section~\ref{sec:CanonicalDD}, there is a natural procedure $- \boxtimes \K$ for obtaining $DD$ bimodules from $DA$ bimodules over the algebras in question. We show that the bimodule $\X^{DD} = \X_1^{DD}$ for a singular crossing between two strands satisfies $\X^{DD} \cong \X^{DA} \boxtimes \K$. We also show that $\Xt^{DA}$ and the simplification $\Xt^{DD}$ of $\X^{DD}$ admit symmetries corresponding to Ozsv{\'a}th--Szab{\'o}'s symmetries $\Rc$ and $o$ on their algebras and bimodules (the unsimplified bimodules $\X^{DA}$ and $\X^{DD}$ only admit a symmetry corresponding to the composition $\Rc o$).

\subsection*{Further directions}

A natural question is whether the bimodules $\X^{DA}$ or $\X_i^{DD}$ extend to invariants of more general singular braids, possibly with many crossings (positive, negative, and singular), and of singular tangles. The bimodules $\X_i^{DD}$ are not well-adapted to building such an extension; $DA$ bimodules would be desired. One could try to define a $DA$ bimodule $\X_i^{DA}$ from $\X_i^{DD}$ by taking a tensor product with the quasi-inverse of the bimodule $\K$ considered below in Section~\ref{sec:CanonicalDD}, but the resulting bimodule would be infinitely generated and hard to manipulate unless an explicit low-rank model for this quasi-inverse could be found.

Abstractly, global extensions of $DA$ bimodules (such as from $\X^{DA}$ to $\X_i^{DA}$) should be covered by higher representation theory via its connection with cornered (or twice-extended) Heegaard Floer homology \cite{DM, DLM}. This connection will be explored by Rouquier and the author in \cite{ManionRouquier}. However, even without a general theory of global extensions, it would be interesting to define $DA$ bimodules $\X_i^{DA}$, extend to more general singular tangles by taking tensor products, and prove invariance.

Assuming that bimodules $\X_i^{DA}$ as described above exist and give singular tangle invariants, the next question would be whether these invariants recover a known Heegaard Floer invariant when restricted to (closed) singular knots. Since the bimodules constructed in this paper do not count holomorphic disks through $O$ basepoints, one expects them to recover the invariant $\widetilde{\HFS}$. One would like to prove this identification by establishing a more local identification between $\X_i^{DA}$ and certain generalized bordered Floer bimodules, defined analytically from Heegaard diagrams. 

In fact, Ozsv{\'a}th--Szab{\'o} plan to define such analytic bimodules for Heegaard diagrams satisfying certain conditions in \cite{OSzHolo}.
\begin{conjecture}\label{conj:AnalyticBimods}
The bimodules $\X^{DA}$ and $\X^{DD}$ constructed in this paper are isomorphic to $DA$ and $DD$ bimodules defined analytically from the Heegaard diagram in Figure~\ref{fig:StabilizedSingularDiag} below using techniques such as those in \cite{OSzHolo}, given certain analytic choices. The bimodule $\X_i^{DD}$ is isomorphic to the $DD$ bimodule of a ``globally extended'' (to the left and right) variant of the diagram, again for some analytic choices.
\end{conjecture}
If Conjecture~\ref{conj:AnalyticBimods} is true, then this paper can be viewed as an explicit computation of the generalized bordered invariant for certain Heegaard diagrams representing singular crossings. We also note that based partially on evidence from \cite{ManionRouquier}, we believe that the proper topological interpretation of this generalization is as a generalization of bordered sutured Floer homology, not of bordered Floer homology as formulated e.g. in \cite{LOTBorderedOrig}; see \cite{MMW2} as well for a relationship between Ozsv{\'a}th--Szab{\'o}'s algebras and certain generalized bordered sutured algebras.

It would be interesting to compare the bimodules in this paper with Alishahi--Dowlin's bimodules for singular crossings in \cite{AlishahiDowlin}. On the surface, the bimodules look very dissimilar; for example, Alishahi--Dowlin's bimodules do not have any higher $\A_{\infty}$ actions, but they are not defined over Kozsul dual algebras like our $DD$ bimodules. However, it is possible that higher $\A_{\infty}$ actions arise when simplifying Alishahi--Dowlin's bimodules using homological perturbation theory.

\subsection*{Organization}

In Section~\ref{sec:BorderedAlg}, we review the relevant algebraic background from bordered Floer homology. In Section~\ref{sec:OSzAlgs}, we review the algebras over which our bimodules are defined. In Section~\ref{sec:LocalDD}, we describe the local $DD$ bimodule $\X^{DD}$; this is the simplest of the bimodules appearing in this paper, although we postpone verifying the $DD$ structure relation because we can deduce it from the $DA$ structure relation for $\X^{DA}$. We also describe the simplified bimodule $\Xt^{DD}$ and its symmetries.

In Section~\ref{sec:LocalDA}, we define $\X^{DA}$ and prove that it is a well-defined $DA$ bimodule. We also compute the result $\Xt^{DA}$ of simplifying $\X^{DA}$, and we describe the symmetries of $\X^{DA}$ and $\Xt^{DA}$. Finally, in Section~\ref{sec:GlobalDD}, we define the globally extended $DD$ bimodule $\X_i^{DD}$, and in Section~\ref{sec:GlobalDDWellDefined} we prove its well-definedness.

\subsection*{Acknowledgments}
I would like to thank Akram Alishahi, Nathan Dowlin, Aaron Lauda, Robert Lipshitz, Peter Ozsv{\'a}th, and Rapha{\"e}l Rouquier for useful discussions. I would especially like to thank Zolt{\'a}n Szab{\'o} for teaching me about the Kauffman-states functor. 

\section{Bordered algebra}\label{sec:BorderedAlg}

\subsection{Differential graded algebras}
Let $\F := \F_2$, and let $\I = \F^{\times N}$ be a direct product of finitely many copies of $\F$. We define an $\I$-\emph{algebra} to be a ring $\A$ (with unit) equipped with a ring homomorphism $\I \to \A$. In other words, we consider algebras in $(\I,\I)$-bimodules, or equivalently $\F$-linear categories with $N$ objects. We will sometimes refer to $\I$ as the \emph{idempotent ring} of $\A$.

A \emph{differential ring} is a ring equipped with an abelian-group endomorphism $\partial$ satisfying $\partial^2=0$ and the Leibniz rule $\partial(ab) = \partial(a)b + a \partial(b)$. A \emph{differential} $\I$-\emph{algebra} is a differential ring $\A$ equipped with a homomorphism of differential rings from $\I$ to $\A$, where $\I$ has zero differential.

\begin{definition}
Let $G$ be a group and let $\lambda$ be a central element of $G$. A \emph{differential} $(G,\lambda)$--\emph{graded ring} (or \emph{dg ring}) is a differential ring $\A$ equipped with a decomposition $\A = \oplus_{g \in G} \A_g$ of abelian groups such that $1 \in \A_e$, $\partial(\A_g) \subset \A_{\lambda g}$, and $\mu(\A_g \otimes \A_{g'}) \subset \A_{gg'}$, where $e$ is the identity element of $G$ and $\partial$ and $\mu$ denote the differential and multiplication of $\A$. A \emph{differential} $(G,\lambda)$--\emph{graded algebra} (or \emph{dg algebra}) over $\I$ is a dg ring equipped with a homomorphism of dg rings from $\I$ to $\A$, where $\I$ is concentrated in degree $e$ and has zero differential.
\end{definition}

\begin{remark}
In this paper, $G$ will always be $\Z \oplus M$ where $M$ is a finitely generated free abelian group. We will refer to $M$ as the intrinsic grading group or the Alexander multi-grading group. The $\Z$--component of the $G$--degree of an algebra element will be called its \emph{homological degree} or \emph{Maslov degree}, and the $M$--component will be called its \emph{intrinsic degree} or \emph{Alexander multi-degree}. The element $\lambda$ of $G$ will be $(1,0)$. Bordered Floer homology considers algebras with more general gradings by $\Z$ central extensions of finitely generated free abelian groups.
\end{remark}

\begin{remark}\label{rem:HomGr}
Our conventions contrast with the usual conventions in bordered Heegaard Floer homology (see \cite[Definition 2.5.2]{LOTBimodules}), in which differentials have degree $\lambda^{-1}$.  
\end{remark}

A $(G,\lambda)$-graded dg algebra $\A$ over $\I$ may be viewed as an $\F$-linear $(G,\lambda)$-graded dg category $\Cat_{\A}$ whose objects are the $N$ primitive idempotents $\Ib \in \I$, with $\Hom_{\Cat_{\A}}(\Ib,\Ib') = \Ib' \cdot \A \cdot \Ib$.

\begin{definition}
If $\A$ is a dg algebra over $\I$, an \emph{augmentation} of $\A$ is a dg algebra homomorphism $\varepsilon\co \A \to \I$ with $\varepsilon(\Ib) = \Ib$ for all $i \in \I$. If $\A$ is an augmented dg algebra with augmentation $\varepsilon$, let $\A_+ = \ker(\varepsilon)$.
\end{definition}
We will not need to use Keller's slightly more general definition from \cite[Section 10.2]{KellerDerivingDG}.

\subsection{\texorpdfstring{$DD$}{DD} bimodules}

\begin{convention}
Tensor products $\otimes$ are over $\F$ unless otherwise specified.
\end{convention}

Let $\A$ and $\A'$ be dg algebras over $\I$ and $\I'$ respectively, with gradings by $(G,\lambda)$ and $(G',\lambda')$. Define $G \times_{\lambda} G' = \frac{G \times G'}{\lambda = \lambda'}$; this group has a distinguished central element $[\lambda] = [\lambda']$.  We can view $\A \otimes \A'$ as a $G \times_{\lambda} G'$--graded dg algebra over $\I \otimes \I' \cong \F^{\times N} \otimes \F^{\times N'} \cong \F^{\times (NN')}$. 

Let $S$ be a left $G \times_{\lambda} G'$--set, or equivalently a set with commuting left actions of $G$ and $G'$ such that the actions of $\lambda$ and $\lambda'$ agree. 

\begin{definition}
A left module $X$ over $\I \otimes \I'$ is called $S$-\emph{graded} if $X$ is equipped with a decomposition $X \cong \oplus_{s \in S} X_s$ of left $\I \otimes \I'$-modules. We define $X[1]$ by $(X[1])_s = X_{\lambda s}$; in other words, $X[1]$ is $X$ with its degrees shifted downward by $\lambda$.
\end{definition}

If $X$ is an $S$-graded left module over $\I \otimes \I'$ and $s \in S$, we can view $(\A \otimes \A') \otimes_{\I \otimes \I'} X$ as an $S$-graded left $\I \otimes \I'$-module by
\begin{equation}\label{eq:ExtendingSGradings}
((\A \otimes \A') \otimes_{\I \otimes \I'} X)_s := \bigoplus_{g \in G \times_{\lambda} G', s' \in S: gs' = s} (\A \otimes \A')_g \otimes_{\I \otimes \I'} X_{s'}.
\end{equation}

\begin{definition}[cf. Definition 2.2.55 of \cite{LOTBimodules}]
An $S$--graded (left, left) $DD$ \emph{bimodule} over $(\A,\A')$ is a pair $(X,\delta^1)$ where $X$ is an $S$-graded left module over $\I \otimes \I'$ and
\[
\delta^1\co X \to (\A \otimes \A') \otimes_{\I \otimes \I'} X[1]
\]
is an $\I \otimes \I'$-linear map that preserves $S$--degrees and satisfies the $DD$ bimodule relation
\[
(\partial \otimes \id_X) \circ \delta^1 + (\mu \circ \id_X) \circ (\id_{\A \otimes \A'} \otimes \delta^1) \circ \delta^1 = 0.
\]
Here $\partial$ and $\mu$ denote the differential and multiplication on the tensor product algebra $\A \otimes \A'$.
\end{definition}

We say that $X$ is \emph{finitely generated} if it is finite-dimensional over $\F$; all $DD$ bimodules we consider are finitely generated. Following \cite{LOTBimodules}, we will sometimes write $X = {^{\A,\A'}X}$ when we want to include the algebras $\A,\A'$ in the notation for $X$.

\begin{remark}\label{rem:DDBimodGradings}
In this paper, with $G = \Z \oplus M$ and $G' = \Z \oplus M'$, we have $G \times_{\lambda} G' \cong \Z \oplus M \oplus M'$. The left $G \times_{\lambda} G'$--sets we will consider always have the form $S = \Z \oplus \overline{M}$ where $\overline{M}$ is another finitely generated free abelian group and the $G \times_{\lambda} G'$ action on $S$ comes from homomorphisms from $M$ and $M'$ to $\overline{M}$.
\end{remark}

\begin{definition}\label{def:IdemTerminology}
Let $X$ be a $DD$ bimodule and let $x \in X$. We say that $x$ has a \emph{unique pair of idempotents} if there exist unique primitive idempotents $\Ib$ and $\Ib'$ of $\I$ and $\I'$ with $(\Ib \otimes \Ib') \cdot x \neq 0$.
\end{definition}

By choosing a basis over $\F$ for $((\Ib \otimes \Ib') \cdot X)_s$ for each $s \in S$ and pair of primitive idempotents $(\Ib,\Ib')$, we can choose an $\F$--basis for $X$ such that each basis element is homogeneous and has a unique pair of idempotents. If $x$ has a unique pair of idempotents $(\Ib,\Ib')$, we will call $\Ib$ and $\Ib'$ the first and second idempotents of $x$ respectively. 

\begin{remark} In bordered Heegaard Floer homology, $DD$ bimodules often arise from certain Heegaard diagrams. Such a diagram gives not just a $DD$ bimodule but also a natural choice of basis satisfying the above properties, given by certain sets of intersection points between curves in the diagram. While the $DD$ operation $\delta^1$ may depend on analytic choices, the basis determined by the diagram does not. See Figure~\ref{fig:GensAsIntPts} for an example. The same is true for other types of bimodules, such as the $DA$ bimodules discussed below.
\end{remark}

\subsection{Homotopy equivalences of \texorpdfstring{$DD$}{DD} bimodules}

\begin{definition}[cf. Definition 2.2.55 of \cite{LOTBimodules}]
Let $X$ and $Y$ be $S$--graded $DD$ bimodules over $(\A,\A')$. A $DD$ \emph{bimodule morphism} $f\co X \to Y$ is an $\I \otimes \I'$-linear map
\[
f\co X \to (\A \otimes \A') \otimes_{\I \otimes \I'} Y,
\]
not necessarily grading-preserving. We say that $f$ has degree $k$ if $f$ maps $X_s$ into $((\A \otimes \A') \otimes_{\I \otimes \I'} Y)_{\lambda^k s}$ for all $s \in S$. Note that while the degree of a morphism may not be unique, the notion of having degree $k$ for a fixed $k$ is unambiguous. 

The $DD$ morphisms from $X$ to $Y$ of degree $k$, for all $k \in \Z$, form a $\Z$-graded chain complex with differential
\begin{alignat*}{2}
&\partial f:= \qquad &&(\partial \otimes \id) \circ f \\
&&&+ (\mu \otimes \id) \circ (\id \otimes \delta^1) \circ f \\
&&&+ (\mu \otimes \id) \circ (\id \otimes f) \circ \delta^1,
\end{alignat*}
where $\delta^1$ denotes the $DD$ bimodule operation on $X$ or $Y$ as appropriate.
\end{definition}

Let $f\co X \to Y$ and $g\co Y \to Z$ be $DD$ bimodule morphisms. We define the composition $g \circ f$ to be the morphism from $X \to Z$ given by the map $(\mu \otimes \id) \circ (\id_{\A \otimes \A'} \otimes g) \circ f$. With this composition, we can form $\Z$--graded dg categories of $S$--graded $DD$ bimodules and $DD$ morphisms. Homotopy equivalence of $DD$ bimodules is defined using these dg categories.

\subsection{\texorpdfstring{$DA$}{DA} bimodules}

Let $\A$ and $\A'$ be dg algebras over $\I$ and $\I'$ with gradings by $(G,\lambda)$ and $(G',\lambda')$. Let $G'^{\op}$ denote $G'$ with its order of multiplication reversed; a left $G \times_{\lambda} (G'^{\op})$--set is equivalently a set with commuting left and right actions of $G$ and $G'$ respectively such that the actions of $\lambda$ and $\lambda'$ agree. 

Let $S$ be such a set and let $X$ be an $S$--graded (left, right) bimodule over $(\I,\I')$. We can view both $\A \otimes_{\I} X$ and $X \otimes_{\I'} \A'^{\otimes(i-1)}$ (for $i \geq 1$) as $S$--graded $(\I,\I')$--bimodules as in equation~\eqref{eq:ExtendingSGradings}.

\begin{convention}
When discussing $DA$ bimodules, all tensor products in symbols like $\A'^{\otimes(i-1)}$ are over the idempotent ring $\I'$ of $\A'$.
\end{convention}

\begin{definition}[cf. Definition 2.2.43 of \cite{LOTBimodules}]
A $DA$ \emph{bimodule} over $(\A,\A')$ is an $S$--graded (left, right) bimodule $X$ over $(\I,\I')$ equipped with, for $i \geq 1$, an $(\I,\I')$-bilinear degree-preserving map
\[
\delta^1_i\co X \otimes_{\I'} (\A')^{\otimes (i-1)} \to \A \otimes_{\I} X[2-i]
\]
satisfying the $DA$ structure relations
\begin{alignat*}{2}
&(\partial \otimes \id) \circ \delta^1_i &&+ \sum_{j=1}^i (\mu \otimes \id) \circ (\id \otimes \delta^1_{i-j+1}) \circ (\delta^1_j \otimes \id) \\
&&&+ \sum_{j=1}^{i-1} \delta^1_i \circ (\id \otimes \partial'_j) \\ 
&&&+ \sum_{j=1}^{i-2} \delta^1_{i-1} \circ (\id \otimes \mu'_{j,j+1}) \\
&= 0
\end{alignat*}
for all $i \geq 1$. Here, $\partial'_j\co (\A')^{\otimes (i-1)} \to (\A')^{\otimes (i-1)}$ is the differential $\partial'$ of $\A'$ on the $j^{th}$ tensor factor and the identity on the rest. Similarly, $\mu'_{j,j+1}$ is the multiplication on tensor factors $j$ and $j+1$ of $(\A')^{\otimes(j-1)}$. Finally, $\mu$ and $\partial$ are the multiplication and differential on $\A$.

\end{definition}

To indicate the algebras $\A$ and $\A'$ in the notation for $X$, we will sometimes write $X = {^{\A}}X_{\A'}$, following \cite{LOTBimodules}.

A $DA$ bimodule $X$ is called \emph{strictly unital} if, for all $x \in X$, we have $\delta^1_2(x \otimes 1) = x$ and $\delta^1_i(x \otimes a'_1 \otimes \cdots \otimes a'_{i-1})$ is zero whenever $i > 2$ and any $a'_j$ is the identity element $1$ of $\A'$. We also call $X$ \emph{finitely generated} if $X$ is finite-dimensional over $\F$. All $DA$ bimodules we consider are strictly unital and finitely generated. 

As in Definition~\ref{def:IdemTerminology}, let $X$ be a $DA$ bimodule and let $x \in X$. We say that $x$ has a \emph{unique pair of idempotents} if there exist unique primitive idempotents $\Ib$ and $\Ib'$ of $\I$ and $\I'$ with $\Ib \cdot x \cdot \Ib' \neq 0$. We call $\Ib$ and $\Ib'$ the left and right idempotents of $x$ respectively.

\subsection{Homotopy equivalences of \texorpdfstring{$DA$}{DA} bimodules}

\begin{definition}[cf. Definition 2.2.43 of \cite{LOTBimodules}]
Let $X$ and $Y$ be $S$--graded $DA$ bimodules over $(\A,\A')$, and assume that we have an augmentation on $\A'$. A $DA$ \emph{bimodule morphism} $f\co X \to Y$ is a collection of maps
\[
f_i\co X \otimes_{\I'} (\A'_+)^{\otimes (i-1)} \to \A \otimes_{\I} X[1-i],
\]
not necessarily grading-preserving. We say that $f$ has degree $k$ if $f_i$ maps $(X \otimes_{\I'} (\A'_+)^{\otimes (i-1)})_s$ into $(\A \otimes_{\I} X[1-i])_{\lambda^k s}$ for all $i,s$. As with $DD$ bimodule morphisms, this notion is well-defined although the degree of a morphism may not be unique.

The $DA$ morphisms from $X$ to $Y$ of degree $k$, for all $k \in \Z$, form a $\Z$-graded chain complex with differential
\begin{alignat*}{2}
&(\partial f)_i := \qquad &&(\partial \otimes \id) \circ f_i \\
&&&+ \sum_{j=1}^i (\mu \otimes \id) \circ (\id \otimes \delta^1_{i-j+1}) \circ (f_j \otimes \id) \\
&&&+ \sum_{j=1}^i (\mu \otimes \id) \circ (\id \otimes f_{i-j+1}) \circ (\delta^1_j \otimes \id) \\
&&&+ \sum_{j=1}^{i-1} f_i \circ (\id \otimes \partial'_j) \\
&&&+ \sum_{j=1}^{i-2} f_{i-1} \circ (\id \otimes \mu'_{j,j+1}). \\
\end{alignat*}
\end{definition}

Let $f\co X \to Y$ and $g\co Y \to Z$ be $DA$ bimodule morphisms. The composition $g \circ f\co X \to Z$ is defined by
\[
(g \circ f)_i := \sum_{j=1}^i (\mu \otimes \id) \circ (\id \otimes g_{i-j+1}) \circ (f_j \otimes \id)
\]
With this composition, we can form $\Z$--graded dg categories of $S$--graded $DA$ bimodules and $DA$ morphisms. Homotopy equivalence of $DA$ bimodules is defined using these dg categories.

\subsection{Box tensor products}\label{sec:BoxTensor}

The language of bordered Floer homology includes a concrete model for the derived tensor product $\widetilde{\otimes}$ called the box tensor product $\boxtimes$. In this paper we will only discuss the box tensor product of a $DA$ bimodule with a $DD$ bimodule, and we can put a simplifying assumption on the grading set of the $DD$ bimodule.

Let $\A, \A'$, and $\A''$ be dg algebras over $\I$, $\I'$, and $\I''$ with gradings by $(G,\lambda), (G',\lambda')$, and $(G',\lambda')$ respectively. Let $X$ be a $DA$ bimodule over $(\A,\A')$, graded by a left $G \times_{\lambda} (G'^{\op})$-set $S$. Let $\K$ be a (left, left) $DD$ bimodule over $(\A',\A'')$, graded by $G'$ as a (left, left) $G' \times_{\lambda} G'$-set where both actions of $G'$ are given by left multiplication. For $j \geq 2$, define $\delta^j\co \K \to (\A' \otimes \A'')^{\otimes j} \otimes_{\I' \otimes \I''} \K[j]$ by $\delta^j := (\id^{\otimes(j-1)} \otimes \delta^1) \circ \cdots \circ \delta$.

\begin{definition}\label{DADDBoxTensorDef}
Assuming the sum in the expression for $\delta^{\boxtimes}$ below is finite, the \emph{box tensor product} $X \boxtimes \K$ is the $S$-graded (left, left) $DD$ bimodule over $(\A,\A'')$ which is defined to be
\[
(X \boxtimes \K)_s := \oplus_{s'g' = s} X_{s'} \otimes_{\I'} \K_{g'}
\]
as an $S$-graded left $\I \otimes \I''$ module and equipped with the $DD$ bimodule operation
\begin{align*}
\delta^{\boxtimes} := \sum_{j \geq 1} X \otimes \K &\xrightarrow{\id \otimes \delta^{j-1}} X \otimes (\A' \otimes \A'')^{\otimes(j-1)} \otimes \K \\
&\xrightarrow{\cong} X \otimes \A'^{\otimes(j-1)} \otimes \A''^{\otimes(j-1)} \otimes \K \\
&\xrightarrow{\delta^1_j \otimes \id \otimes \id} \A \otimes X \otimes \A''^{\otimes(j-1)} \otimes \K \\
&\xrightarrow{\textrm{swap and multiply } \A'' \textrm{ factors}} \A \otimes \A'' \otimes X \otimes \K.
\end{align*}
The sum is finite if the $DA$ bimodule operations $\delta^1_j$ on $X$ vanish for sufficiently large $j$; this condition holds for all $DA$ bimodules in this paper.

\end{definition}
If $\K$ is a (left, left) $DD$ bimodule over $(\A'',\A')$ instead of over $(\A',\A'')$, we can also define a $DD$ bimodule $X \boxtimes \K$ over $(\A,\A'')$, modifying Definition~\ref{DADDBoxTensorDef} appropriately. 

\subsection{Graphical depictions of \texorpdfstring{$DD$}{DD} bimodules and \texorpdfstring{$DA$}{DA} bimodules}\label{sec:Graphical}

Let $X$ be a finitely generated $DD$ bimodule; choose a basis for $X$ consisting of grading-homogeneous elements with unique pairs of idempotents. We can depict $X$ using a directed graph with labeled edges. Vertices of the graph are basis elements of $X$. There is an edge from $x$ to $y$ when $(\sum a \otimes a') \otimes y$ appears in the basis expansion of $\delta^1(x)$ for some nonzero element $\sum a \otimes a'$ of $\A \otimes \A'$; in this case, we label the edge $\sum a \otimes a'$. See Figure~\ref{fig:UnsimplifiedLocalBimod12} for an example of the graph of a $DD$ bimodule; there are also examples in \cite{OSzNew}.

In fact, it will often be useful to define $DD$ bimodules in terms of their directed graphs. Let $\Gamma$ be a directed graph with vertices labeled with names, degrees, and idempotents, and edges labeled by elements of $\A \otimes \A'$. We say $\Gamma$ has \emph{compatible grading and idempotent data} if:
\begin{itemize}
\item each vertex and algebra element in an edge label is homogeneous with a unique pair of idempotents, 
\item for every edge from $x$ to $y$ labeled by $\sum a \otimes a'$, the degree of each summand of $(\sum a \otimes a') \otimes y$ (multiplied by $\lambda^{-1}$) agrees with the degree of $x$, and
\item for every edge from $x$ to $y$ labeled by $\sum a \otimes a'$, the right idempotent of each element $a$ (respectively $a'$) is the first idempotent (respectively second idempotent) of $y$, and the left idempotent of $a$ (respectively $a'$) is the first idempotent (respectively second idempotent) of $x$.
\end{itemize}

If $\Gamma$ has compatible data, we have a $DD$ bimodule $X$ corresponding to $\Gamma$ when the following condition is satisfied: for any two vertices $x$ and $y$ of $G$, let $S_{1}$ be the set of composable pairs of two edges $x \xrightarrow{\sum a \otimes a'} z \xrightarrow{\sum b \otimes b'} y$, and let $S_{2}$ be the set of all single edges $x \xrightarrow{\sum c \otimes c'} y$ directly from $x$ to $y$. Then the sum over $S_{1}$ of the products $\sum \sum ab \otimes a'b'$ of the edge labels, plus the sum over $S_{2}$ of the derivatives $\sum \partial(c)\otimes c' + c\otimes \partial(c')$ of the edge labels, must be zero in $\A \otimes \A'$.

If $X$ is a finitely generated $DA$ bimodule, we can depict $X$ similarly. Besides choosing a basis for $X$, we also choose an $\F$-basis for $\A'$; both bases should consist of homogeneous elements with unique left and right idempotents, and the basis for $\A'$ should contain the primitive idempotents $\Ib' \in \I'$.

Define a directed graph whose vertices are basis elements of $X$ (with grading and idempotent data recorded explicitly or implicitly). There is an edge from $x$ to $y$ when $a \otimes y$ appears in the basis expansion of $\delta^1_i(x \otimes a'_1 \otimes \cdots \otimes a'_{i-1})$ for some element $a \neq 0$ of $\A$ and basis element $a'_1 \otimes \cdots \otimes a'_{i-1}$ of $(\A')^{\otimes(i-1)}$. Following \cite{OSzNew}, we label the edge from $x$ to $y$ with the formal sum of expressions $a \otimes (a'_1, \ldots, a'_{i-1})$ over all terms of $\delta^1(x)$ as above. It is often useful to view these formal sums, or sums in the algebra inputs $a'_j$, as a shorthand for multiple edges between the same vertices. When $i = 1$, we omit the parentheses (so the labels look like $a \otimes a'$), and when $i = 0$, we omit the $\otimes$ symbol and just write $a$ for the label. Note that $a \otimes a'$ has different meanings in graphs describing $DD$ and $DA$ bimodules.

See Figure~\ref{fig:UnsimplifiedDABimod12} for an example. For visual convenience, $\delta^1_1$ edges are drawn in blue, $\delta^1_3$ edges are drawn in green, $\delta^1_4$ edges are drawn in red, and $\delta^1_5$ actions are drawn in teal. The $DA$ bimodules in this paper do not have nontrivial $\delta^1_2$ actions or any $\delta^1_i$ actions for $i > 5$.

\begin{warning}\label{warn:IdentityEdges}
Since the $DA$ bimodules $X$ we consider are strictly unital, for every vertex $x$ of the directed graph there is an edge from $x$ to itself with label $\Ib \otimes \Ib'$ where $\Ib$ and $\Ib'$ are the left and right idempotents of $x$. To save space, we will omit these edges from the diagrams.
\end{warning}

The condition for a directed graph $\Gamma$ labeled as above to have compatible grading and idempotent data is as follows:
\begin{itemize}
\item each vertex and algebra element in an edge label is homogeneous with a unique pair of idempotents, 
\item for every edge from $x$ to $y$ labeled by $a \otimes (a'_1, \ldots, a'_{i-1})$, the degree of $a \otimes y$ (multiplied by $\lambda^{i-1}$) agrees with the degree of $x \otimes a'_1 \otimes \cdots \otimes a'_{i-1}$, and
\item for every edge from $x$ to $y$ labeled by $a \otimes (a'_1, \ldots, a'_{i-1})$, the left idempotent of $x$ is the left idempotent of $a$, the right idempotent of $a'_{i-1}$ is the right idempotent of $y$, the right idempotent of $x$ is the left idempotent of $a'_1$, for $1 \leq j \leq i-2$ the right idempotent of $a'_j$ is the left idempotent of $a'_{j+1}$, and the right idempotent of $a$ is the left idempotent of $y$.
\end{itemize}

If $\Gamma$ has compatible data, we have a $DA$ bimodule $X$ corresponding to $\Gamma$ when the following condition is satisfied: for vertices $x$ and $y$ of $\Gamma$, and algebra basis elements $a'_1, \ldots, a'_{i-1}$, let $S_1$ be the set of composable pairs of edges 
\[
x \xrightarrow{a \otimes (a'_1, \ldots, a'_{j-1})} z \xrightarrow{b \otimes (a'_j,\ldots,a'_{i-1})} y.
\]
Let $S_2$ be the set of single edges $x \xrightarrow{c \otimes (a''_1,\ldots,a''_{i-1})} y$ where some $a''_j$ is a nonzero term in the basis expansion of $\partial(a'_j)$ and all other $a''_k$ equal $a'_k$. Let $S_3$ be the set of single edges $x \xrightarrow{c \otimes (a''_1,\ldots,a''_{i-2})} y$ where some $a''_j$ is a nonzero term in the basis expansion of the product of $a'_j$ and $a'_{j+1}$, and $a''_k = a'_k$ for $k < j$, $a''_k = a'_{k+1}$ for $k > j$. Finally, let $S_4$ be the set of single edges $x \xrightarrow{c \otimes (a'_1, \ldots, a'_{i-1})} y$. For $X$ to be a $DA$ bimodule, the sum over $S_1$ of the product $ab$ of the edge labels, plus the sum over $S_2$ and $S_3$ of the edge labels $c$, plus the sum over $S_4$ of the derivatives $\partial(c)$ of the edge labels, must be zero in $\A$ for all $(x,y,a'_1,\ldots, a'_{i-1})$.

\begin{remark}\label{rem:IgnoreIdentityEdges}
When checking the above condition, assume that the only edges whose algebra inputs contain a primitive idempotent are the edges mentioned in Warning~\ref{warn:IdentityEdges}. Also assume that no primitive idempotent appears in the basis expansion of $\partial(a')$ for any $a' \in \A'$; these conditions will always be satisfied in this paper. It follows that one can ignore the edges of Warning~\ref{warn:IdentityEdges} when checking that $\Gamma$ defines a $DA$ bimodule. Indeed, these edges only contribute to the sets $S_1$ and $S_3$ by assumption, and the contribution to $S_1$ cancels the contribution to $S_3$.
\end{remark}

\subsection{Simplifying \texorpdfstring{$DD$}{DD} and \texorpdfstring{$DA$}{DA} bimodules}\label{sec:HowToSimplify}
For convenience, we briefly summarize a convenient way of constructing homotopy equivalences between $DD$ bimodules and between $DA$ bimodules, based on homological perturbation theory.
\begin{definition}\label{def:DDCancellablePair}
Let $\I$ be a finite direct product of copies of $\F$, let $\A$ be a dg algebra over $\I$, and let $X = (X,\delta^1)$ be a $DD$ bimodule over $\A$. A \emph{cancellable pair} in $X$ is a pair of basis elements $x$ and $y$ of $X$ such that $\delta^1(x) = (\Ib \otimes \Ib') \otimes y + \sum_{x_i \neq y} (a_i \otimes a'_i) x_i$ where $x_i$ are basis elements of $X$ and $\Ib, \Ib'$ are the first and second idempotents of $y$ (or equivalently of $x$).
\end{definition}

Note that a cancellable pair in a $DD$ bimodule $X$ corresponds to an edge labeled $\Ib \otimes \Ib'$ in the directed graph of $X$.

\begin{definition}\label{def:DDCancelling}
Let $(x,y)$ be a cancellable pair in $X$. Define a $DD$ bimodule $X'$ (under a condition to be specified) as follows: let $\Gamma$ be the labeled directed graph of $X$, and let $\Gamma'_0$ be $\Gamma$ with $x$, $y$, and all edges adjacent to them removed. For each ``zig-zag'' pattern of edges 
\[
z \xrightarrow{a_1 \otimes a'_1} y \xleftarrow{\Ib \otimes \Ib'} x \xrightarrow{a_2 \otimes a'_2} y \xleftarrow{\Ib\otimes \Ib'} x \xrightarrow{a_3 \otimes a'_3} \cdots \xrightarrow{a_n \otimes a'_n} w
\]
in $\Gamma$, where neither $z$ nor $w$ is equal to $x$ or $y$ and none of the edges labeled $a_j \otimes a'_j$ are the edge $x \xrightarrow{\Ib \otimes \Ib'}$ being cancelled, add a new edge in $\Gamma'_0$ from $z$ to $w$ with label $(a_1 \cdots a_n) \otimes (a'_1 \cdots a'_n)$. We need to assume that only finitely many edges with nonzero labels are produced by this step; in this case, we say we have a \emph{valid cancellable pair}. If the cancellable pair is valid, call the result $\Gamma'$ and let $X'$ be the $DD$ bimodule associated to $\Gamma'$.
\end{definition}

It is a standard result that $X'$ is a well-defined $DD$ bimodule with same grading structure as $X$, and that $X'$ is homotopy equivalent to $X$; we sketch a proof for completeness. One can check that $\Gamma'$ has compatible gradings and idempotents. Let $T\co X \to X$ send $y$ to $(\Ib \otimes \Ib') \otimes x$ and send all other basis elements to zero. Schematically, the differential $\delta'^1$ may be written as $\sum_{i=0}^{\infty} \delta(T\delta)^i$, or equivalently as $\sum_{i=0}^{\infty} (\delta T)^i\delta$, where $\delta$ represents all terms of $\delta^1$ except for the term $(\Ib \otimes \Ib') \otimes y$ of $\delta^1(x)$. The sums are finite since the cancellable pair is valid. 

The expression $(\mu \otimes \id) \circ (\id \otimes \delta'^1) \circ \delta'^1$ evaluates to
\[
\sum_{i=0}^{\infty} \sum_{j=0}^{\infty} (\delta T)^i \delta \delta (T\delta)^j = \sum_{i=0}^{\infty} \sum_{j=0}^{\infty} (\delta T)^i (\partial(\delta)) (T\delta)^j.
\]
This is the same result we get from $(\partial \otimes \id)\circ \delta'^1$, so $X'$ is a $DD$ bimodule.

\begin{definition}\label{def:TypeDHomEqMaps}
In the notation above, define maps $f\co X' \to X$, $g\co X \to X'$, and $h\co X \to X$ by:
\begin{itemize}
\item $f := \sum_{i = 0}^{\infty} (T \delta)^i$
\item $g := \sum_{i = 0}^{\infty} (\delta T)^i$
\item $h := \sum_{i = 0}^{\infty} T(\delta T)^i = \sum_{i \geq 0} (T \delta)^i T$,
\end{itemize}
(we implicitly ignore any outputs of $g$ involving the basis elements $x$ and $y$ that we are cancelling). The sums are finite since the cancellable pair is valid.
\end{definition}

\begin{proposition}\label{prop:TypeDHE}
The maps of Definition~\ref{def:TypeDHomEqMaps} satisfy $\partial(f) = 0$, $\partial(g) = 0$, $g \circ f = \id_{X'}$, and $f \circ g + \id_X = d(h)$. Thus, $X$ is homotopy equivalent to $X'$.
\end{proposition}
One can check Proposition~\ref{prop:TypeDHE} with the same type of manipulations used to check that $X'$ is a $DD$ bimodule.

\begin{definition}
Let $X$ be a $DA$ bimodule over $(\A,\A')$. A \emph{cancellable pair} in $X$ is a pair of basis elements $x$ and $y$ of $X$ such that $\delta^1_1(x) = \Ib \otimes y + \sum_{x_i \neq y} a_i \otimes x_i$ where $x_i$ are basis elements of $X$ and $\Ib$ is the left idempotent of $y$ (or equivalently of $x$). Graphically, a cancellable pair in $X$ corresponds to an edge labeled $\Ib$ in the directed graph of $X$.
\end{definition}

Let $X$ be a strictly unital $DA$ bimodule with a cancellable pair $(x,y)$. We say that $(x,y)$ is \emph{valid} if it is a valid cancellable pair in the underlying type $D$ structure of $X$. The definition of type $D$ structures is such that a $DD$ bimodule over $(\A,\A')$ is exactly a type $D$ structure over $\A \otimes \A'$; see \cite[Section 2.2.3]{LOTBimodules} for more details. Graphically, the underlying type $D$ structure of $X$ is obtained by discarding the edges representing $\delta^1_i$ actions for $i > 1$. Valid cancellable pairs in type $D$ structures are defined as in Definitions \ref{def:DDCancellablePair} and \ref{def:DDCancelling}.

Let $(x,y)$ be a valid cancellable pair in $X$. Let $\Gamma$ be the labeled directed graph of $X$, and let $\Gamma'_0$ be $\Gamma$ with $x$, $y$, and all edges adjacent to them removed. For each ``zig-zag'' pattern of edges 
\[
z \xrightarrow{a_1 \otimes (a'_{1,1}, \ldots, a'_{1,i_1-1})} y \xleftarrow{\Ib} x \xrightarrow{a_2 \otimes (a'_{2,1}, \ldots, a'_{2,i_2-1})} \cdots \xrightarrow{a_n \otimes (a'_{n,1}, \ldots, a'_{n,i_n-1})} w
\]
in $\Gamma$, where neither $z$ nor $w$ is equal to $x$ or $y$ and no edge labeled $a_j \otimes (a'_{j,1}, \ldots, a'_{j,i_j-1})$ is the edge $x \xrightarrow{\Ib \otimes \Ib'}$ being cancelled, add a new edge in $\Gamma'_0$ from $z$ to $w$ with label 
\[
a_1 \cdots a_n \otimes (a'_{1,1}, \ldots, a'_{1,i_1-1}, a'_{2,1}, \ldots, a'_{2,i_2-1},\ldots,a'_{n,1}, \ldots, a'_{n,i_n-1}).
\]
Denote the result by $\Gamma'$. Infinitely many new edges may have been added, but only finitely many were added for any given sequence of algebra inputs. 

Ozsv{\'a}th--Szab{\'o} give a version of homological perturbation theory for $DA$ bimodules in \cite[Lemma 2.12]{OSzNew}. This lemma implies that $\Gamma'$ defines a $DA$ bimodule $X'$ which is homotopy equivalent to $X$.

\subsection{Duals of modules and bimodules}\label{sec:DualsOfModules}
When discussing certain symmetries in the bimodules we define below, it will useful to have a notion of duality for $DD$ and $DA$ bimodules. Given $(G,\lambda)$ and a left $G$--set $S$, let $S^*$ denote the right $G$--set with the same elements as $S$ (written $s^*$ for $s \in S$) and $G$--action defined by $s^*g = (g^{-1}s)^*$; see \cite[Definition 2.5.19]{LOTBimodules}. We have $(S^*)^* \cong S$. We may equivalently view $S^*$ as a left $G^{\op}$--set.

If $\A$ is a dg algebra graded by $(G,\lambda)$, we may view $\A^{\op}$ as a dg algebra graded by $(G^{\op},\lambda)$ with $(\A^{\op})_G = \A_g$ for $g \in G = G^{\op}$ (where the identification $G = G^{\op}$ is of sets without multiplication).

\begin{definition}[Definition 2.2.31 of \cite{LOTBimodules}]
Let $X$ be a finitely generated $DD$ bimodule over $(\A,\A')$. Suppose that $\A$ and $\A'$ are graded by $(G,\lambda)$ and $(G',\lambda')$, and that $X$ is graded by a left $G \times_{\lambda} G'$--set $S$. The \emph{dual} $X^{\vee}$ of $X$ (called the opposite of $X$ by Lipshitz--Ozsv{\'a}th--Thurston) is a $DD$ bimodule over $(\A^{\op},\A'^{\op})$, graded by $S^*$, which is defined by $(X^{\vee})_{s^*} := \Hom_{\F}(X_s,\F)$ as a vector space over $\F$. Identifying $\I$ with $\I^{\op}$ and $\I'$ with $\I'^{\op}$, we can view $\A^{\op}$ and $\A'^{\op}$ as dg algebras over $\I$ and $\I'$. The left $\I \otimes \I'$-module structure on $X^{\vee}$ is given by 
\[
(\Ib \otimes \Ib') \cdot \phi = \phi((\Ib \otimes \Ib') \cdot -)
\]
for $\phi \in X^{\vee} = \Hom(X,\F)$. The $DD$ bimodule operation $(\delta^{\vee})^1$ on $X^{\vee}$ sends $\phi \in \Hom_{\F}(X,\F)$ to 
\[
(\id \otimes \phi) \circ \delta^1 \in \Hom_{\F}(X, \A^{\op} \otimes \A'^{\op}) \cong (\A^{\op} \otimes \A'^{\op}) \otimes_{\F} \Hom_{\F}(X,\F).
\]
Given a basis for $X$ satisfying the usual conditions, we can choose the dual basis for $X^{\vee}$. In these bases, the labeled directed graph of $X^{\vee}$ is obtained from that of $X$ by reversing all the arrows and interpreting their labels as elements of $\A^{\op}$. It follows that $X^{\vee}$ is a valid $DD$ bimodule over $(\A^{\op},\A'^{\op})$ and that we may naturally identify $X$ and $(X^{\vee})^{\vee}$ (one could also check these statements algebraically; see \cite[Lemma 2.2.32]{LOTBimodules}).
\end{definition}

\begin{remark}\label{rem:DualGrading}
In this paper, with $G = \Z \oplus M$, $G' = \Z \oplus M'$, and $S = \Z \oplus \overline{M}$ where $\overline{M}$ is equipped with homomorphisms from $M$ and $M'$, there is an isomorphism of right $G \times_{\lambda} G'$--sets (i.e. left $(G \times_{\lambda} G')^{\op}$--sets or just left $(G \times_{\lambda} G')$--sets since the groups involved are abelian) from $S^*$ to $S$ sending $(n,m)$ to $(-n,-m)$. For the $DD$ bimodules $X$ we consider, graded by $S$, we will view $X^{\vee}$ as graded by the left $G \times_{\lambda} G'$--set $S$ via this isomorphism. Concretely, all degrees (both homological and intrinsic) of dual basis elements of $X^{\vee}$ are the negatives of the degrees for the corresponding basis elements of $X$. 
\end{remark}

\begin{definition}[Definition 2.2.53 of \cite{LOTBimodules}]
Let $X$ be a finitely generated $DA$ bimodule over $(\A,\A')$. The \emph{dual} $X^{\vee}$ of $X$ (called the opposite of $X$ by Lipshitz--Ozsv{\'a}th--Thurston) is a $DA$ bimodule over $(\A^{\op},\A'^{\op})$ defined as an $S^*$-graded vector space over $\F$ by $(X^{\vee})_{s^*} := \Hom_{\F}(X_s,\F)$. We view $X^{\vee}$ as a (left, right) bimodule over $(\I,\I')$ by $\Ib \cdot \phi = \phi(\Ib \cdot -)$ and $\phi \cdot \Ib' = \phi(- \cdot \Ib')$ for $\Ib \in \I$, $\Ib' \in \I'$ and $\phi \in X^{\vee} = \Hom(X,\F)$. 

The $DA$ operations $(\delta^{\vee})^1_i$ on $X^{\vee}$ send
\[
\phi \otimes a'_1 \otimes \cdots \otimes a'_{i-1}
\]
where $\phi \in \Hom(X,\F)$, to 
\[
(\id \otimes \phi) \circ \delta^1_i(-,a'_{i-1},\ldots,a'_i) \in \Hom_{\F}(X, \A^{\op}) \cong \A^{\op} \otimes \Hom_{\F}(X,\F).
\]
Given a basis for $X$ satisfying the usual conditions, we can choose the dual basis for $X^{\vee}$. In these bases, the labeled directed graph of $X^{\vee}$ is obtained from that of $X$ by reversing all the arrows and reversing the order of each sequence of algebra basis elements appearing as an input label for a $\delta^1_i$ action with $i \geq 3$. It follows that $X^{\vee}$ is a valid $DA$ bimodule over $(\A^{\op},\A'^{\op})$ and that we may naturally identify $X$ and $(X^{\vee})^{\vee}$.
\end{definition}

As with $DD$ bimodules, we can view the dual $X^{\vee}$ of any of the $S$--graded $DA$ bimodules $X$ considered in this paper as being $S$--graded itself; the degree of a basis element of $X^{\vee}$ is the negative of the degree of the corresponding basis element of $X$.

\begin{remark}
In \cite[Definition 2.2.53]{LOTBimodules}, Lipshitz--Ozsv{\'a}th--Thurston define the opposite of a $DA$ bimodule over $(\A,\A')$ to be an $AD$ bimodule over $(\A',\A)$. Such a bimodule can equivalently be viewed as a $DA$ bimodule over $(\A^{\op},\A'^{\op})$; this perspective is responsible for the reversal of order in the sequences above. While Lipshitz--Ozsv{\'a}th--Thurston only discuss the opposites of finitely generated $DA$ bimodules over $(\A,\A')$ when $\A$ and $\A'$ are also finite-dimensional over $\F$, the above definition gives a valid $DA$ bimodule even when $\A$ and $\A'$ are infinite-dimensional over $\F$ (Ozsv{\'a}th--Szab{\'o}'s algebras $\B(n)$ and $\B^!(n)$, discussed below, are infinite--dimensional over $\F$ for all $n \geq 1$). The identification $X \cong (X^{\vee})^{\vee}$ depends only on the finite-dimensionality of $X$.
\end{remark}

\section{Ozsv{\'a}th--Szab{\'o}'s algebras}\label{sec:OSzAlgs}

\subsection{Definitions}

We review some dg algebras introduced by Ozsv{\'a}th--Szab{\'o} in \cite{OSzNew}. The algebra $\B(n)$ mentioned above is a direct sum of algebras $\B(n,k)$ for $0 \leq k \leq n$. The following generators-and-relations description of $\B(n,k)$ is shown in \cite{MMW1} to be equivalent to the definition given in \cite{OSzNew}.

\begin{definition}[cf. Section 3 of \cite{OSzNew}, Theorem 1.2 of \cite{MMW1}]\label{OSzAlgDef}
For $n \geq 0$ and $0 \leq k \leq n$, the dg algebra $\B(n,k)$ is the algebra of the following quiver, with zero differential and with relations to be specified (the grading will be discussed in Section~\ref{sec:AlgGrading}). The vertices of the quiver are subsets $\x$ of $\{0,\ldots,n\}$ with $|\x| = k$. If $\x \cap \{i-1,i\} = \{i-1\}$ and $\y = \x \setminus \{i-1\} \cup \{i\}$ for some $i$, we add an arrow from $\x$ to $\y$ with label $R_i$. If $\x \cap \{i-1,i\} = \{i\}$ and $\y = \x \setminus \{i\} \cup \{i-1\}$ for some $i$, we add an arrow from $\x$ to $\y$ with label $L_i$. For all vertices $\x$ of the quiver and all $i \in \{1, \ldots, n\}$, we add an arrow from $\x$ to itself with label $U_i$.  The relations are of the following type:
\begin{itemize}
\item $[R_i,R_j] = 0$, $[L_i,L_j] = 0$, and $[R_i,L_j] = 0$ if $|i-j|>1$
\item $[U_i,A] = 0$  for all labels $A$
\item $L_i R_i = U_i$, $R_i L_i = U_i$
\item $R_i R_{i+1} = 0$, $L_i L_{i-1} = 0$
\item $U_i = 0$ at a vertex $\x$ if $\x \cap \{i-1,i\} = \varnothing$.
\end{itemize}
These relations (except those of the form $U_i = 0$) are assumed to hold whenever any composable pair of arrows has labels appearing as a term in one of the above expressions. Note that not every nonzero $U_i$ generator may be factored as $R_i L_i$ or $L_i R_i$.
\end{definition}

Following Ozsv{\'a}th--Szab{\'o}'s terminology in \cite{OSzNew}, we will refer to vertices $\x$ of the above quiver as \emph{I-states}. For each I-state $\x$, we have an element of $\B(n,k)$ corresponding to the empty sequence of arrows based at $\x$. We call this element $\Ib_{\x}$, following \cite[Section 3.1]{OSzNew}; these elements form a set of $\binom{n}{k}$ orthogonal idempotents of $\B(n,k)$ summing to the identity. Via this set of idempotents, we can view $\B(n,k)$ as an algebra over the ring $\I(n,k) \cong \F^{\times \binom{n}{k}}$ generated by the elements $\Ib_{\x}$.

We may define a related algebra $\B^!(n,k)$ by adding edges $C_i$, for $1 \leq i \leq n$, from each vertex $\x$ to itself. We impose the relations $C_i^2 = 0$ and $[C_i,A] = 0$ for all labels $A$. We give $\B^!(n,k)$ a differential by defining $\partial(C_i) = U_i$. We may also view $\B^!(n,k)$ as an algebra over $\I(n,k)$. In the notation of \cite{OSzNew}, we have $\B(n,k) = \B(n,k,\varnothing)$ and $\B^!(n,k) = \B(n,k,\Sc)$ where $\Sc = \{1,\ldots,n\}$.

Let $\B(n) := \oplus_{k=0}^n \B(n,k)$ and $\B^!(n) := \oplus_{k=0}^n \B^!(n,k)$. We may view $\B(n)$ and $\B'(n)$ as dg algebras over $\I(n) := \prod_{k=0}^n \I(n,k)$.

\begin{remark}
Ozsv{\'a}th--Szab{\'o} show in \cite[Section 2.9]{OSzNew} that $\B'(n,n-k)$ is Koszul dual to $\B(n,k)$.
\end{remark}

\subsection{Basis}

To define the $DA$ bimodule $\X^{DA}$ over $\B(2,k)$ graphically, we need a basis for $\B(2,k)$. The basis should contain the idempotents $\Ib_{\x}$ and consist of homogeneous elements with unique pairs of idempotents. Ozsv{\'a}th--Szab{\'o} give such a basis for $\B(n,k)$ in \cite[Proposition 3.7]{OSzNew} The basis elements can be described in terms of quiver generators as in \cite[Corollary 4.12]{MMW1}; we do this for $n = 2$ below.

\begin{proposition}
A basis for $\B(2,0)$ is given by the single element $\{\Ib_{\varnothing}\}$. For $\B(2,1)$:
\begin{itemize}
\item A basis for $\Ib_{\{0\}} \B(2,1) \Ib_{\{0\}}$ is given by elements $U_1^i$ for $i \geq 0$.
\item A basis for $\Ib_{\{0\}} \B(2,1) \Ib_{\{1\}}$ is given by elements $R_1 U_1^i$ for $i \geq 0$.
\item We have $\Ib_{\{0\}} \B(2,1) \Ib_{\{2\}} = 0$.
\item A basis for $\Ib_{\{1\}} \B(2,1) \Ib_{\{0\}}$ is given by elements $L_1U_1^i$ for $i \geq 0$.
\item A basis for $\Ib_{\{1\}} \B(2,1) \Ib_{\{1\}}$ is given by elements $\Ib_{\{1\}}$, $U_1^i$ for $i \geq 1$, and $U_2^i$ for $i \geq 1$.
\item A basis for $\Ib_{\{1\}} \B(2,1) \Ib_{\{2\}}$ is given by elements $R_2 U_2^i$ for $i \geq 0$.
\item We have $\Ib_{\{2\}} \B(2,1) \Ib_{\{0\}} = 0$.
\item A basis for $\Ib_{\{2\}} \B(2,1) \Ib_{\{1\}}$ is given by elements $L_2 U_2^i$ for $i \geq 0$.
\item A basis for $\Ib_{\{2\}} \B(2,1) \Ib_{\{2\}}$ is given by elements $U_2^i$ for $i \geq 0$.
\end{itemize}
For $\B(2,2)$:
\begin{itemize}
\item A basis for $\Ib_{\{0,1\}} \B(2,2) \Ib_{\{0,1\}}$ is given by elements $U_1^i U_2^j$ for $i,j \geq 0$.
\item A basis for $\Ib_{\{0,1\}} \B(2,2) \Ib_{\{0,2\}}$ is given by elements $R_2 U_1^i U_2^j$ for $i,j \geq 0$.
\item A basis for $\Ib_{\{0,1\}} \B(2,2) \Ib_{\{1,2\}}$ is given by elements $R_2 R_1 U_1^i U_2^j$ for $i,j \geq 0$.
\item A basis for $\Ib_{\{0,2\}} \B(2,2) \Ib_{\{0,1\}}$ is given by elements $L_2 U_1^i U_2^j$ for $i,j \geq 0$.
\item A basis for $\Ib_{\{0,2\}} \B(2,2) \Ib_{\{0,2\}}$ is given by elements $U_1^i U_2^j$ for $i,j \geq 0$.
\item A basis for $\Ib_{\{0,2\}} \B(2,2) \Ib_{\{1,2\}}$ is given by elements $R_1 U_1^i U_2^j$ for $i,j \geq 0$.
\item A basis for $\Ib_{\{1,2\}} \B(2,2) \Ib_{\{0,1\}}$ is given by elements $L_1 L_2 U_1^i U_2^j$ for $i,j \geq 0$.
\item A basis for $\Ib_{\{1,2\}} \B(2,2) \Ib_{\{0,2\}}$ is given by elements $L_1 U_1^i U_2^j$ for $i,j \geq 0$.
\item A basis for $\Ib_{\{1,2\}} \B(2,2) \Ib_{\{1,2\}}$ is given by elements $U_1^i U_2^j$ for $i,j \geq 0$.
\end{itemize}
Finally, a basis for $\B(2,3) = \Ib_{\{1,2,3\}} \B(2,3) \Ib_{\{1,2,3\}}$ is given by elements $U_1^i U_2^j$ for $i,j \geq 0$.
\end{proposition}

\subsection{Gradings}\label{sec:AlgGrading}
The algebra $\B(n,k)$ has a grading by $(\frac{1}{2}\Z)^n$ that we call the refined Alexander multi-grading, as well as a grading by $\Z^{2n}$ considered in \cite{MyKSQuiver, MMW1, MMW2} and called the unrefined Alexander multi-grading. Our terminology here contrasts with that of \cite{MyKSQuiver} but is more in line with the use of ``refined'' and ``unrefined'' in \cite{LOTBorderedOrig}. Indeed, it is shown in \cite{MMW2} that $\B(n,k)$ is quasi-isomorphic to a generalized bordered strands algebra $\A(n,k)$ such that the $(\frac{1}{2}\Z)^n$ grading and $\Z^{2n}$ grading correspond respectively to the usual refined and unrefined gradings on strands algebras.

We will refer to the standard generators of $\Z^n$ as $e_1, \ldots, e_n$ and the standard generators of $\Z^{2n}$ as $\tau_1, \beta_1, \ldots, \tau_n, \beta_n$. 
\begin{definition}
The unrefined Alexander multi-degrees of the generators of $\B(n,k)$ are defined as follows:
\begin{itemize}
\item $\deg^{\un}(R_i) = \tau_i$
\item $\deg^{\un}(L_i) = \beta_i$
\item $\deg^{\un}(U_i) = \tau_i + \beta_i$.
\end{itemize}
\end{definition}

Define a homomorphism $\eta$ from the unrefined group $\Z^{2n}$ to the refined group $(\frac{1}{2}\Z)^n$ by sending $\tau_i$ and $\beta_i$ to $\frac{e_i}{2}$. The refined grading on $\B(n,k)$ may be obtained by applying $\eta$ to the unrefined degrees. We may further collapse the refined grading into a single Alexander grading by applying the sum map from $(\frac{1}{2}\Z)^n$ to $\frac{1}{2}\Z$. For convenience, we list the refined and single Alexander degrees of generators below.

\begin{proposition}
The refined Alexander multi-degrees and single Alexander degrees of the generators of $\B(n,k)$ are given as follows:
\begin{itemize}
\item $R_i$ and $L_i$ have refined Alexander multi-degree $\frac{e_i}{2}$ and single Alexander degree $\frac{1}{2}$.
\item $U_i$ has refined Alexander multi-degree $e_i$ and single Alexander degree $1$.
\end{itemize}
\end{proposition}

The algebra $\B^!(n,k)$ also has an unrefined grading; the following grading is appropriate for defining $DD$ bimodules over $\B(n,k) \otimes \B^!(n,n-k)$ as we will do below.
\begin{definition}
The unrefined Alexander multi-grading on the generators of $\B^!(n,k)$ is defined as follows:
\begin{itemize}
\item $\deg^{\un}(R_i) = -\beta_i$
\item $\deg^{\un}(L_i) = -\tau_i$
\item $\deg^{\un}(U_i) = \deg^{\un}(C_i) = -\tau_i - \beta_i$.
\end{itemize}
\end{definition}

\begin{proposition}
The refined Alexander multi-degrees and single Alexander degrees of the generators of $\B^!(n,k)$ are given as follows:
\begin{itemize}
\item $R_i$ and $L_i$ have refined Alexander multi-degree $-\frac{e_i}{2}$ and single Alexander degree $-\frac{1}{2}$.
\item $U_i$ and $C_i$ have refined Alexander multi-degree $-e_i$ and single Alexander degree $-1$.
\end{itemize}
\end{proposition}

The algebra $\B(n,k)$ has no homological grading; equivalently, it is placed in homological degree zero. For $\B^!(n,k)$, we reverse the signs of Ozsv{\'a}th--Szab{\'o}'s homological grading as mentioned in Remark~\ref{rem:HomGr}. In our conventions, $R_i$, $L_i$, and $C_i$ have degree $1$, while $U_i$ has degree $2$. 

\begin{remark}
Lipshitz--Ozsv{\'a}th--Thurston's algebras in \cite{LOTBorderedOrig, LOTBimodules} usually do not admit homological gradings by $\Z$. These complications do not arise for the algebras we consider or for their strands-algebra versions $\A(n,k)$ from \cite{MMW2}.
\end{remark}

\subsection{Symmetries}\label{sec:AlgSymm}

Ozsv{\'a}th--Szab{\'o} point out two symmetries of their algebras, including $\B(n,k)$ and $\B^!(n,k)$, in \cite[Section 3.6]{OSzNew}. The first symmetry, which they call $\Rc$, gives dg ring endomorphisms of $\B(n,k)$ and $\B^!(n,k)$ with $\Rc^2 = \id$ (this symmetry was called $\rho$ in \cite{MMW1, MMW2}). It acts as a nontrivial involution on the primitive idempotents; explicitly, $\Rc$ restricts to a map from $\I(n,k)$ to itself sending $\Rc(\Ib_{\x}) = \Ib_{\y}$ where $\y := \{n-i \vert i \in \x\}$. On the generators of the algebras $\B(n,k)$ and $\B^!(n,k)$, $\Rc$ sends:
\begin{itemize}
\item $R_i \leftrightarrow L_{n-i}$
\item $U_i \leftrightarrow U_{n-i}$
\item $C_i \leftrightarrow C_{n-i}$.
\end{itemize}
The second symmetry, which Ozsv{\'a}th--Szab{\'o} call $o$, gives dg algebra homomorphisms from $\B(n,k)$ to $\B(n,k)^{\op}$ and from $\B^!(n,k)$ to $\B^!(n,k)^{\op}$ with $o^2 = \id$. It acts trivially on $\I(n,k)$. On the algebra generators, it sends:
\begin{itemize}
\item $R_i \leftrightarrow L_i$
\item $U_i \leftrightarrow U_i$
\item $C_i \leftrightarrow C_i$.
\end{itemize}

Both $\Rc$ and $o$ preserve the homological grading and are compatible with corresponding symmetries of the unrefined grading group. Let $\Rc\co \Z^{2n} \to \Z^{2n}$ send $\tau_i$ to $\beta_{n-i}$ and $\beta_i$ to $\tau_{n-i}$. Let $o\co \Z^{2n} \to \Z^{2n}$ send $\tau_i$ to $\beta_i$ and $\beta_i$ to $\tau_i$. Then for an algebra element $a$ homogeneous with respect to the unrefined grading, we have $\deg^{\un}(\Rc(a)) = \Rc(\deg^{\un}(a))$ and $\deg^{\un}(o(a)) = o(\deg^{\un}(a))$.

\subsection{A canonical \texorpdfstring{$DD$}{DD} bimodule}\label{sec:CanonicalDD}

In \cite[Section 3.7]{OSzNew}, Ozsv{\'a}th--Szab{\'o} define a $DD$ bimodule over $(\B(n,k),\B^!(n,n-k))$. They use this bimodule to show that $\B(n,k)$ and $\B^!(n,n-k)$ are Koszul dual. We review the definition of this bimodule below.

\begin{definition}[Section 3.7 of \cite{OSzNew}]\label{def:KoszulDualizing}
If $\x \subset \{1,\ldots,n\}$ is an I-state, the complement of $\x$ will refer to the I-state $\{1, \ldots, n\} \setminus \x$. Let ${^{\B(n,k),\B^!(n,n-k)}}\K$ be the $\F$--vector space formally spanned by elements $k_{\x}$ for $\x \subset \{0,\ldots,1\}$ with $|\x| = k$. Define a left action of $\I(n,k) \otimes \I(n,n-k)$ on $\K$ by
\[
(\Ib_{\x'} \otimes \Ib_{\x''}) \cdot k_{\x} := \begin{cases} k_{\x} & \x' = \x \textrm{ and } \x'' \textrm{ is the complement of }\x \\ 0 & \textrm{otherwise.} \end{cases}
\]
Define a map $\delta^1\co \K \to \B(n,k) \otimes \B^!(n,n-k) \otimes_{\I(n,k) \otimes \I(n,n-k)} \K$ by
\[
\delta^1(k_{\x}) = \sum_{\y} \bigg( (\Ib_{\x} \otimes \Ib_{\x'}) \cdot \bigg(\sum_{i=1}^n (L_i \otimes R_i + R_i \otimes L_i + U_i \otimes C_i)\bigg) \cdot (\Ib_{\y} \otimes \Ib_{\y'}) \bigg) \otimes k_{\y}.
\]
Here, $R_i$, $L_i$, $U_i$, and $C_i$ stand for sums of all algebra elements represented by quiver arrows with the corresponding label.
\end{definition}

In \cite[Theorem 3.17]{OSzNew}, Ozsv{\'a}th--Szab{\'o} show that $\K$ is quasi--invertible (see \cite[Section 2.9]{OSzNew}), so that $\B^!(n,n-k)^{\op}$ is Koszul dual to $\B(n,k)$. Below we show that our $DD$ bimodule $\X^{DD}$ and our $DA$ bimodule $\X^{DA}$ for a singular crossing are related by $\X^{DD} \cong \X^{DA} \boxtimes \K$.

\section{The local \texorpdfstring{$DD$}{DD} bimodule for a singular crossing}\label{sec:LocalDD}
\subsection{Definitions}\label{sec:LocalDDDefs}

In this section we will define $\X^{DD}$, a (left, left) $DD$ bimodule over $(\B(2), \B^!(2))$. We start with names for the I-states $\x$ giving rise to the primitive idempotents $\Ib_{\x}$ of $\B(2)$; for convenience, let
\begin{align*}
&A := \{0\}, \quad B := \{1\}, \quad C := \{2\}, \\
&AB := \{0,1\}, \quad AC := \{0,2\}, \quad BC := \{1,2\}, \\
&ABC := \{0,1,2\}.
\end{align*}

The $DD$ bimodule $\X^{DD}$ respects the decomposition $\B(2) = \B(2,0) \oplus \B(2,1) \oplus \B(2,2)$. Indeed, we will define three $DD$ bimodules 
\[
{^{\B(2,0),\B^!(2,3)}}\X^{DD}, \quad {^{\B(2,1),\B^!(2,2)}}\X^{DD}, \quad {^{\B(2,2),\B^!(2,1)}}\X^{DD};
\]
we will let $\X^{DD}$ be their direct sum (we can let ${^{\B(2,3),\B^!(2,0)}}\X^{DD} = 0$).
\begin{definition}\label{def:XDD}
We will leave the gradings for Section~\ref{sec:LocalDDGradings} below. The $DD$ bimodule ${^{\B(2,0),\B^!(2,3)}}\X^{DD}$ has basis elements 
\[
\{S^-_t, S^-_b, S^+_b, S^+_t\},
\]
all of which have first idempotent $\Ib_{\varnothing}$ and second idempotent $\Ib_{ABC}$. The one nonzero term of $\delta^1$ is
\[
\delta^1(S^+_b) = (\Ib_{\varnothing} \otimes \Ib_{ABC}) \otimes S^-_b.
\]

The $DD$ bimodule ${^{\B(2,1),\B^!(2,2)}}\X^{DD}$ has basis elements 
\begin{align*}
&\{ {_{A}}(S^-_t)^{BC}, {_{C}}(S^-_t)^{AB}, {_{B}}(W_t)^{BC}, {_{A}}(S^-_b)^{BC}, {_{C}}(S^-_b)^{AB}, {_{B}}(E_b)^{AB}, \\
&{_{B}}(W_b)^{BC},  {_{A}}(S^+_b)^{BC}, {_{C}}(S^+_b)^{AB}, {_{B}}(E_t)^{AB}, {_{A}}(S^+_t)^{BC}, {_{C}}(S^+_t)^{AB} \},
\end{align*}
where the first idempotent is indicated as a subscript to the left and the second idempotent is indicated as a superscript to the right. 

Label these basis elements, in the given order, as $(1)$ through $(12)$. The $DD$ operation $\delta^1$ is defined by:
\begin{itemize}
\item 
$\delta^1((1)) = (R_1 \otimes L_1 L_2) \otimes (6) + (R_1 \otimes U_1 C_2) \otimes (7) + (U_1 \otimes U_1 C_2) \otimes (8)$
\item
$\delta^1((2)) = (L_2 \otimes R_2 R_1) \otimes (3) + (L_2 \otimes C_1 U_2) \otimes (10)$
\item
$\delta^1((3)) = (L_1 \otimes \Ib_{BC}) \otimes (1) + (\Ib_{B} \otimes L_1 L_2) \otimes (10) + (L_1 \otimes U_1 C_2) \otimes (11)$
\item
$\delta^1((4)) = 0$
\item
$\delta^1((5)) = (U_2 \otimes \Ib_{AB}) \otimes (2) + (L_2 \otimes R_2 R_1) \otimes (7) + (U_2 \otimes C_1 U_2) \otimes (12)$
\item $\delta^1((6)) = (R_2 \otimes \Ib_{AB}) \otimes (2) + (\Ib_{B} \otimes R_2 R_1) \otimes (7)$
\newline $+ (L_1 \otimes R_2 R_1) \otimes (8) + (R_2 \otimes C_1 U_2) \otimes (12)$
\item
$\delta^1((7)) = (U_2 \otimes \Ib_{BC}) \otimes (3) + (L_1 \otimes \Ib_{BC}) \otimes (4) + (R_2 \otimes L_1 L_2) \otimes (12)$
\item
$\delta^1((8)) = (\Ib_{A} \otimes \Ib_{BC}) \otimes (4)$
\item
$\delta^1((9)) = (\Ib_{C} \otimes \Ib_{AB}) \otimes (5) + (L_2 \otimes \Ib_{AB}) \otimes (6)$
\item
$\delta^1((10)) = (U_1 \otimes \Ib_{AB}) \otimes (6) + (L_1 \otimes R_2 R_1) \otimes (11)$
\item
$\delta^1((11)) = (R_1 \otimes \Ib_{BC}) \otimes (7) + (U_1 \otimes 1) \otimes (8)$
\item
$\delta^1((12)) = (L_2 \otimes \Ib_{AB}) \otimes (10)$.
\end{itemize}

The $DD$ bimodule ${^{\B(2,2),\B^!(2,1)}}\X^{DD}$ has basis elements 
\begin{align*}
&\{ {_{AC}}(S^-_t)^{B}, {_{BC}}(W_t)^{B}, {_{AC}}(S^-_b)^{B}, {_{AB}}(E_b)^{B}, \\
&{_{BC}}(W_b)^{B},  {_{AC}}(S^+_b)^{B}, {_{AB}}(E_t)^{B}, {_{AC}}(S^+_t)^{B} \},
\end{align*}
where the idempotents are indicated as above. Label these basis elements, in the given order, as $(1)$ through $(8)$. The $DD$ operation $\delta^1$ is defined by:
\begin{itemize}
\item $\delta^1((1)) = (R_1 U_2 \otimes C_1 C_2) \otimes (2) + (L_2 U_1 \otimes C_1 C_2) \otimes (4)$
\newline $+ (R_1 \otimes U_1 C_2) \otimes (5) + (U_1 \otimes U_1 C_2) \otimes (6) + (L_2 \otimes C_1 U_2) \otimes (7)$
\item $\delta^1((2)) = (L_1 \otimes \Ib_{B}) \otimes (1) + (L_1 L_2 \otimes C_1 C_2) \otimes (7) + (L_1 \otimes U_1 C_2) \otimes (8)$
\item $\delta^1((3)) = (U_2 \otimes \Ib_{B}) \otimes (1) + (R_1 U_2 \otimes C_1 C_2) \otimes (5)$
\newline $+ (U_1 U_2 \otimes C_1 C_2) \otimes (6) + (U_2 \otimes C_1 U_2) \otimes (8)$
\item $\delta^1((4)) = (R_2 \otimes \Ib_{B}) \otimes (1) + (R_2 R_1 \otimes C_1 C_2) \otimes (5)$
\newline $+ (R_2 U_1 \otimes C_1 C_2) \otimes (6) + (R_2 \otimes C_1 U_2) \otimes (8)$
\item $\delta^1((5)) = (U_2 \otimes \Ib_{B}) \otimes (2) + (L_1 \otimes \Ib_{B}) \otimes (3) + (L_1 U_2 \otimes C_1 C_2) \otimes (8)$
\item $\delta^1((6)) = (\Ib_{AC} \otimes \Ib_{B}) \otimes (3) + (L_2 \otimes \Ib_{B}) \otimes (4)$
\item $\delta^1((7)) =  (R_2 R_1 \otimes \Ib_{B}) \otimes (2) + (U_1 \otimes \Ib_{B}) \otimes (4) + (R_2 U_1 \otimes C_1 C_2) \otimes (8)$
\item $\delta^1((8)) = (R_1 \otimes \Ib_{B}) \otimes (5) + (U_1 \otimes \Ib_{B}) \otimes (6) + (L_2 \otimes 1) \otimes (7)$.
\end{itemize}
As mentioned above, the $DD$ bimodule $\X^{DD}$ is defined as
\[
\X^{DD} := ({^{\B(2,0),\B^!(2,3)}}\X^{DD}) \oplus {(^{\B(2,1),\B^!(2,2)}}\X^{DD}) \oplus ({^{\B(2,2),\B^!(2,1)}}\X^{DD}).
\]
\end{definition}

\begin{figure}
	\includegraphics[scale=0.625]{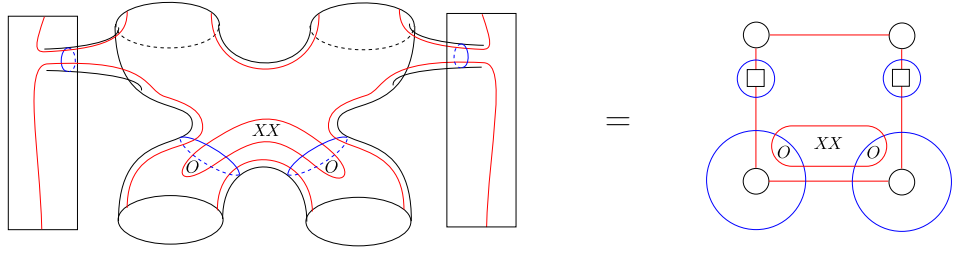}
	\caption{Stabilized and bordered-sutured version of the singular piece from Figure~\ref{fig:SingularDiag}. The diagram on the right is drawn in $S^2$.}
	\label{fig:StabilizedSingularDiag}
\end{figure}

\begin{figure}
	\includegraphics[scale=0.625]{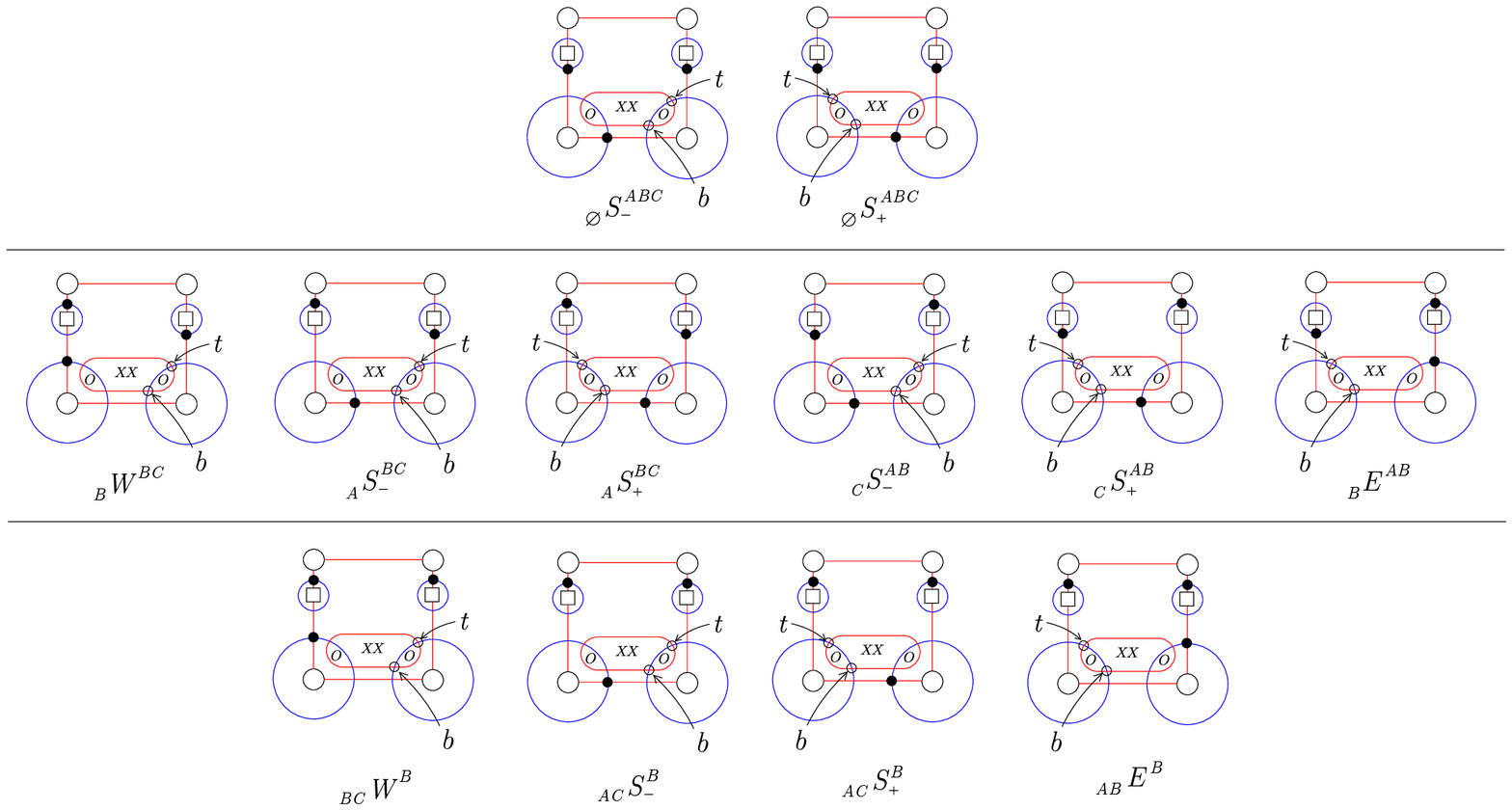}
	\caption{Top row: basis elements of ${^{\B(2,0),\B^!(2,3)}}\X^{DD}$. Middle row: basis elements of ${^{\B(2,1),\B^!(2,2)}}\X^{DD}$. Bottom row: basis elements of ${^{\B(2,2),\B^!(2,1)}}\X^{DD}$.}
	\label{fig:GensAsIntPts}
\end{figure}

\begin{figure}
	\includegraphics[scale=0.625]{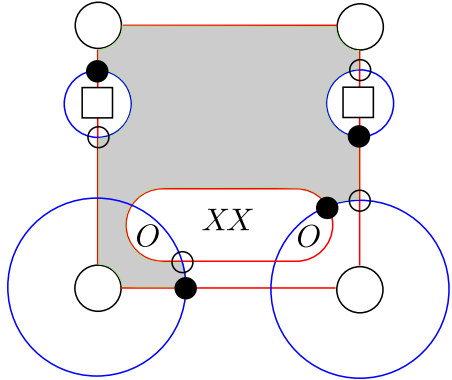}
	\caption{Domain giving rise to the term $(R_1 \otimes L_1 L_2) \otimes {_B}E_b^{AB}$ of $\delta^1({_A}(S^-_t)^{BC})$. The basis element ${_A}(S^-_t)^{BC}$ is drawn using solid dots and the basis element ${_B}E_b^{AB}$ is drawn using hollow dots.}
	\label{fig:DomainExample}
\end{figure}

\begin{remark}
The basis elements and terms of $\delta^1$ in Definition~\ref{def:XDD}, including the idempotent data, are motivated by applying the ideas of bordered sutured Floer homology \cite{Zarev} to a stabilized and bordered sutured version of the Heegaard diagram from Figure~\ref{fig:SingularDiag}. See Figure~\ref{fig:StabilizedSingularDiag} for an illustration of the diagram. The black circles and horizontal segments of squares in the diagram should be interpreted as the ``bordered'' portion of the boundary; the vertical segments of squares are the ``sutured'' portion. Note that bordered sutured Floer homology for Heegaard diagrams with closed circles in their bordered boundary has not yet been defined in generality; Ozsv{\'a}th--Szab{\'o}'s algebras and bimodules are not covered by Zarev's constructions, although they should be covered by a generalization of it. See \cite{MMW2} for more discussion on this topic.

The Heegaard diagram in question (except for the corners defining the structure as a bordered sutured Heegaard diagram) is also one of the diagrams considered in \cite{ManionDiagrams}. See \cite[Figures 14, 17]{ManionDiagrams}, as well as \cite[Figure 18]{ManionDiagrams} for an explanation of how to go between the two ways of drawing the diagram. A similar stabilized diagram motivated the idempotent structure of Ozsv{\'a}th--Szab{\'o}'s theory. 

Figure~\ref{fig:GensAsIntPts} shows the basis elements of each of the three summands of $\X^{DD}$ in terms of sets of intersection points in the diagram of Figure~\ref{fig:StabilizedSingularDiag}. The $DD$ operation $\delta^1$ on $\X^{DD}$ is motivated by counting holomorphic disks in the Heegaard diagram of Figure~\ref{fig:SingularDiag} analogously to how Ozsv{\'a}th--Szab{\'o} count disks in their local Heegaard diagrams. As in \cite{OSzNew, OSzNewer}, we do not prove results about holomorphic geometry in this paper, but one can still identify terms of $\delta^1$ with domains in the Heegaard diagram and try to apply reasonable counting rules. When doing this, note that to get the right answers, the orientation of the Heegaard surface should be the reverse of its usual orientation (so that a small circle oriented clockwise on the ``front face'' of the surface bounds a positive region). See Figure~\ref{fig:DomainExample} for an illustration of a domain representing one term of the $DD$ operation $\delta^1$ on $X^{DD}$.
\end{remark}

\begin{remark}
The labels ``$t$'' or ``$b$'' in the names of the basis elements of $\X^{DD}$ indicate that the corresponding set of intersection points in Figure~\ref{fig:GensAsIntPts} includes the top or bottom open-circled intersection point, respectively.
\end{remark}

\begin{remark}
In \cite{OSSz}, the Heegaard diagram in Figure~\ref{fig:StabilizedSingularDiag} would represent a singular crossing with two upward-pointing strands. However, an inspection of the local Alexander and Maslov gradings for nonsingular crossings in \cite{OSSz} reveals that to be compatible with the theory of \cite{OSzNew}, one must reverse the orientations on all strands in \cite{OSSz} (or some other equivalent change of conventions). Thus, in the context of the Kauffman-states functor, we take the Heegaard diagram of Figure~\ref{fig:StabilizedSingularDiag} to represent a singular crossing with two downward-pointing strands.
\end{remark}

 \begin{figure}
	\includegraphics[scale=0.7]{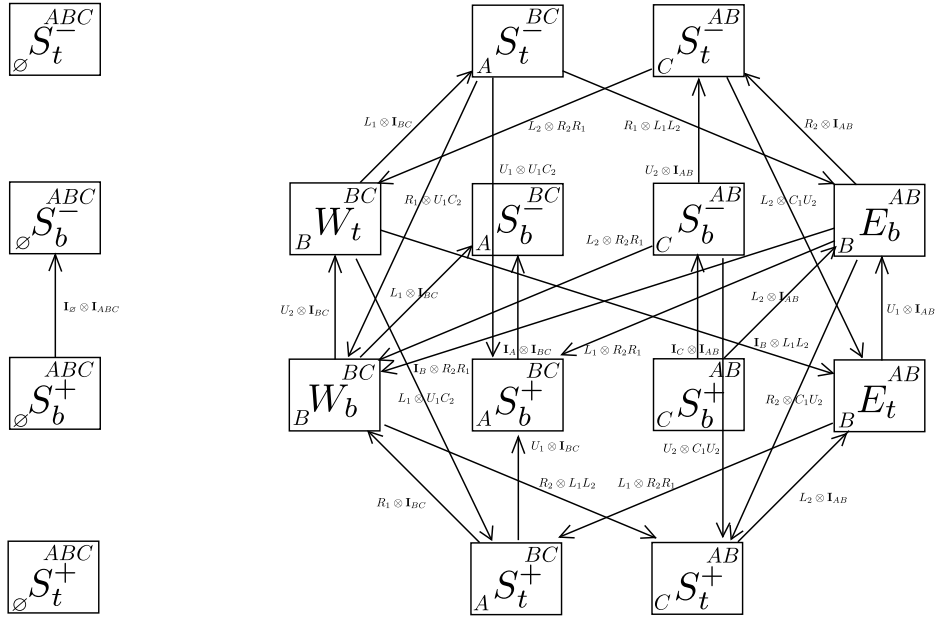}
	\caption{The first two summands of the (unsimplified) local $DD$ bimodule $\X^{DD}$ for a singular point.}
	\label{fig:UnsimplifiedLocalBimod12}
\end{figure}

\begin{figure}
	\includegraphics[scale=0.7]{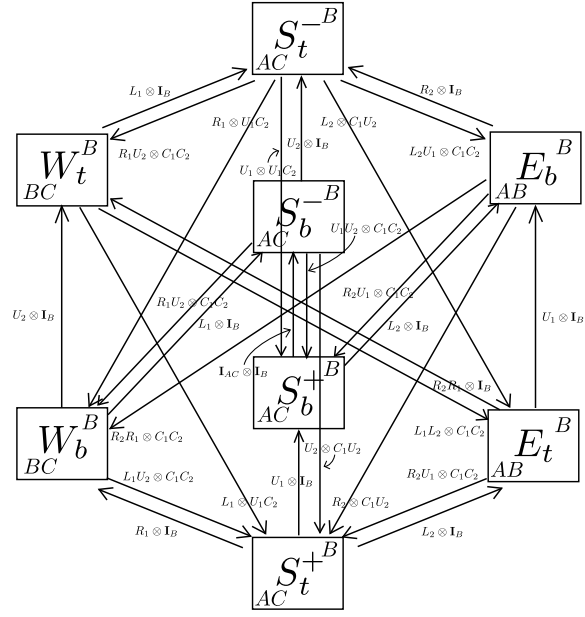}
	\caption{The third summand of the (unsimplified) local $DD$ bimodule $\X^{DD}$ for a singular point.}
	\label{fig:UnsimplifiedLocalBimod3}
\end{figure}

Figure~\ref{fig:UnsimplifiedLocalBimod12} shows the summands ${^{\B(2,0),\B^!(2,3)}}\X^{DD}$ and ${^{\B(2,1),\B^!(2,2)}}\X^{DD}$ of $\X^{DD}$ in the graphical notation of Section~\ref{sec:Graphical}. Figure~\ref{fig:UnsimplifiedLocalBimod3} shows ${^{\B(2,2),\B^!(2,1)}}\X^{DD}$; there is no summand of $\X^{DD}$ over $(\B(2,3),\B^!(2,0))$. 

The vertices in these graphs are labeled with the names of the basis elements and their first and second idempotents (lower-left and upper-right corners respectively). The numbering of the basis elements above corresponds to reading each row of these figures from right to left, and reading the rows from top to bottom. One can check compatibility of the idempotent data; we will define the grading data below. We will delay verifying the $DD$ relations; they will be deduced from the $DA$ relations in Section~\ref{sec:LocalDA}.

\subsection{Simplifying the local \texorpdfstring{$DD$}{DD} bimodule}
The edges labeled $\Ib_{\varnothing} \otimes \Ib_{ABC}$, $\Ib_{A} \otimes \Ib_{BC}$, $\Ib_{C} \otimes \Ib_{AB}$, and $\Ib_{AC} \otimes \Ib_B$ in Figure~\ref{fig:UnsimplifiedLocalBimod12} and Figure~\ref{fig:UnsimplifiedLocalBimod3} indicates the presence of cancellable pairs in $\X^{DD}$.

\begin{proposition}
$\X^{DD}$ is homotopy equivalent to the $DD$ bimodule $\Xt^{DD}$ defined graphically in Figure~\ref{fig:SimplifiedLocalBimod12} and Figure~\ref{fig:SimplifiedLocalBimod3}.
\end{proposition}

\begin{proof}
This claim follows from Section~\ref{sec:HowToSimplify}; the cancellable pairs are valid, and Figure~\ref{fig:SimplifiedLocalBimod12} and Figure~\ref{fig:SimplifiedLocalBimod3} are obtained from Figure~\ref{fig:UnsimplifiedLocalBimod12} and Figure~\ref{fig:UnsimplifiedLocalBimod3} as described in Definition~\ref{def:DDCancelling}.
\end{proof}

\begin{figure}
	\includegraphics[scale=0.7]{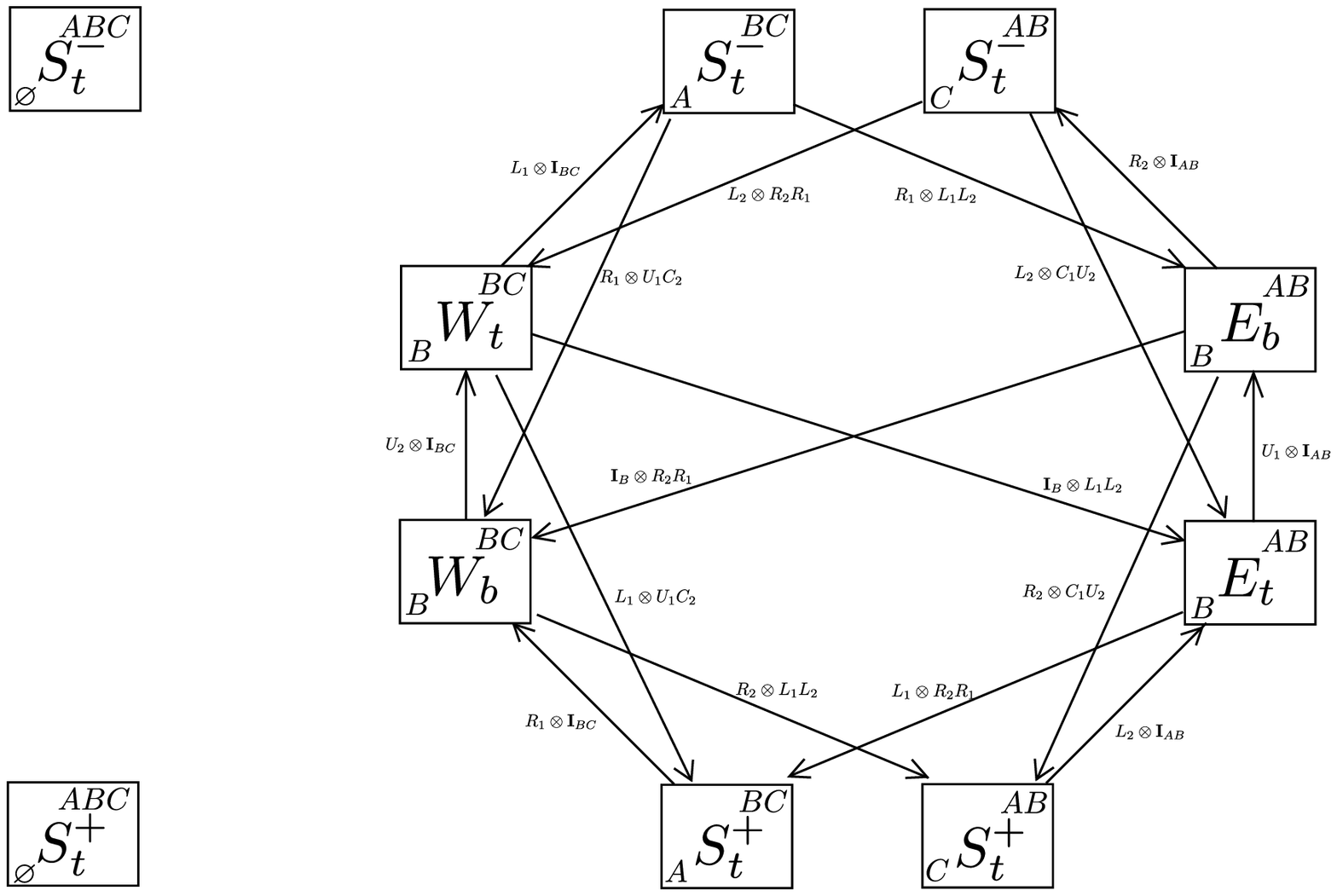}
	\caption{The first two summands of the simplified local $DD$ bimodule $\Xt^{DD}$ for a singular point.}
	\label{fig:SimplifiedLocalBimod12}
\end{figure}

\begin{figure}
	\includegraphics[scale=0.7]{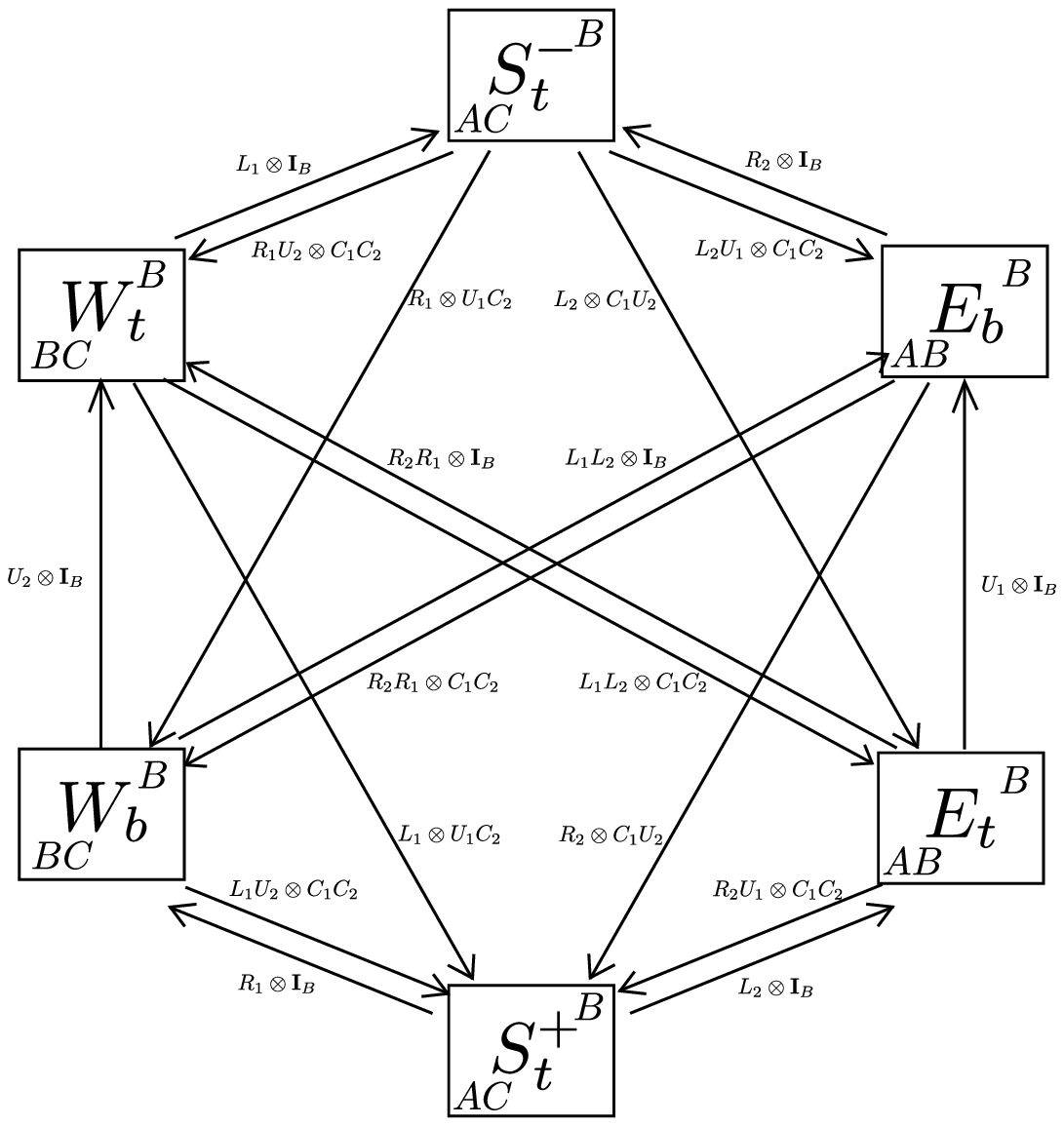}
	\caption{The third summand of the simplified local $DD$ bimodule $\Xt^{DD}$ for a singular point.}
	\label{fig:SimplifiedLocalBimod3}
\end{figure}

\subsection{Symmetries in the local \texorpdfstring{$DD$}{DD} bimodule}\label{sec:DDSymm}

The simplified $DD$ bimodule $\Xt^{DD}$ has two symmetries corresponding to the $\Rc$ and $o$ symmetries on the algebra. We define $\Rc\co \Xt^{DD} \to \Xt^{DD}$ by
\begin{itemize}
\item ${_{\varnothing}}(S^-_t)^{ABC} \leftrightarrow {_{\varnothing}}(S^-_t)^{ABC}$
\item ${_{\varnothing}}(S^+_t)^{ABC} \leftrightarrow {_{\varnothing}}(S^+_t)^{ABC}$
\item ${_{A}}(S^-_t)^{BC} \leftrightarrow {_{C}}(S^-_t)^{AB}$
\item ${_{B}}(W_t)^{BC} \leftrightarrow {_{B}}(E_b)^{AB}$
\item ${_{B}}(W_b)^{BC} \leftrightarrow {_{B}}(E_t)^{AB}$
\item ${_{A}}(S^+_t)^{BC} \leftrightarrow {_{C}}(S^+_t)^{AB}$
\item ${_{AC}}(S^-_t)^{B} \leftrightarrow {_{AC}}(S^-_t)^{B}$
\item ${_{BC}}(W_t)^{B} \leftrightarrow {_{AB}}(E_b)^{B}$
\item ${_{BC}}(W_b)^{B} \leftrightarrow {_{AB}}(E_t)^{B}$
\item ${_{AC}}(S^+_t)^{B} \leftrightarrow {_{AC}}(S^+_t)^{B}$.
\end{itemize}
The $DD$ operation $\delta^1$ on $\Xt^{DD}$ is compatible with $\Rc$ in the sense that the square
\[
\xymatrix{ \Xt^{DD} \ar[d]_{\delta^1} \ar[r]^{\Rc} & \Xt^{DD} \ar[d]^{\delta^1} \\
(\B(2) \otimes \B^!(2)) \otimes_{\I(2) \otimes \I(2)} \Xt^{DD} \ar[r]_{\Rc \otimes \Rc} & (\B(2) \otimes \B^!(2)) \otimes_{\I(2) \otimes \I(2)} \Xt^{DD}
}
\]
commutes, where $\Rc$ acting on $\B(2) \otimes \B^!(2)$ is the tensor product of $\Rc$ acting on each factor. Equivalently, the map $\Rc$ is a $DD$ bimodule isomorphism from $\Xt^{DD}$ to $\Ind_{\Rc} \Xt^{DD}$, where $\Ind_{\Rc} \Xt^{DD}$ is the $DD$ bimodule obtained from $\Xt^{DD}$ by modifying the action of $\I(2) \otimes \I(2)$ by $\Rc$ and applying $\Rc$ to the algebra outputs of $\delta^1$ (see \cite[Section 2.4.2]{LOTBimodules} for more details on induction and restriction in this context). Graphically, Figure~\ref{fig:SimplifiedLocalBimod12} and Figure~\ref{fig:SimplifiedLocalBimod3} are unchanged by reflecting about the vertical axis of the graph for each summand while applying $\Rc$ to each tensor factor of each algebra label. 

For the second symmetry $o$ on $\Xt^{DD}$, we define:
\begin{itemize}
\item ${_{\varnothing}}(S^-_t)^{ABC} \leftrightarrow {_{\varnothing}}(S^+_t)^{ABC}$
\item ${_{A}}(S^-_t)^{BC} \leftrightarrow {_{A}}(S^+_t)^{BC}$
\item ${_{C}}(S^-_t)^{AB} \leftrightarrow {_{C}}(S^+_t)^{AB}$
\item ${_{B}}(W_t)^{BC} \leftrightarrow {_{B}}(W_b)^{BC}$
\item ${_{B}}(E_b)^{AB} \leftrightarrow {_{B}}(E_t)^{AB}$
\item ${_{AC}}(S^-_t)^{B} \leftrightarrow {_{AC}}(S^+_t)^{B}$
\item ${_{BC}}(W_t)^{B} \leftrightarrow {_{BC}}(W_b)^{B}$
\item ${_{AB}}(E_b)^{B} \leftrightarrow {_{AB}}(E_t)^{B}$.
\end{itemize}
This symmetry corresponds to reflection across the horizontal axes in Figure~\ref{fig:SimplifiedLocalBimod12} and Figure~\ref{fig:SimplifiedLocalBimod3}. Let the above correspondence define $o\co \Xt^{DD} \to (\Xt^{DD})^{\vee}$ (under the natural identification of basis and dual basis elements), where $(\Xt^{DD})^{\vee}$ is the dual of $\Xt^{DD}$ as defined in Section~\ref{sec:DualsOfModules}. The $DD$ operation $\delta^1$ on $\Xt^{DD}$ is compatible with $o$ in the sense that the square
\[
\xymatrix{ \Xt^{DD} \ar[r]^{o} \ar[d]_{\delta^1} & (\Xt^{DD})^{\vee} \ar[d]^{(\delta_1)^{\vee}} \\
(\B(2) \otimes \B^!(2)) \otimes_{\I(2) \otimes \I(2)} \Xt^{DD} \ar[r]_-{o \otimes o} & (\B(2) \otimes \B^!(2))^{\op} \otimes_{\I(2) \otimes \I(2)} (\Xt^{DD})^{\vee}
}
\]
commutes, where $o$ acting on $(\B(2) \otimes \B^!(2))^{\op} \cong \B(2)^{\op} \otimes (\B^!(2))^{\op}$ is the tensor product of $o$ acting on each factor. Equivalently, the map $o$ is a $DD$ bimodule isomorphism from $\Xt^{DD}$ to $\Ind_{o} (\Xt^{DD})^{\vee}$. Graphically, Figure~\ref{fig:SimplifiedLocalBimod12} and Figure~\ref{fig:SimplifiedLocalBimod3} are unchanged by reflecting about the horizontal axis of the graph for each summand while applying $o$ to each tensor factor of each algebra label and reversing the directions of the edges.

\begin{remark}
Note that $o$ and $\Rc$ commute. Their composition $\Rc o$ corresponds to rotating Figures~\ref{fig:SimplifiedLocalBimod12} and \ref{fig:SimplifiedLocalBimod3} by $180^{\circ}$, applying $\Rc o$ to each tensor factor of each algebra label, and reversing the directions of the edges. This composite symmetry may be realized even on the unsimplified $DD$ bimodule $\X^{DD}$. Figures~\ref{fig:UnsimplifiedLocalBimod12} and \ref{fig:UnsimplifiedLocalBimod3} are symmetric under $180^{\circ}$ rotation (changing the edges as specified), although not under reflection across either vertical or horizontal axes.
\end{remark}

\subsection{Gradings}\label{sec:LocalDDGradings}

\begin{definition}\label{def:HomGrLocalDD}
The homological (or Maslov) grading on $\X^{DD}$ is defined as follows:
\begin{itemize}
\item $\m(S^-_t) = 1$
\item $\m(W_t) = 0$
\item $\m(S^-_b) = 0$
\item $\m(E_b) = 0$
\item $\m(W_b) = -1$
\item $\m(S^+_b) = -1$
\item $\m(E_t) = -1$
\item $\m(S^+_t) = -2$
\end{itemize}
\end{definition}

These homological degrees do not depend on the idempotents of the basis elements of $\X^{DD}$. The rows of Figures \ref{fig:UnsimplifiedLocalBimod12}, \ref{fig:SimplifiedLocalBimod12}, and \ref{fig:SimplifiedLocalBimod3} from top to bottom correspond to homological degrees $1$ through $-2$ respectively. The same is approximately true in Figure~\ref{fig:UnsimplifiedLocalBimod3}, although the middle two basis elements are shifted a bit for spacing reasons.

We will define two versions of the unrefined grading on $\X^{DD}$. The first version is a grading by $\Z^{4}$ with basis $\tau_1, \tau_2, \beta_1, \beta_2$.
\begin{definition}\label{def:Ref1LocalDD}
The first unrefined grading on $\X^{DD}$ is defined as follows:
\begin{itemize}
\item $\deg^{\un}_1(S^-_t) = 0$
\item $\deg^{\un}_1(W_t) = \beta_1$
\item $\deg^{\un}_1(S^-_b) = \tau_2 + \beta_2$
\item $\deg^{\un}_1(E_b) = \tau_2$
\item $\deg^{\un}_1(W_b) = \tau_2 + \beta_1 + \beta_2$
\item $\deg^{\un}_1(S^+_b) = \tau_2 + \beta_2$
\item $\deg^{\un}_1(E_t) = \tau_1 + \tau_2 + \beta_1$
\item $\deg^{\un}_1(S^+_t) = \tau_1 + \tau_2 + \beta_1 + \beta_2$.
\end{itemize}
\end{definition}

\begin{proposition}\label{prop:FirstTwoGradingFacts}
The graphs in Figure~\ref{fig:UnsimplifiedLocalBimod12} and Figure~\ref{fig:UnsimplifiedLocalBimod3} are compatible with the homological and first unrefined gradings. If we let the $DD$ bimodule symmetry $\Rc$ from Section~\ref{sec:DDSymm} act trivially on the homological grading group $\Z$, then for each basis element $x$ of $\Xt^{DD}$, we have $\m(\Rc(x)) = \m(x)$ and $\deg^{\un}_1(\Rc(x)) = \Rc(\deg^{\un}_1(x))$.

For the second symmetry $o$, we have $\m(o(x)) = -\m(x) - 1$ for any basis element $x$ of $\Xt^{DD}$. We have $\deg^{\un}_1(o(x)) = \tau_1 + \tau_2 + \beta_1 + \beta_2 - o(\deg(x))$.
\end{proposition}

\begin{proof}
The first claim follows from inspection of Figure~\ref{fig:UnsimplifiedLocalBimod12} and Figure~\ref{fig:UnsimplifiedLocalBimod3}. The second claim follows from inspection of Figure~\ref{fig:SimplifiedLocalBimod12} and Figure~\ref{fig:SimplifiedLocalBimod3}.
\end{proof}

\begin{definition}\label{def:Ref2LocalDD}
The second unrefined grading $\deg^{\un}_2$ on $\X^{DD}$ is a grading by $(\frac{1}{4}\Z)^{2}$ obtained by subtracting $3\frac{\tau_1 + \tau_2 + \beta_1 + \beta_2}{4}$ from the first unrefined degrees.
\end{definition}

\begin{proposition}
If $x$ is a basis element of $\Xt^{DD}$, we have $\deg^{\un}_2(\Rc(x)) = \Rc(\deg^{\un}_2(x))$ and $\deg^{\un}_2(o(x)) = -\frac{\tau_1 + \tau_2 + \beta_1 + \beta_2}{2} - o(\deg^{\un}_2(x))$.
\end{proposition}
The second unrefined grading has the additional advantage that it reduces to Ozsv{\'a}th--Stipsicz--Szab{\'o}'s Alexander grading from \cite{OSSz} (after multiplying homological degrees by $-1$); it should also be more natural from the perspective of categorification. It has the disadvantage that one must work over $(\frac{1}{4}\Z)^{4}$ rather than $\Z^4$.

For convenience, we list the second unrefined degrees of basis elements of $\X^{DD}$. We also list the refined Alexander multi-degrees obtained from the second unrefined degrees by the map $\eta\co \Z^4 \to (\frac{1}{2}\Z)^2$ sending $\tau_1, \beta_1 \mapsto \frac{e_1}{2}$ and $\tau_2, \beta_2 \mapsto \frac{e_2}{2}$ (see Section~\ref{sec:AlgGrading}), and the single Alexander degrees obtained from the refined Alexander multi-degrees from the sum map from $\Z^2$ to $\Z$.
\begin{proposition}
The second unrefined grading on $\X^{DD}$ is:
\begin{itemize}
\item $\deg^{\un}_2(S^-_t) = -3\frac{\tau_1 + \tau_2 + \beta_1 + \beta_2}{4}$
\item $\deg^{\un}_2(W_t) = \frac{-3\tau_1 -3\tau_2 + \beta_1 - 3\beta_2}{4}$
\item $\deg^{\un}_2(S^-_b) = \frac{-3\tau_1 + \tau_2 - 3\beta_1 + \beta_2}{4}$
\item $\deg^{\un}_2(E_b) = \frac{-3\tau_1 + \tau_2 - 3\beta_1 - 3\beta_2}{4}$
\item $\deg^{\un}_2(W_b) = \frac{-3\tau_1 + \tau_2 + \beta_1 + \beta_2}{4}$
\item $\deg^{\un}_2(S^+_b) = \frac{-3\tau_1 + \tau_2 - 3\beta_1 + \beta_2}{4}$
\item $\deg^{\un}_2(E_t) = \frac{\tau_1 + \tau_2 + \beta_1 - 3\beta_2}{4}$
\item $\deg^{\un}_2(S^+_t) = \frac{\tau_1 + \tau_2 + \beta_1 + \beta_2}{4}$.
\end{itemize}
\end{proposition}

\begin{definition}\label{def:BimoduleRefinedAndSingleGr}
The refined grading on $\X^{DD}$ is:
\begin{itemize}
\item $\deg^{\refi}(S^-_t) = -3\frac{e_1 + e_2}{4}$
\item $\deg^{\refi}(W_t) = \frac{-e_1 - 3e_2}{4}$
\item $\deg^{\refi}(S^-_b) = \frac{-3e_1 + e_2}{4}$
\item $\deg^{\refi}(E_b) = \frac{-3e_1-e_2}{4}$
\item $\deg^{\refi}(W_b) = \frac{-e_1 +e_2}{4}$
\item $\deg^{\refi}(S^+_b) = \frac{-3e_1 + e_2}{4}$
\item $\deg^{\refi}(E_t) = \frac{e_1 -e_2}{4}$
\item $\deg^{\refi}(S^+_t) = \frac{e_1 + e_2}{4}$.
\end{itemize}
The single Alexander grading on $\X^{DD}$ is
\begin{itemize}
\item $\Alex(S^-_t) = -\frac{3}{2}$
\item $\Alex(W_t) = -1$
\item $\Alex(S^-_b) = -\frac{1}{2}$
\item $\Alex(E_b) = -1$
\item $\Alex(W_b) = 0$
\item $\Alex(S^+_b) = -\frac{1}{2}$
\item $\Alex(E_t) = 0$
\item $\Alex(S^+_t) = \frac{1}{2}$.
\end{itemize}
\end{definition}

\begin{proposition}
Under the correspondence shown in Figure~\ref{fig:GensAsIntPts} (and after multiplying the homological degrees by $-1$), the homological degrees and single Alexander degrees of the basis elements of $\X^{DD}$ agree with the local degrees of the corresponding types of basis elements listed in \cite[Figures 8 and 9]{OSSz}.
\end{proposition}

\begin{proof}
Our basis elements of type $W$, $E$, $S_+$, and $S_-$ correspond to the left corner, right corner, bottom corner labeled $D_+$, and bottom corner labeled $D_-$ respectively in \cite[Figures 8 and 9]{OSSz}. Ozsv{\'a}th--Stipsicz--Szab{\'o} show only the highest Alexander degree among the two types of local basis elements that we call $\{t,b\}$. 

A comparison of \cite[Figure 11]{OSSz} with Figure~\ref{fig:GensAsIntPts} and Definition~\ref{def:BimoduleRefinedAndSingleGr} shows that our highest-degree basis elements agree with Ozsv{\'a}th--Stipsicz--Szab{\'o}'s. By Definition~\ref{def:BimoduleRefinedAndSingleGr} and \cite[Figure 9]{OSSz}, the single Alexander degrees of corresponding highest-degree basis elements are the same; by Definition~\ref{def:HomGrLocalDD} and \cite[Figure 8]{OSSz}, the homological degrees of corresponding basis elements are the same after multiplying by $-1$.

For any other basis element $x$, there is a basis element $x_0$ with ``highest degree'' as above, and such that $\Alex(x) = \Alex(x_0) - 1$ and $\m(x) = \m(x_0) + 1$ as defined here. In the Heegaard diagram of Figure~\ref{fig:StabilizedSingularDiag}, there is a bigon from $x$ to $x_0$ passing through a basepoint labeled $O$ (or $z^j$, $j \in \{1,2\}$ in Ozsv{\'a}th--Stipsicz--Szab{\'o}'s notation) and no other basepoints. As described in \cite[p. 386]{OSSz}, this bigon implies that Ozsv{\'a}th--Stipsicz--Szab{\'o}'s Maslov degree also satisfies $\m(x) = \m(x_0) + 1$ (after the usual multiplication by $-1$). Similarly, by \cite[equation 3]{OSSz}, the bigon implies that Ozsv{\'a}th--Stipsicz--Szab{\'o}'s Alexander degree satisfies $\Alex(x) = \Alex(x_0) - 1$.
\end{proof}

\section{The local \texorpdfstring{$DA$}{DA} bimodule for a singular crossing}\label{sec:LocalDA}

Now we will describe a $DA$ bimodule $\X^{DA}$ over $(\B(2),\B(2))$ representing a local singular crossing. We will have $\X^{DD} \cong \X^{DA} \boxtimes \K$, where $\X^{DD}$ is the local $DD$ bimodule for a singular crossing from Section~\ref{sec:LocalDD} and $\K$ is the canonical $DD$ bimodule over $(\B(2),\B^!(2))$ from Definition~\ref{def:KoszulDualizing}. We will use the same names for I-states and their corresponding idempotents as in Section~\ref{sec:LocalDD}. Like $\X^{DD}$, $\X^{DA}$ will have three summands. 

\begin{definition}\label{def:XDA}
The $DA$ bimodule ${^{\B(2,0)}}(\X^{DA})_{\B(2,0)}$ has basis elements 
\[
\{S^-_t, S^-_b, S^+_b, S^+_t\},
\]
all of which have left idempotent $\Ib_{\varnothing}$ and right idempotent $\Ib_{\varnothing}$. The one nonzero term of $\delta^1_1$ is
\[
\delta^1(S^+_b) = \Ib_{\varnothing} \otimes S^-_b.
\]
There are no nonzero terms of $\delta^1_i$ for $i > 1$.

The $DA$ bimodule ${^{\B(2,1)}}(\X^{DA})_{\B(2,1)}$ has basis elements 
\begin{align*}
&\{ {_{A}}(S^-_t)^{A}, {_{C}}(S^-_t)^{C}, {_{B}}(W_t)^{A}, {_{A}}(S^-_b)^{A}, {_{C}}(S^-_b)^{C}, {_{B}}(E_b)^{C}, \\
&{_{B}}(W_b)^{A},  {_{A}}(S^+_b)^{A}, {_{C}}(S^+_b)^{C}, {_{B}}(E_t)^{C}, {_{A}}(S^+_t)^{A}, {_{C}}(S^+_t)^{C} \},
\end{align*}
where the left idempotent is indicated as a subscript to the left and the right idempotent is indicated as a superscript to the right. 

Label these basis elements, in the given order, as $(1)$ through $(12)$. The $DA$ operation $\delta^1_1$ is defined by:
\begin{itemize}
\item $\delta^1_1((1)) = 0$
\item $\delta^1_1((2)) = 0$
\item $\delta^1_1((3)) = L_1 \otimes (1)$
\item $\delta^1_1((4)) = 0$
\item $\delta^1_1((5)) = U_2 \otimes (2)$
\item $\delta^1_1((6)) = R_2 \otimes (2)$
\item $\delta^1_1((7)) = U_2 \otimes (3) + L_1 \otimes (4)$
\item $\delta^1_1((8)) = \Ib_{A} \otimes (4)$
\item $\delta^1_1((9)) = \Ib_{C} \otimes (5) + L_2 \otimes (6)$
\item $\delta^1_1((10)) = U_1 \otimes (6)$
\item $\delta^1_1((11)) = R_1 \otimes (7) + U_1 \otimes (8)$
\item $\delta^1_1((12)) = L_2 \otimes (10)$.
\end{itemize}
The $DA$ operation $\delta^1_2$ is only $\delta^1_2(x,\Ib') = \Ib \otimes x$ where $\Ib,\Ib'$ are the left and right idempotents of each basis element $x$. The operation $\delta^1_3$ has the following nonzero terms:
\begin{itemize}
\item $\delta^1_3((1)\otimes R_1 \otimes R_2) = R_1 \otimes (6)$
\item $\delta^1_3((2)\otimes L_2 \otimes L_1) = L_2 \otimes (3)$
\item $\delta^1_3((3)\otimes R_1 \otimes R_2) = \Ib_{B} \otimes (10)$
\item $\delta^1_3((5)\otimes L_2 \otimes L_1) = L_2 \otimes (7)$
\item $\delta^1_3((6)\otimes L_2 \otimes L_1) = \Ib_{B} \otimes (7) + L_1 \otimes (8)$
\item $\delta^1_3((7)\otimes R_1 \otimes R_2) = R_2 \otimes (12)$
\item $\delta^1_3((10)\otimes L_2 \otimes L_1) = L_1 \otimes (11)$.
\end{itemize}
The operation $\delta^1_4$ has the following nonzero terms:
\begin{itemize}
\item $\delta^1_4((1) \otimes R_1 \otimes U_2 \otimes L_1) = R_1 \otimes (7) + U_1 \otimes (8)$
\item $\delta^1_4((2) \otimes L_2 \otimes U_1 \otimes R_2) = L_2 \otimes (10)$
\item $\delta^1_4((3) \otimes R_1 \otimes U_2 \otimes L_1) = L_1 \otimes (11)$
\item $\delta^1_4((5) \otimes L_2 \otimes U_1 \otimes R_2) = U_2 \otimes (12)$
\item $\delta^1_4((6) \otimes L_2 \otimes U_1 \otimes R_2) = R_2 \otimes (12)$.
\end{itemize}

The $DA$ bimodule ${^{\B(2,2)}}(\X^{DA})_{\B(2,2)}$ has basis elements 
\begin{align*}
&\{ {_{AC}}(S^-_t)^{AC}, {_{BC}}(W_t)^{AC}, {_{AC}}(S^-_b)^{AC}, {_{AB}}(E_b)^{AC}, \\
&{_{BC}}(W_b)^{AC},  {_{AC}}(S^+_b)^{AC}, {_{AB}}(E_t)^{AC}, {_{AC}}(S^+_t)^{AC} \}.
\end{align*}
Label these basis elements, in the given order, as $(1)$ through $(8)$. The $DA$ operation $\delta^1_1$ is defined by:
\begin{itemize}
\item $\delta^1_1((1)) = 0$
\item $\delta^1_1((2)) = L_1 \otimes (1)$
\item $\delta^1_1((3)) = U_2 \otimes (1)$
\item $\delta^1_1((4)) = R_2 \otimes (1)$
\item $\delta^1_1((5)) = U_2 \otimes (2) + L_1 \otimes (3)$
\item $\delta^1_1((6)) = \Ib_{AC} \otimes (3) + L_2 \otimes (4)$
\item $\delta^1_1((7)) = R_2 R_1 \otimes (2) + U_1 \otimes (4)$
\item $\delta^1_1((8)) = R_1 \otimes (5) + U_1 \otimes (6) + L_2 \otimes (7)$.
\end{itemize}
The $DA$ operation $\delta^1_2$ is only $\delta^1_2(x,\Ib') = \Ib \otimes x$ where $\Ib,\Ib'$ are the left and right idempotents of each basis element $x$. The operation $\delta^1_3$ has several nonzero terms; we group them according to their starting vertex and their labels in Figure~\ref{fig:UnsimplifiedDABimod3} below. 

The terms corresponding to the label $(T2)$ or $(T5)$, and starting at basis element $(1)$, are
\begin{itemize}
\item $\delta^1_3((1)\otimes U_1 \otimes U_2) = R_1 U_2 \otimes (2) + L_2 U_1 \otimes (4)$
\item $\delta^1_3((1)\otimes L_2 U_1 \otimes R_2) = R_1 U_2 \otimes (2) + L_2 U_1 \otimes (4)$
\item $\delta^1_3((1)\otimes R_1 \otimes L_1 U_2) = R_1 U_2 \otimes (2) + L_2 U_1 \otimes (4)$
\end{itemize}
(those whose output vertex is $(2)$ have label $(T2)$, and those whose output vertex is $(4)$ have label $(T5)$). 

The terms corresponding to the label $(T1)$ (starting at basis element $(2)$) are
\begin{itemize}
\item $\delta^1_3((2)\otimes U_1 \otimes U_2) = L_1 L_2 \otimes (7)$
\item $\delta^1_3((2)\otimes L_2 U_1 \otimes R_2) = L_1 L_2 \otimes (7)$
\item $\delta^1_3((2)\otimes R_1 \otimes L_1 U_2) = L_1 L_2 \otimes (7)$.
\end{itemize}

The terms corresponding to the labels $(T2)$ or $(T0)$, and starting at basis element $(3)$, are
\begin{itemize}
\item $\delta^1_3((3)\otimes U_1 \otimes U_2) = R_1 U_2 \otimes (5) + U_1 U_2 \otimes (6)$
\item $\delta^1_3((3)\otimes L_2 U_1 \otimes R_2) = R_1 U_2 \otimes (5) + U_1 U_2 \otimes (6)$
\item $\delta^1_3((3)\otimes R_1 \otimes L_1 U_2) = R_1 U_2 \otimes (5) + U_1 U_2 \otimes (6)$
\end{itemize}
(those whose output vertex is $(5)$ have label $(T2)$, and those whose output vertex is $(6)$ have label $(T0)$). 

The terms corresponding to the labels $(T3)$ and $(T6)$, and starting at basis element $(4)$, are 
\begin{itemize}
\item $\delta^1_3((4)\otimes U_1 \otimes U_2) = R_2 R_1 \otimes (5) + R_2 U_1 \otimes (6)$
\item $\delta^1_3((4)\otimes L_2 U_1 \otimes R_2) = R_2 R_1 \otimes (5) + R_2 U_1 \otimes (6)$
\item $\delta^1_3((4)\otimes R_1 \otimes L_1 U_2) = R_2 R_1 \otimes (5) + R_2 U_1 \otimes (6)$
\end{itemize}
(those whose output vertex is $(5)$ have label $(T6)$, and those whose output vertex is $(6)$ have label $(T3)$). 

The terms corresponding to the label $(T4)$ (starting at basis element $(5)$) are
\begin{itemize}
\item $\delta^1_3((5)\otimes U_1 \otimes U_2) = L_1 U_2 \otimes (8)$
\item $\delta^1_3((5)\otimes L_2 U_1 \otimes R_2) = L_1 U_2 \otimes (8)$
\item $\delta^1_3((5)\otimes R_1 \otimes L_1 U_2) = L_1 U_2 \otimes (8)$.
\end{itemize}

The terms corresponding to the label $(T3)$ and starting at basis element $(7)$ are
\begin{itemize}
\item $\delta^1_3((7)\otimes U_1 \otimes U_2) = R_2 U_1 \otimes (8)$
\item $\delta^1_3((7)\otimes L_2 U_1 \otimes R_2) = R_2 U_1 \otimes (8)$
\item $\delta^1_3((7)\otimes R_1 \otimes L_1 U_2) = R_2 U_1 \otimes (8)$.
\end{itemize}

The operation $\delta^1_4$ has the following nonzero terms:
\begin{itemize}
\item $\delta^1_4((1) \otimes R_1 \otimes U_2 \otimes L_1) = R_1 \otimes (5) + U_1 \otimes 6$
\item $\delta^1_4((1) \otimes L_2 \otimes U_1 \otimes R_2) = L_2 \otimes (7)$
\item $\delta^1_4((2) \otimes R_1 \otimes U_2 \otimes L_1) = L_1 \otimes (8)$
\item $\delta^1_4((3) \otimes L_2 \otimes U_1 \otimes R_2) = U_2 \otimes (8)$
\item $\delta^1_4((4) \otimes L_2 \otimes U_1 \otimes R_2) = R_2 \otimes (8)$.
\end{itemize}

Finally, the operation $\delta^1_5$ has the following nonzero terms, corresponding to the label $(T7)$ in Figure~\ref{fig:UnsimplifiedDABimod3}.
\begin{itemize}
\item $\delta^1_5((1) \otimes U_1 \otimes U_2 \otimes U_1 \otimes U_2) = U_1 U_2 \otimes (8)$
\item $\delta^1_5((1) \otimes L_2 U_1 \otimes R_2 \otimes U_1 \otimes U_2) = U_1 U_2 \otimes (8)$
\item $\delta^1_5((1) \otimes R_1 \otimes L_1 U_2 \otimes U_1 \otimes U_2) = U_1 U_2 \otimes (8)$
\item $\delta^1_5((1) \otimes U_1 \otimes U_2 \otimes L_2 U_1 \otimes R_2) = U_1 U_2 \otimes (8)$
\item $\delta^1_5((1) \otimes L_2 U_1 \otimes R_2 \otimes L_2 U_1 \otimes R_2) = U_1 U_2 \otimes (8)$
\item $\delta^1_5((1) \otimes R_1 \otimes L_1 U_2 \otimes L_2 U_1 \otimes R_2) = U_1 U_2 \otimes (8)$
\item $\delta^1_5((1) \otimes U_1 \otimes U_2 \otimes R_1 \otimes L_1 U_2) = U_1 U_2 \otimes (8)$
\item $\delta^1_5((1) \otimes L_2 U_1 \otimes R_2 \otimes R_1 \otimes L_1 U_2) = U_1 U_2 \otimes (8)$
\item $\delta^1_5((1) \otimes R_1 \otimes L_1 U_2 \otimes R_1 \otimes L_1 U_2) = U_1 U_2 \otimes (8)$.
\end{itemize}

The $DA$ bimodule $\X^{DA}$ is defined as
\[
\X^{DA} := ({^{\B(2,0)}}(\X^{DA})_{\B(2,0)}) \oplus ({^{\B(2,1)}}(\X^{DA})_{\B(2,1)}) \oplus ({^{\B(2,2)}}(\X^{DA})_{\B(2,2)}).
\]
\end{definition}

Figure~\ref{fig:UnsimplifiedDABimod12} shows the summands ${^{\B(2,0)}}(\X^{DA})_{\B(2,0)}$ and ${^{\B(2,1)}}(\X^{DA})_{\B(2,1)}$ of $\X^{DA}$ in the graphical notation of Section~\ref{sec:Graphical}. Figure~\ref{fig:UnsimplifiedDABimod3} shows ${^{\B(2,2)}}(\X^{DA})_{\B(2,2)}$. For convenience, $\delta^1_1$ edges are blue, $\delta^1_3$ edges are green, $\delta^1_4$ edges are red, and $\delta^1_5$ edges are teal. Edges labeled by sums do not have their labels written out fully; the labels are listed above in Definition~\ref{def:XDA}.

\begin{figure}
	\includegraphics[scale=0.7]{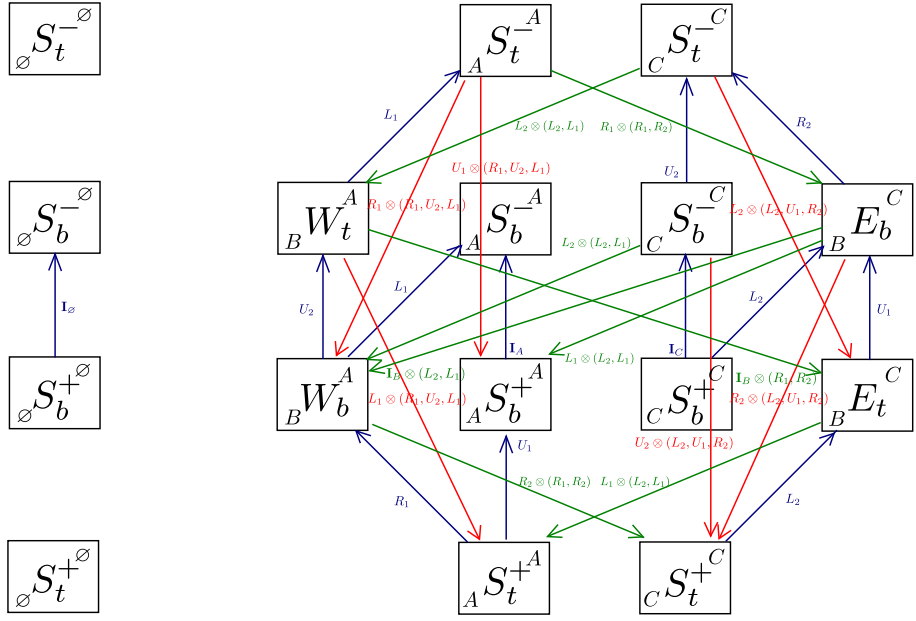}
	\caption{First two summands of the unsimplified local DA bimodule for a singular crossing.}
	\label{fig:UnsimplifiedDABimod12}
\end{figure}

\begin{figure}
	\includegraphics[scale=0.7]{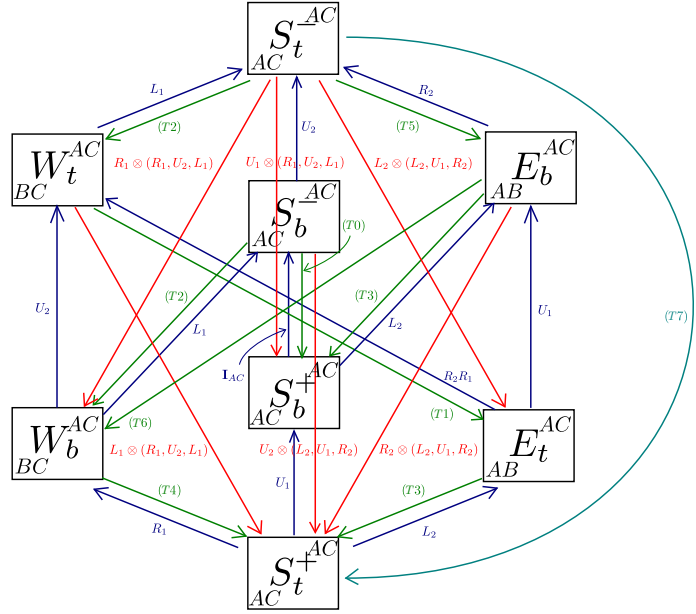}
	\caption{Third summand of the unsimplified local DA bimodule for a singular crossing.}
	\label{fig:UnsimplifiedDABimod3}
\end{figure}

Note that there is a natural one-to-one correspondence between the basis elements of $\X^{DA}$ and of $\X^{DD}$.
\begin{definition}
The homological degree, both refined degrees, and the multiple and single Alexander degrees of a basis element of $\X^{DA}$ are defined to be the degrees of the corresponding basis element of $\X^{DD}$.
\end{definition}

\subsection{Checking the \texorpdfstring{$DA$}{DA} bimodule relations}

\begin{theorem}
$\X^{DA}$ is a valid $DA$ bimodule over $(\B(2),\B(2))$.
\end{theorem}

\begin{proof}
One can check that the idempotent and grading data of the graphs in Figure~\ref{fig:UnsimplifiedDABimod12} and Figure~\ref{fig:UnsimplifiedDABimod3} are compatible. We need to verify the remaining properties discussed in Section~\ref{sec:Graphical}. We use the terminology of that section: we have sets of edges (or composable pairs) $S_1$, $S_2$, $S_3$, and $S_4$, and we need to check that certain sums are zero.

In fact, since $\B(2)$ has no differential, the set $S_2$ is empty, and $S_4$ does not contribute to the sum. For each choice of an input basis element $x$ and output basis element $y$ of $\X^{DA}$, and sequence of algebra inputs $(a'_1, \ldots, a'_{i-1})$, we want to check that the sum over $S_1$ of the product of the edge labels, plus the sum over $S_3$ of the edge labels, is zero. We may verify this condition for each summand of $\X^{DA}$ individually. 

For the summand of $\X^{DA}$ over $(\B(2,0),\B(2,0))$, the $i = 1$ sums are zero because there are no composable pairs of $\delta^1_1$ arrows. The $i=2$ sums are zero because the only $\delta^1_2$ arrows are identity arrows (labeled $\Ib \otimes \Ib'$ and omitted from the diagram as in Warning~\ref{warn:IdentityEdges}). The $i=3$ sum also zero, and only identity arrows are involved. In general, we can ignore the identity arrows below by Remark~\ref{rem:IgnoreIdentityEdges}.

For the summand of $\X^{DA}$ over $(\B(2,1),\B(2,1))$, label the basis elements as $(1)$--$(12)$ as above. The $i=1$ sums come only from $S_1$; each term is a product of labels on a composable pair of $\delta^1_1$ arrows (blue arrows on the right diagram in Figure~\ref{fig:UnsimplifiedDABimod12}). We have the following nonzero terms, written as labeled arrows and ordered by the middle vertex of the composable pair:
\begin{alignat*}{2}
&(9) \xrightarrow{U_2} (2) \quad && (9) \xrightarrow{U_2} (2) \\
&(11) \xrightarrow{U_1} (4) \quad && (11) \xrightarrow{U_1} (4). \\
\end{alignat*}
Terms which are zero due to relations in the algebra are omitted. The above terms sum to zero (recall that we are working over $\F_2$). The $i=2$ sums also come only from $S_1$, and they all involve identity arrows.

The $i=3$ sums come only from $S_1$, since there are no non-identity $\delta^1_2$ arrows. Each term is a product of output labels on a composable pair of $\delta^1_1$ (blue) arrows and $\delta^1_3$ (green) arrows (in either order) in Figure~\ref{fig:UnsimplifiedDABimod12}. The nonzero terms coming from a green arrow followed by a blue arrow are:
\begin{alignat*}{2}
&(3) \xrightarrow{U_1 \otimes (R_1, R_2)} (6) \quad && (5) \xrightarrow{L_2 U_2 \otimes (L_2, L_1)} (3) \\
&(6) \xrightarrow{U_2 \otimes (L_2, L_1)} (3) \quad && (6) \xrightarrow{L_1 \otimes (L_2, L_1)} (4) \\
& (6) \xrightarrow{L_1 \otimes (L_2, L_1)} (4) \quad && (7) \xrightarrow{U_2 \otimes (R_1, R_2)} (10) \\ 
& (10) \xrightarrow{U_1 \otimes (L_2, L_1)} (7) \quad && (10) \xrightarrow{L_1 U_1 \otimes (L_2, L_1)} (8).
\end{alignat*}
The nonzero terms coming from a blue arrow followed by a green arrow are:
\begin{alignat*}{2}
&(3) \xrightarrow{U_1 \otimes (R_1, R_2)} (6) \quad && (5) \xrightarrow{L_2 U_2 \otimes (L_2, L_1)} (3) \\
&(6) \xrightarrow{U_2 \otimes (L_2, L_1)} (3) \quad && (7) \xrightarrow{U_2 \otimes (R_1, R_2)} (10) \\
& (9) \xrightarrow{L_2 \otimes (L_2, L_1)} (7) \quad && (9) \xrightarrow{L_2 \otimes (L_2, L_1)} (7) \\
& (10) \xrightarrow{U_1 \otimes (L_2, L_1)} (7) \quad && (10) \xrightarrow{L_1 U_1 \otimes (L_2, L_1)} (8). \\
\end{alignat*}
These terms sum to zero, so the $i=3$ sums are zero.

The $i=4$ sums (ignoring identity arrows) still come only from $S_1$; there are no contributions from $S_3$ because no $\delta^1_3$ (green) arrow has an algebra input that can be factored nontrivially. Each term is a product of output labels on a composable pair of $\delta^1_1$ (blue) arrows and $\delta^1_4$ (red) arrows (in either order) in Figure~\ref{fig:UnsimplifiedDABimod12}. The nonzero terms coming from a blue arrow followed by a red arrow are:
\begin{alignat*}{2}
&(3) \xrightarrow{U_1 \otimes (R_1, U_2, L_1)} (7) \quad && (3) \xrightarrow{L_1 U_1 \otimes (R_1, U_2, L_1)} (8) \\
&(5) \xrightarrow{L_2 U_2 \otimes (L_2, U_1, R_2)} (10) \quad && (6) \xrightarrow{U_2 \otimes (L_2, U_1, R_2)} (10) \\
&(9) \xrightarrow{U_2 \otimes (L_2, U_1, R_2)} (12) \quad && (9) \xrightarrow{U_2 \otimes (L_2, U_1, R_2)} (12) \\
\end{alignat*}
The nonzero terms coming from a red arrow followed by a blue arrow are:
\begin{alignat*}{2}
&(1) \xrightarrow{U_1 \otimes (R_1, U_2, L_1)} (4) \quad && (1) \xrightarrow{U_1 \otimes (R_1, U_2, L_1)} (4) \\
&(3) \xrightarrow{U_1 \otimes (R_1, U_2, L_1)} (7) \quad && (3) \xrightarrow{L_1 U_1 \otimes (R_1, U_2, L_1)} (8) \\
&(5) \xrightarrow{L_2 U_2 \otimes (L_2, U_1, R_2)} (10) \quad && (6) \xrightarrow{U_2 \otimes (L_2, U_1, R_2)} (10). \\
\end{alignat*}
These terms sum to zero, so the $i=4$ sums are zero.

The $i = 5$ sums (ignoring identity arrows) come from both $S_1$ and $S_3$. The $S_1$ terms are products of output labels on composable pairs of two $\delta^1_3$ (green) arrows in Figure~\ref{fig:UnsimplifiedDABimod12}. The $S_3$ terms are output labels on $\delta^1_4$ (red) arrows, one of whose input algebra elements has been factored nontrivially. We have the following nonzero terms from $S_1$:
\begin{alignat*}{2}
&(1) \xrightarrow{R_1 \otimes (R_1, R_2, L_2, L_1)} (7) \quad && (1) \xrightarrow{U_1 \otimes (R_1, R_2, L_2, L_1)} (8) \\
&(2) \xrightarrow{L_2 \otimes (L_2, L_1, R_1, R_2)} (10) \quad && (3) \xrightarrow{L_1 \otimes (R_1, R_2, L_2, L_1)} (11)  \\
&(5) \xrightarrow{U_2 \otimes (L_2, L_1, R_1, R_2)} (12) \quad && (6) \xrightarrow{R_2 \otimes (L_2, L_1, R_1, R_2)} (12). \\
\end{alignat*}
The terms from $S_3$ are the same, so the $i=5$ sums are zero.

The $i = 6$ sums (ignoring identity arrows) can come only from $S_1$, since there are no $\delta^1_5$ arrows. Each term would be a product of output labels on a composable pair of $\delta^1_3$ (green) arrows and $\delta^1_4$ (red) arrows (in either order). However, all such products are zero by relations in the algebra, so the $i = 6$ sums are zero. The $i = 7$ sums are zero because there are no composable pairs of $\delta^1_4$ (red) arrows in Figure~\ref{fig:UnsimplifiedDABimod12}. The sums for $i \geq 8$ are also zero. Thus, the summand of $\X^{DA}$ over $(\B(2,1),\B(2,1))$ is a valid $DA$ bimodule.

For the summand of $\X^{DA}$ over $(\B(2,2),\B(2,2))$, label the basis elements as $(1)$--$(8)$. The $i=1$ sums come from composable pairs of $\delta^1_1$ (blue) arrows in Figure~\ref{fig:UnsimplifiedDABimod3}. The nonzero terms, each appearing twice, are as follows:
\begin{alignat*}{2}
&(5) \xrightarrow{L_1 U_2} (1) \quad && (6) \xrightarrow{U_2} (1)  \\
& (7) \xrightarrow{R_2 U_1} (1) \quad && (8) \xrightarrow{R_1 U_2} (2) \\
& (8) \xrightarrow{U_1} (3) \quad &&  (8) \xrightarrow{L_2 U_1} (4). \\
\end{alignat*}
Since each term appears twice, the $i=1$ sums are zero. The $i = 2$ sums all involve identity arrows.

The $i=3$ sums come from composable pairs of $\delta^1_1$ (blue) arrows and $\delta^1_3$ (green) arrows (in either order) in Figure~\ref{fig:UnsimplifiedDABimod3}. The sums taken together have the following nonzero terms, organized by the vertex appearing in the middle of the composition.

\begin{convention}\label{conv:LabelingWithSets}
A set of algebra elements in a term below should be expanded out into multiple terms; we use set notation to save space.
\end{convention}

\begin{itemize}

\item With middle vertex $(1)$:
\begin{alignat*}{1}
&(2) \xrightarrow{U_1 U_2 \otimes \{(U_1, U_2), (L_2 U_1, R_2), (R_1, L_1 U_2)\}} (2) \\ 
&(2) \xrightarrow{L_1 L_2 U_1 \otimes \{(U_1, U_2), (L_2 U_1, R_2), (R_1, L_1 U_2)\}} (4) \\ 
&(3) \xrightarrow{R_1 U_2^2 \otimes \{(U_1, U_2), (L_2 U_1, R_2), (R_1, L_1 U_2)\}} (2) \\ 
&(3) \xrightarrow{L_2 U_1 U_2 \otimes \{(U_1, U_2), (L_2 U_1, R_2), (R_1, L_1 U_2)\}} (4) \\ 
&(4) \xrightarrow{R_2 R_1 U_2 \otimes \{(U_1, U_2), (L_2 U_1, R_2), (R_1, L_1 U_2)\}} (2) \\ 
&(4) \xrightarrow{U_1 U_2 \otimes \{(U_1, U_2), (L_2 U_1, R_2), (R_1, L_1 U_2)\}} (4). \\ 
\end{alignat*}

\item With middle vertex $(2)$:
\begin{alignat*}{1}
&(1) \xrightarrow{U_1 U_2 \otimes \{(U_1, U_2), (L_2 U_1, R_2), (R_1, L_1 U_2)\}} (1) \\ 
&(5) \xrightarrow{L_1 L_2 U_2 \otimes \{(U_1, U_2), (L_2 U_1, R_2), (R_1, L_1 U_2)\}} (7) \\ 
&(7) \xrightarrow{U_1 U_2 \otimes \{(U_1, U_2), (L_2 U_1, R_2), (R_1, L_1 U_2)\}} (7). \\ 
\end{alignat*}

\item With middle vertex $(3)$: 
\begin{alignat*}{1}
&(5) \xrightarrow{U_1 U_2 \otimes \{(U_1, U_2), (L_2 U_1, R_2), (R_1, L_1 U_2)\}} (5) \\ 
&(5) \xrightarrow{L_1 U_1 U_2 \otimes \{(U_1, U_2), (L_2 U_1, R_2), (R_1, L_1 U_2)\}} (6) \\ 
&(6) \xrightarrow{R_1 U_2 \otimes \{(U_1, U_2), (L_2 U_1, R_2), (R_1, L_1 U_2)\}} (5) \\ 
&(6) \xrightarrow{U_1 U_2 \otimes \{(U_1, U_2), (L_2 U_1, R_2), (R_1, L_1 U_2)\}} (6). \\ 
\end{alignat*}

\item With middle vertex $(4)$:
\begin{alignat*}{1}
&(1) \xrightarrow{U_1 U_2 \otimes \{(U_1, U_2), (L_2 U_1, R_2), (R_1, L_1 U_2)\}} (1). \\ 
&(6) \xrightarrow{R_1 U_2 \otimes \{(U_1, U_2), (L_2 U_1, R_2), (R_1, L_1 U_2)\}} (5) \\ 
&(6) \xrightarrow{U_1 U_2 \otimes \{(U_1, U_2), (L_2 U_1, R_2), (R_1, L_1 U_2)\}} (6) \\ 
&(7) \xrightarrow{R_2 R_1 U_1 \otimes \{(U_1, U_2), (L_2 U_1, R_2), (R_1, L_1 U_2)\}} (5) \\ 
&(7) \xrightarrow{R_2 U_1^2 \otimes \{(U_1, U_2), (L_2 U_1, R_2), (R_1, L_1 U_2)\}} (6) \\ 
\end{alignat*}

\item With middle vertex $(5)$:
\begin{alignat*}{1}
&(3) \xrightarrow{R_1 U_2^2 \otimes \{(U_1, U_2), (L_2 U_1, R_2), (R_1, L_1 U_2)\}} (2) \\ 
&(3) \xrightarrow{U_1 U_2 \otimes \{(U_1, U_2), (L_2 U_1, R_2), (R_1, L_1 U_2)\}} (3) \\ 
&(4) \xrightarrow{R_2 R_1 U_2 \otimes \{(U_1, U_2), (L_2 U_1, R_2), (R_1, L_1 U_2)\}} (2) \\ 
&(4) \xrightarrow{R_2 U_1 \otimes \{(U_1, U_2), (L_2 U_1, R_2), (R_1, L_1 U_2)\}} (3) \\ 
&(8) \xrightarrow{U_1 U_2 \otimes \{(U_1, U_2), (L_2 U_1, R_2), (R_1, L_1 U_2)\}} (8). \\ 
\end{alignat*}

\item With middle vertex $(6)$:
\begin{alignat*}{1}
&(3) \xrightarrow{U_1 U_2 \otimes \{(U_1, U_2), (L_2 U_1, R_2), (R_1, L_1 U_2)\}} (3) \\ 
&(3) \xrightarrow{L_2 U_1 U_2 \otimes \{(U_1, U_2), (L_2 U_1, R_2), (R_1, L_1 U_2)\}} (4) \\ 
&(4) \xrightarrow{R_2 U_1 \otimes \{(U_1, U_2), (L_2 U_1, R_2), (R_1, L_1 U_2)\}} (3) \\ 
&(4) \xrightarrow{U_1 U_2 \otimes \{(U_1, U_2), (L_2 U_1, R_2), (R_1, L_1 U_2)\}} (4) \\ 
\end{alignat*}

\item With middle vertex $(7)$:
\begin{alignat*}{1}
&(2) \xrightarrow{U_1 U_2 \otimes \{(U_1, U_2), (L_2 U_1, R_2), (R_1, L_1 U_2)\}} (2) \\ 
&(2) \xrightarrow{L_1 L_2 U_1 \otimes \{(U_1, U_2), (L_2 U_1, R_2), (R_1, L_1 U_2)\}} (4) \\ 
&(8) \xrightarrow{U_1 U_2 \otimes \{(U_1, U_2), (L_2 U_1, R_2), (R_1, L_1 U_2)\}} (8). \\ 
\end{alignat*}

\item With middle vertex $(8)$:
\begin{alignat*}{1}
&(5) \xrightarrow{U_1 U_2 \otimes \{(U_1, U_2), (L_2 U_1, R_2), (R_1, L_1 U_2)\}} (5) \\ 
&(5) \xrightarrow{L_1 U_1 U_2 \otimes \{(U_1, U_2), (L_2 U_1, R_2), (R_1, L_1 U_2)\}} (6) \\ 
&(5) \xrightarrow{L_1 L_2 U_2 \otimes \{(U_1, U_2), (L_2 U_1, R_2), (R_1, L_1 U_2)\}} (7) \\ 
&(7) \xrightarrow{R_2 R_1 U_1 \otimes \{(U_1, U_2), (L_2 U_1, R_2), (R_1, L_1 U_2)\}} (5) \\ 
&(7) \xrightarrow{R_2 U_1^2 \otimes \{(U_1, U_2), (L_2 U_1, R_2), (R_1, L_1 U_2)\}} (6) \\ 
&(7) \xrightarrow{U_1 U_2 \otimes \{(U_1, U_2), (L_2 U_1, R_2), (R_1, L_1 U_2)\}} (7). \\ 
\end{alignat*}
\end{itemize}
These terms cancel in pairs, so the $i=3$ sums are zero.

The $i=4$ sums have $S_1$ terms which come from composable pairs of $\delta^1_1$ (blue) arrows and $\delta^1_4$ (red) arrows (in either order) in Figure~\ref{fig:UnsimplifiedDABimod3}. They also have $S_3$ terms which come from factorizing an algebra input of a $\delta^1_3$ (green) arrow. The $S_1$ terms are as follows.

\begin{itemize}
\item With middle vertex $(1)$:
\begin{alignat*}{3}
&(2) \xrightarrow{U_1 \otimes (R_1, U_2, L_1)} (5) \quad
&&(2) \xrightarrow{L_1 U_1 \otimes (R_1, U_2, L_1)} (6) \quad
&&(2) \xrightarrow{L_1 L_2 \otimes (L_2, U_1, R_2)} (7) \\
&(3) \xrightarrow{R_1 U_2 \otimes (R_1, U_2, L_1)} (5) \quad
&&(3) \xrightarrow{U_1 U_2 \otimes (R_1, U_2, L_1)} (6) \quad
&&(3) \xrightarrow{L_2 U_2 \otimes (L_2, U_1, R_2)} (7) \\
&(4) \xrightarrow{R_2 R_1 \otimes (R_1, U_2, L_1)} (5) \quad
&&(4) \xrightarrow{R_2 U_1 \otimes (R_1, U_2, L_1)} (6) \quad
&&(4) \xrightarrow{U_2 \otimes (L_2, U_1, R_2)} (7). \\
\end{alignat*}

\item With middle vertex $(2)$: 
\begin{alignat*}{2}
&(5) \xrightarrow{L_1 U_2 \otimes (R_1, U_2, L_1)} (8) \quad && (7) \xrightarrow{R_2 U_1 \otimes (R_1, U_2, L_1)} (8). \\
\end{alignat*}

\item With middle vertex $(3)$: 
\begin{alignat*}{2}
&(5) \xrightarrow{L_1 U_2 \otimes (L_2, U_1, R_2)} (8) \quad && (6) \xrightarrow{U_2 \otimes (L_2, U_1, R_2)} (8). \\
\end{alignat*}

\item With middle vertex $(4)$:
\begin{alignat*}{2}
&(6) \xrightarrow{U_2 \otimes (L_2, U_1, R_2)} (8) \quad && (7) \xrightarrow{R_2 U_1 \otimes (L_2, U_1, R_2)} (8). \\
\end{alignat*}

\item With middle vertex $(5)$:
\begin{alignat*}{2}
&(1) \xrightarrow{R_1 U_2 \otimes (R_1, U_2, L_1)} (2) \quad && (1) \xrightarrow{U_1 \otimes (R_1, U_2, L_1)} (3). \\
\end{alignat*}

\item With middle vertex $(6)$:
\begin{alignat*}{2}
&(1) \xrightarrow{U_1 \otimes (R_1, U_2, L_1)} (3) \quad && (1) \xrightarrow{L_2 U_1 \otimes (R_1, U_2, L_1)} (4). \\
\end{alignat*}

\item With middle vertex $(7)$:
\begin{alignat*}{2}
&(1) \xrightarrow{R_1 U_2 \otimes (L_2, U_1, R_2)} (2) \quad && (1) \xrightarrow{L_2 U_1 \otimes (L_2, U_1, R_2)} (4). \\
\end{alignat*}

\item With middle vertex $(8)$:
\begin{alignat*}{3}
&(2) \xrightarrow{U_1 \otimes (R_1, U_2, L_1)} (5) \quad && (2) \xrightarrow{L_1 U_1 \otimes (R_1, U_2, L_1)} (6)  \quad && (2) \xrightarrow{L_1 L_2 \otimes (R_1, U_2, L_1)} (7)  \\
&(3) \xrightarrow{R_1 U_2 \otimes (L_2, U_1, R_2)} (5) \quad && (3) \xrightarrow{U_1 U_2 \otimes (L_2, U_1, R_2)} (6) \quad && (3) \xrightarrow{L_2 U_2 \otimes (L_2, U_1, R_2)} (7)  \\
&(4) \xrightarrow{R_2 R_1 \otimes (L_2, U_1, R_2)} (5)  \quad && (4) \xrightarrow{R_2 U_1 \otimes (L_2, U_1, R_2)} (6) \quad && (4) \xrightarrow{U_2 \otimes (L_2, U_1, R_2)} (7). \\
\end{alignat*}
\end{itemize}

For the $S_3$ terms, note that when multiplied by the idempotent $\Ib_{AC}$, the nontrivial factorizations of the elements $U_1$ and $U_2$ of $\B(2)$ are $U_1 = (R_1) (L_1)$ and $U_2 = (L_2) (R_2)$. The nontrivial factorizations of ${_{AC}}(L_2 U_1)_{AB}$ are $L_2 U_1 = (L_2) (U_1) = (U_1) (L_2) = (R_1)(L_1 L_2)$ (for primitive idempotents $\Ib_{\x}$ and $\Ib_{\x'}$ of $\B(2)$, we write an element $a \in \B(2)$ as $a = {_{\x}}a_{\x'}$ if $a$ is a sum of paths from $\Ib_{\x}$ to $\Ib_{\x'}$ in the quiver defining $\B(2)$). The nontrivial factorizations of ${_{BC}}(L_1 U_2)_{AC}$ are $L_1 U_2 = (L_1) (U_2) = (U_2) (L_1) = (L_1 L_2) (R_2)$. 

Each green arrow in Figure~\ref{fig:UnsimplifiedDABimod3} labeled $(T0)$--$(T6)$ represents three terms of $\delta^1_3$ with algebra inputs $(U_1, U_2)$, $(L_2U_1, R_2)$, and $(R_1, L_1 U_2)$. Those with algebra inputs $(U_1,U_2)$ contribute to the sum over $S_3$ for the input sequences $(R_1, L_1, U_2)$ and $(U_1, L_2, R_2)$. The terms with algebra inputs $(L_2 U_1, R_2)$ contribute to the sum over $S_3$ for the input sequences $(L_2, U_1, R_2)$, $(U_1, L_2, R_2)$, and $(R_1, L_1 L_2, R_2)$. Finally, the terms with algebra inputs $(R_1, L_1 U_2)$ contribute to the sum over $S_3$ for the input sequences $(R_1, L_1, U_2)$, $(R_1, U_2, L_1)$, and $(R_1, L_1 L_2, R_2)$. 

All of the above contributions to the $S_3$ sum for a given input sequence are the same, and they cancel in pairs except when the input sequence is $(L_2, U_1, R_2)$ or $(R_1, U_2, L_1)$. Thus, the $S_3$ terms are as follows:
\begin{alignat*}{1}
& (1) \xrightarrow{R_1 U_2 \otimes \{ (L_2, U_1, R_2), (R_1, U_2, L_1)\}} (2) \\
& (1) \xrightarrow{L_2 U_1 \otimes \{ (L_2, U_1, R_2), (R_1, U_2, L_1)\}} (4) \\
& (2) \xrightarrow{L_1 L_2 \otimes \{ (L_2, U_1, R_2), (R_1, U_2, L_1)\}} (7) \\
& (3) \xrightarrow{R_1 U_2 \otimes \{ (L_2, U_1, R_2), (R_1, U_2, L_1)\}} (5) \\
& (3) \xrightarrow{U_1 U_2 \otimes \{ (L_2, U_1, R_2), (R_1, U_2, L_1)\}} (6) \\
& (4) \xrightarrow{R_2 R_1 \otimes \{ (L_2, U_1, R_2), (R_1, U_2, L_1)\}} (5) \\
& (4) \xrightarrow{R_2 U_1 \otimes \{ (L_2, U_1, R_2), (R_1, U_2, L_1)\}} (6) \\
& (5) \xrightarrow{L_1 U_2 \otimes \{ (L_2, U_1, R_2), (R_1, U_2, L_1)\}} (8) \\
& (7) \xrightarrow{R_2 U_1 \otimes \{ (L_2, U_1, R_2), (R_1, U_2, L_1)\}} (8). \\
\end{alignat*}
The $S_1$ and $S_3$ terms sum to zero, when taken together, so the $i=4$ sums are zero.

The $i = 5$ sums have $S_1$ terms which come from composable pairs of two $\delta^1_3$ (green) arrows in Figure~\ref{fig:UnsimplifiedDABimod3}. They also have $S_1$ terms which come from composable pairs of a $\delta^1_1$ (blue) arrow and a $\delta^1_5$ (teal) arrow. There are no $S_3$ terms because no algebra input of a $\delta^1_4$ (red) arrow can be factored nontrivially, in contrast with the summand of $\X^{DA}$ over $(\B(2,1),\B(2,1))$. 

Since each green arrow represents three terms of $\delta^1_3$, each composition of two green arrows represents nine terms, with nine different algebra input sequences. These input sequences are:
\begin{alignat*}{1}
& (U_1, U_2, U_1, U_2) \\
& (U_1, U_2, L_2 U_1, R_2) \\
& (U_1, U_2, R_1, L_1 U_2) \\
& (L_2 U_1, R_2, U_1, U_2) \\
& (L_2 U_1, R_2, L_2 U_1, R_2) \\
& (L_2 U_1, R_2, R_1, L_1 U_2) \\
& (R_1, L_1 U_2, U_1, U_2) \\
& (R_1, L_1 U_2, L_2 U_1, R_2) \\
& (R_1, L_1 U_2, R_1, L_1 U_2)
\end{alignat*}

Denote this set of nine sequences by $*$. In this notation, the $S_1$ terms coming from two green arrows are as follows.

\begin{itemize}
\item With middle vertex $(1)$: none. 

\item With middle vertex $(2)$:
\begin{alignat*}{1}
(1) \xrightarrow{L_2 U_1 U_2 \otimes *} (7) \\
\end{alignat*}

\item With middle vertex $(3)$: none. 

\item With middle vertex $(4)$:
\begin{alignat*}{1}
&(1) \xrightarrow{R_1 U_1 U_2 \otimes *} (5) \\
&(1) \xrightarrow{U_1^2 U_2 \otimes *} (6) \\
\end{alignat*}

\item With middle vertex $(5)$:
\begin{alignat*}{1}
&(3) \xrightarrow{U_1 U_2^2 \otimes *} (8) \\
&(4) \xrightarrow{R_2 U_1 U_2 \otimes *} (8) \\
\end{alignat*}

\item With middle vertex $(6)$: none. 

\item With middle vertex $(7)$:
\begin{alignat*}{1}
&(2) \xrightarrow{L_1 U_1 U_2 \otimes *} (8) \\
\end{alignat*}

\item With middle vertex $(8)$: none.
\end{itemize}

The teal edge labeled $(T7)$ in Figure~\ref{fig:UnsimplifiedDABimod3} represents nine terms of $\delta^1_5$, with $9$ different algebra input sequences. In fact, these sequences are exactly the ones in the set $*$. Thus, in the above notation, the $S_3$ terms coming from a blue and a teal arrow are given as follows. 
\begin{itemize}
\item With middle vertex $(1)$:
\begin{alignat*}{1}
&(2) \xrightarrow{L_1 U_1 U_2 \otimes *} (8) \\
&(3) \xrightarrow{U_1 U_2^2 \otimes *} (8) \\
&(4) \xrightarrow{R_2 U_1 U_2 \otimes *} (8) \\
\end{alignat*}

\item With middle vertex $(2)$--$(7)$: none. 

\item With middle vertex $(8)$:
\begin{alignat*}{1}
&(1) \xrightarrow{R_1 U_1 U_2 \otimes *} (5) \\
&(1) \xrightarrow{U_1^2 U_2 \otimes *} (6) \\
&(1) \xrightarrow{L_2 U_1 U_2 \otimes *} (7) \\
\end{alignat*}
\end{itemize}
The terms coming from a blue and a teal arrow and from two green arrows cancel in pairs, so the $i = 5$ sums are zero.

The $i=6$ sums have $S_1$ terms which come from composable pairs of $\delta^1_3$ (green) arrows and $\delta^1_4$ (red) arrows (in either order) in Figure~\ref{fig:UnsimplifiedDABimod3}. They also have $S_3$ terms which come from factorizing an algebra input of a $\delta^1_5$ (teal) arrow. The $S_1$ terms are as follows.
\begin{itemize}
\item With middle vertex $(1)$: none. 
\item With middle vertex $(2)$: 
\begin{alignat*}{1}
&(1) \xrightarrow{U_1 U_2 \otimes (U_1, U_2, R_1, U_2, L_1)} (8) \\
&(1) \xrightarrow{U_1 U_2 \otimes (L_2 U_1, R_2, R_1, U_2, L_1)} (8) \\
&(1) \xrightarrow{U_1 U_2 \otimes (R_1, L_1 U_2, R_1, U_2, L_1)} (8). \\
\end{alignat*}
\item With middle vertex $(3)$: none. 
\item With middle vertex $(4)$:
\begin{alignat*}{1}
&(1) \xrightarrow{U_1 U_2 \otimes (U_1, U_2, L_2, U_1, R_2)} (8) \\
&(1) \xrightarrow{U_1 U_2 \otimes (L_2 U_1, R_2, L_2, U_1, R_2)} (8) \\
&(1) \xrightarrow{U_1 U_2 \otimes (R_1, L_1 U_2, L_2, U_1, R_2)} (8). \\
\end{alignat*}
\item With middle vertex $(5)$:
\begin{alignat*}{1}
&(1) \xrightarrow{U_1 U_2 \otimes (R_1, U_2, L_1, U_1, U_2)} (8) \\
&(1) \xrightarrow{U_1 U_2 \otimes (R_1, U_2, L_1, L_2 U_1, R_2)} (8) \\
&(1) \xrightarrow{U_1 U_2 \otimes (R_1, U_2, L_1, R_1, L_1 U_2)} (8). \\
\end{alignat*}
\item With middle vertex $(6)$: none. 
\item With middle vertex $(7)$:
\begin{alignat*}{1}
&(1) \xrightarrow{U_1 U_2 \otimes (L_2, U_1, R_2, U_1, U_2)} (8) \\
&(1) \xrightarrow{U_1 U_2 \otimes (L_2, U_1, R_2, L_2 U_1, R_2)} (8) \\
&(1) \xrightarrow{U_1 U_2 \otimes (L_2, U_1, R_2, R_1, L_1 U_2)} (8). \\
\end{alignat*}
\item With middle vertex $(8)$: none.
\end{itemize}

For the $S_3$ terms, we analyze the possible factorizations of the inputs. For $(U_1, U_2, U_1, U_2)$, we have 
\begin{alignat*}{2}
&(R_1, L_1, U_2, U_1, U_2) \quad &&(U_1, L_2, R_2, U_1, U_2) \\
&(U_1, U_2, R_1, L_1, U_2) \quad && (U_1, U_2, U_1, L_2, R_2) \\
\end{alignat*}

For $(L_2 U_1, R_2, U_1, U_2)$, we have 
\begin{alignat*}{2}
&(L_2, U_1, R_2, U_1, U_2) \quad &&(U_1, L_2, R_2, U_1, U_2) \\
&(R_1, L_1 L_2, R_2, U_1, U_2) \quad &&(L_2 U_1, R_2, R_1, L_1, U_2) \\
&(L_2 U_1, R_2, U_1, L_2, R_2).
\end{alignat*} 

For $(R_1, L_1 U_2, U_1, U_2)$, we have 
\begin{alignat*}{2}
&(R_1, L_1, U_2, U_1, U_2)  \quad && (R_1, U_2, L_1, U_1, U_2)\\
&(R_1, L_1 L_2, R_2, U_1, U_2) \quad && (R_1, L_1 U_2, R_1, L_1, U_2) \\
&(R_1, L_1 U_2, U_1, L_2, R_2). 
\end{alignat*}

For $(U_1, U_2, L_2 U_1, R_2)$, we have 
\begin{alignat*}{2}
&(R_1, L_1, U_2, L_2 U_1, R_2) \quad && (U_1, L_2, R_2, L_2 U_1, R_2)\\
&(U_1, U_2, L_2, U_1, R_2) \quad && (U_1, U_2, U_1, L_2, R_2) \\
&(U_1, U_2, R_1, L_1 L_2, R_2). 
\end{alignat*}

For $(L_2 U_1, R_2, L_2 U_1, R_2)$, we have 
\begin{alignat*}{2}
&(L_2, U_1, R_2, L_2 U_1, R_2) \quad && (U_1, L_2, R_2, L_2 U_1, R_2) \\
&(R_1, L_1 L_2, R_2, L_2 U_1, R_2) \quad &&(L_2 U_1, R_2, L_2, U_1, R_2) \\
&(L_2 U_1, R_2, U_1, L_2, R_2) \quad &&(L_2 U_1, R_2, R_1, L_1 L_2, R_2).
\end{alignat*}
. 

For $(R_1, L_1 U_2, L_2 U_1, R_2)$, we have
\begin{alignat*}{2}
&(R_1, L_1, U_2, L_2 U_1, R_2) \quad && (R_1, U_2, L_1, L_2 U_1, R_2) \\
&(R_1, L_1 L_2, R_2, L_2 U_1, R_2) \quad && (R_1, L_1 U_2, L_2, U_1, R_2)\\
&(R_1, L_1 U_2, U_1, L_2, R_2) \quad && (R_1, L_1 U_2, R_1, L_1 L_2, R_2).
\end{alignat*}

For $(U_1, U_2, R_1, L_1 U_2)$, we have 
\begin{alignat*}{2}
&(R_1, L_1, U_2, R_1, L_1 U_2) \quad && (U_1, L_2, R_2, R_1, L_1 U_2) \\
&(U_1, U_2, R_1, L_1, U_2) \quad && (U_1, U_2, R_1, U_2, L_1) \\
&(U_1, U_2, R_1, L_1 L_2, R_2). 
\end{alignat*}

For $(L_2 U_1, R_2, R_1, L_1 U_2)$, we have 
\begin{alignat*}{2}
&(L_2, U_1, R_2, R_1, L_1 U_2) \quad && (U_1, L_2, R_2, R_1, L_1 U_2) \\
&(R_1, L_1 L_2, R_2, R_1, L_1 U_2) \quad && (L_2 U_1, R_2, R_1, L_1, U_2) \\
&(L_2 U_1, R_2, R_1, U_2, L_1) \quad && (L_2 U_1, R_2, R_1, L_1 L_2, R_2).
\end{alignat*} 

Finally, for $(R_1, L_1 U_2, R_1, L_1 U_2)$, we have 
\begin{alignat*}{2}\
&(R_1, L_1, U_2, R_1, L_1 U_2) \quad && (R_1, U_2, L_1, R_1, L_1 U_2) \\
&(R_1, L_1 L_2, R_2, R_1, L_1 U_2) \quad && (R_1, L_1U_2, R_1, L_1, U_2) \\
&(R_1, L_1 U_2, R_1, U_2, L_1) \quad && (R_1, L_1 U_2, R_1, L_1 L_2, R_2).
\end{alignat*}

Many pairs of these input factorizations cancel; the remaining ones are:
\begin{alignat*}{2}
&(L_2, U_1, R_2, U_1, U_2) \quad && (R_1, U_2, L_1, U_1, U_2) \\
&(U_1, U_2, L_2, U_1, R_2) \quad && (L_2, U_1, R_2, L_2 U_1, R_2) \\
&(L_2 U_1, R_2, L_2, U_1, R_2) \quad && (R_1, U_2, L_1, L_2 U_1, R_2) \\
&(R_1, L_1 U_2, L_2, U_1, R_2) \quad && (U_1, U_2, R_1, U_2, L_1) \\
&(L_2, U_1, R_2, R_1, L_1 U_2) \quad && (L_2 U_1, R_2, R_1, U_2, L_1) \\
&(R_1, U_2, L_1, R_1, L_1 U_2) \quad && (R_1, L_1 U_2, R_1, U_2, L_1). \\
\end{alignat*}
Thus, the $S_3$ terms cancel the $S_1$ terms, and the $i = 6$ sums are zero.

The $i = 7$ sums are zero; there are no composable pairs of $\delta^1_3$ (green) and $\delta^1_5$ (teal) arrows (in either order), or composable pairs of two $\delta^1_4$ (red) arrows, in Figure~\ref{fig:UnsimplifiedDABimod3}. The $i = 8$ sums are zero; there are no composable pairs of $\delta^1_4$ (red) and $\delta^1_5$ (teal) arrows (in either order). The $i = 9$ sums are zero; there are no composable pairs of $\delta^1_5$ (teal) arrows. The sums are zero for $i > 9$ as well. Thus, the summand of $\X^{DA}$ over $(\B(2,2),\B(2,2))$ is a valid $DA$ bimodule, so $\X^{DA}$ is a valid $DA$ bimodule.
\end{proof}

\subsection{Simplifying the local \texorpdfstring{$DA$}{DA} bimodule}

An inspection of Figure~\ref{fig:UnsimplifiedDABimod12} and Figure~\ref{fig:UnsimplifiedDABimod3} reveals three cancellable pairs, as defined in Section~\ref{sec:HowToSimplify}. We conclude that the $DA$ bimodules shown graphically in Figure~\ref{fig:SimplifiedDABimod12} and Figure~\ref{fig:SimplifiedDABimod3}, which are obtained from the summands of $\X^{DA}$ by the procedure described in Section~\ref{sec:HowToSimplify}, are valid $DA$ bimodules homotopy equivalent to the summands of $\X^{DA}$. We will call these three summands together $\Xt^{DA}$. The grading data of basis elements of $\Xt^{DA}$ agrees with the grading data of the corresponding basis elements of $\X^{DA}$.

\begin{figure}
	\includegraphics[scale=0.7]{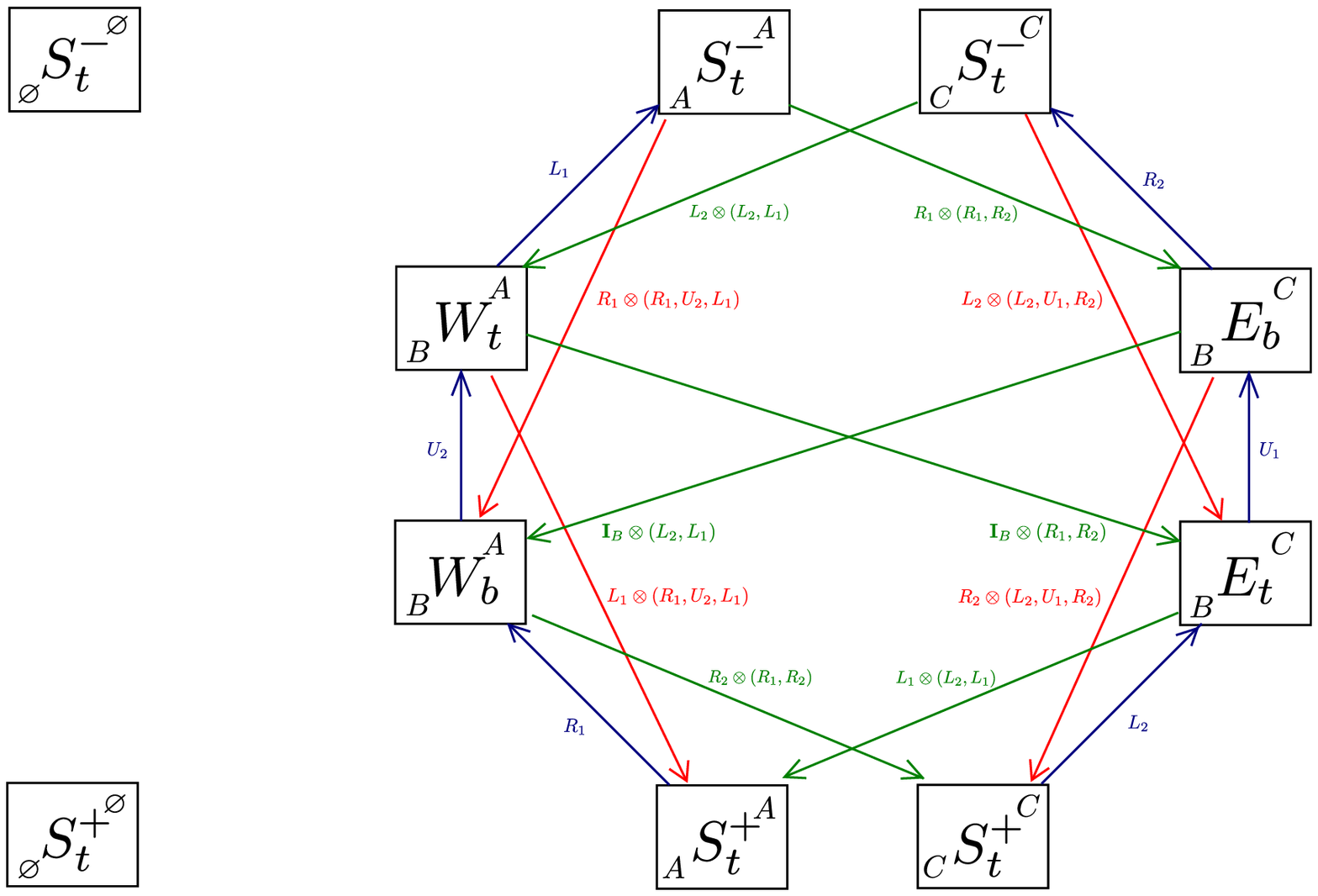}
	\caption{First two summands of the simplified local DA bimodule $\Xt^{DA}$ for a singular crossing.}
	\label{fig:SimplifiedDABimod12}
\end{figure}

\begin{figure}
	\includegraphics[scale=0.7]{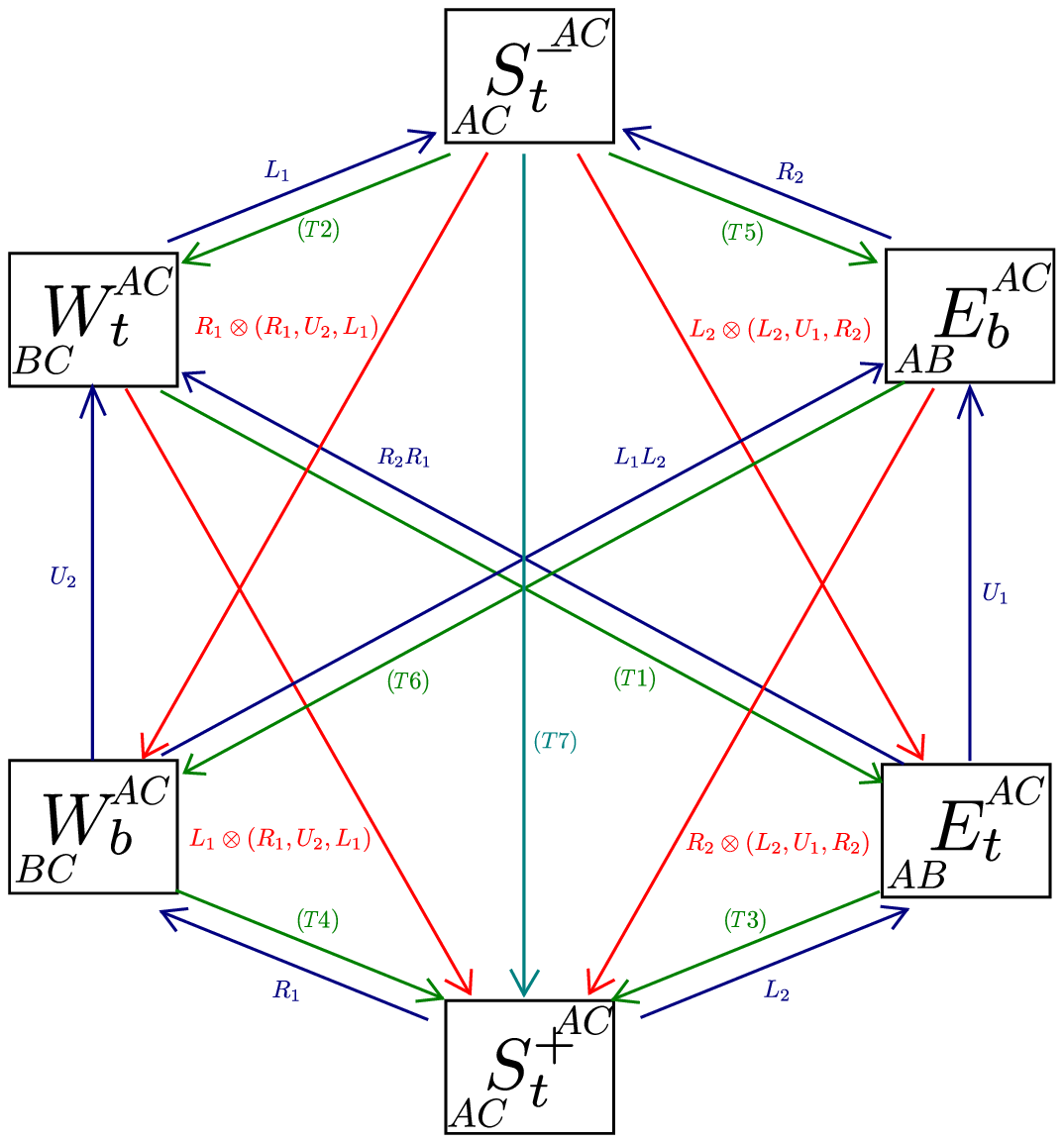}
	\caption{Third summand of the simplified local DA bimodule $\Xt^{DA}$ for a singular crossing.}
	\label{fig:SimplifiedDABimod3}
\end{figure}

\subsection{Symmetries in the local \texorpdfstring{$DA$}{DA} bimodule}

Like $\Xt^{DD}$, the summand of the bimodule $\Xt^{DA}$ over $(\B(2,1),\B(2,1))$ has two symmetries $\Rc$ and $o$. The summand over $(\B(2,2),\B(2,2))$ only has the symmetry $\Rc$, while $o$ exchanges this summand with a modified version where the actions are slightly different.

The basis elements of $\Xt^{DA}$ naturally correspond to the basis elements of $\Xt^{DD}$ (note that the right I-states of basis elements of $\Xt^{DA}$ are the complements of the right I-states of basis elements of $\Xt^{DD}$). We define $\Rc$ and $o$ on $\Xt^{DA}$ as in Section~\ref{sec:DDSymm}. The $DA$ bimodule operations on $\Xt^{DA}$ respect the symmetry $\Rc$: for $i \geq 1$, the square
\[
\xymatrix{
\Xt^{DA} \otimes_{\I(2)} \B(2)^{\otimes(i-1)} \ar[rr]^{\Rc \otimes (\Rc^{\otimes(i-1)})} \ar[d]_{\delta^1_i} & & \Xt^{DA} \otimes_{\I(2)} \B(2)^{\otimes(i-1)} \ar[d]^{\delta^1_i} \\
\B(2) \otimes_{\I(2)} \Xt^{DA} \ar[rr]_{\Rc \otimes \Rc} & & \B(2) \otimes_{\I(2)} \Xt^{DA}
}
\]
commutes, where $\delta^1_i$ is the $i^{th}$ $DA$ operation on $\Xt^{DA}$. Equivalently, $\Rc$ gives a $DA$ bimodule isomorphism from $\Xt^{DA}$ to $\Ind_{\Rc} \Rest_{\Rc} \Xt^{DA}$, where this latter bimodule has operations defined by applying $\Rc^{\otimes(i-1)}$ to algebra inputs in $\B(2)^{\otimes(i-1)}$, applying the operation $\delta^1_i$ on $\Xt^{DD}$, and then applying $\Rc$ to the algebra output (see \cite[Section 2.4.2]{LOTBimodules}). Graphically, Figure~\ref{fig:SimplifiedDABimod12} and Figure~\ref{fig:SimplifiedDABimod3} are unchanged by reflecting about the vertical axis of the graph for each summand while applying $\Rc$ to each input and output algebra label. 

Similarly, we view $o$ as a map from $\Xt^{DA}$ to its dual $(\Xt^{DA})^{\vee}$ as defined in Section~\ref{sec:DualsOfModules}. For the summand of $\Xt^{DA}$ over $(\B(2,1),\B(2,1))$, the square
\begin{equation}\label{eq:OSymmetryOnDA}
\xymatrix{
\Xt^{DA} \otimes \B(2)^{\otimes(i-1)} \ar[rrr]^{o \otimes o^{\otimes(i-1)}} \ar[d]_{\delta^1_i} & & & (\Xt^{DA})^{\vee} \otimes (\B(2)^{\op})^{\otimes(i-1)} \ar[d]^{(\delta^1_i)^{\vee}} \\
\B(2) \otimes \Xt^{DA} \ar[rrr]_{o \otimes o} & & & \B(2)^{\op} \otimes (\Xt^{DA})^{\vee}
}
\end{equation}
commutes. Equivalently, $o$ gives a $DA$ bimodule isomorphism from this summand $\Xt^{DA}$ to the corresponding summand of $\Ind_o \Rest_o (\Xt^{DA})^{\vee}$. Graphically, Figure~\ref{fig:SimplifiedDABimod12} is unchanged by reflecting about the horizontal axis of the graph for this summand while applying $o$ to each input and output algebra label, reversing the directions of arrows, and reversing the order of each algebra input sequence (this order reversal is implicit in $(\delta^1_i)^{\vee}$ as discussed in Section~\ref{sec:DualsOfModules}).

\begin{remark}\label{rem:OSymmetry}
The square \eqref{eq:OSymmetryOnDA} for the summand of $\Xt^{DA}$ over $(\B(2,2),\B(2,2))$ does not commute. Note that there is an asymmetry in the definition of this summand, not visible in the corresponding $DD$ bimodule. The green $\delta^1_3$ edges in Figure~\ref{fig:UnsimplifiedDABimod3} (and thus Figure~\ref{fig:SimplifiedDABimod3}) have algebra inputs $(U_1,U_2)$, $(L_2U_1, R_2)$, and $(R_1, L_1U_2)$, but we could have equally well used $(U_2,U_1)$, $(L_2, R_1 U_1)$, and $(R_1 U_2, L_1)$. The teal $\delta^1_5$ edge is similar. If we let $(\delta')^1_i$ denote the $DA$ operation on $\Xt^{DA}$ defined using this alternate pattern of algebra inputs instead of the original pattern when $i = 3,5$, then the square \eqref{eq:OSymmetryOnDA} commutes when the left edge is replaced by $(\delta')^1_i$. In other words, $o$ gives an isomorphism between this modification of the $\delta^1_3$ and $\delta^1_5$ actions on this summand of $\Xt^{DD}$ and the corresponding summand of the bimodule $\Ind_o \Rest_o (\Xt^{DA})^{\vee}$. 

Note that the alternate version of $\X^{DA}$ is well-defined even though it does not admit a symmetry corresponding to $o$; we do not see any reason to prefer the set \[
\{(U_1,U_2), (L_2U_1, R_2), (R_1, L_1U_2)\}
\]
over the set 
\[
\{(U_2,U_1), (L_2, R_1 U_1), (R_1 U_2, L_1)\}.
\] 
\end{remark}

As with $\X^{DD}$, the unsimplified bimodule $\X^{DA}$ has a symmetry corresponding to the composition of $\Rc$ and $o$ (with the same caveat as in Remark~\ref{rem:OSymmetry}); graphically, Figure~\ref{fig:UnsimplifiedDABimod12} and Figure~\ref{fig:UnsimplifiedDABimod3} are unchanged under $180^{\circ}$ rotation combined with applying $\Rc o$ to each algebra label, reversing the directions of edges, and reversing the orders of sequences of algebra inputs for $\delta^1_i$ edges with $i \geq 3$.

\subsection{Obtaining \texorpdfstring{$DD$}{DD} bimodules from \texorpdfstring{$DA$}{DA} bimodules}

Let $\K$ denote the canonical (left, left) $DD$ bimodule over $(\B(2),\B^!(2))$ from Definition~\ref{def:KoszulDualizing}. Since $\X^{DA}$ has no $\delta^1_j$ operations for $j > 5$, the box tensor product $\X^{DA} \boxtimes \K$ makes sense. We have a natural correspondence between basis elements of $\X^{DD}$ and $\X^{DA}$ and thus between basis elements of $\X^{DD}$ and $\X^{DA} \boxtimes \K$. 

\begin{proposition}\label{prop:LocalDDFromLocalDA}
The natural correspondence of basis elements gives isomorphisms $\X^{DD} \cong \X^{DA} \boxtimes \K$ and $\Xt^{DD} \cong \Xt^{DA} \boxtimes \K$ of $DD$ bimodules over $(\B(2),\B^!(2))$.
\end{proposition}

\begin{proof}
The correspondence of basis elements is an $(\I(2),\I(2))$--bilinear map preserving the degrees of all basis elements. We must check that the operation $\delta^1$ on $\X^{DD}$ agrees with $\delta^{\boxtimes}$ on $\X^{DA} \boxtimes \K$ under this correspondence. We can compute $\delta^{\boxtimes}$ as follows: for each edge of a directed graph in Figure~\ref{fig:UnsimplifiedDABimod12} and Figure~\ref{fig:UnsimplifiedDABimod3}, each term in the sequence of algebra inputs on the edge is either $R_i$, $L_i$, $U_i$, or something else. If any terms are not $R_i$, $L_i$, or $U_i$, the edge does not contribute to $\delta^{\boxtimes}$. Otherwise, we form an element of $\B^!(2)$ as the product (in order) of $L_i$ for each $R_i$ term in the sequence, $R_i$ for each $L_i$ term, and $C_i$ for each $U_i$ term. In the directed graph for $\X^{DA} \boxtimes \K$, the sequence of algebra inputs on this edge is replaced by the element of $\B^!(2)$ that was produced. One can check that following this procedure produces the graphs in Figure~\ref{fig:UnsimplifiedLocalBimod12} and Figure~\ref{fig:UnsimplifiedLocalBimod3} from the graphs in Figure~\ref{fig:UnsimplifiedDABimod12} and Figure~\ref{fig:UnsimplifiedDABimod3}. Comparing Figure~\ref{fig:SimplifiedDABimod12} and Figure~\ref{fig:SimplifiedDABimod3} with Figure~\ref{fig:SimplifiedLocalBimod12} and Figure~\ref{fig:SimplifiedLocalBimod3}, the same is true for the simplified bimodules.
\end{proof}

\begin{corollary}
The operation $\delta^1$ on $\X^{DD}$ is a well-defined $DD$ bimodule operation. The same is true for $\Xt^{DD}$.
\end{corollary}

\begin{proof}
By Proposition~\ref{prop:LocalDDFromLocalDA}, $\delta^1$ corresponds to $\delta^{\boxtimes}$ which is a well-defined $DD$ bimodule operation. The same argument holds for $\Xt^{DD}$; alternatively, $\Xt^{DD}$ is obtained from $\X^{DD}$ by cancelling a valid cancellable pair as in Definition~\ref{def:DDCancelling}.
\end{proof}

Strictly speaking, we should have defined $\X^{DD}$ to be $\X^{DA} \boxtimes \K$; then the proof of Proposition~\ref{prop:LocalDDFromLocalDA} checks that the $DD$ bimodule operation $\delta^1$ on $\X^{DD}$ is accurately described by the directed graphs in Figure~\ref{fig:UnsimplifiedLocalBimod12} and Figure~\ref{fig:UnsimplifiedLocalBimod3}.

\section{The global \texorpdfstring{$DD$}{DD} bimodule for a singular crossing}\label{sec:GlobalDD}

\subsection{Definition as a module over the idempotent ring}

We want to define a $DD$ bimodule ${^{\B(n),\B^!(n)}}\X^{DD}_i$ for $1 \leq i \leq n-1$. Our local model will be the unsimplified local bimodule $\X^{DD}$. For convenience, let $\B := \B(n) \otimes \B^!(n)$. Similarly, let $\B_{\loc} := \B(2) \otimes \B^!(2)$. We will define $\X^{DD}_i$ as $(M,\delta^1)$, where $M$ is a left module over $\I := \I(n) \otimes \I(n)$ and $\delta^1\co M \to \B \otimes_{\I} M$ is an $\I$-linear map satisfying the $DD$ bimodule relations.

\begin{definition} Given $i$ with $1 \leq i \leq n-1$, we define an \emph{external I-state} to be a subset $\x_{\ext}$ of $S_{\external} = \{0,\ldots,i-2\} \cup \{i+2, \ldots,n\}$. Visually, $S_{\external}$ is the set of regions between and outside $n$ parallel strands except for the three regions adjacent to strands $i$ and $i+1$.
\end{definition}

\begin{definition}\label{def:GlobalDDAsVectSpace} As a vector space over $\F$, $M$ is defined to be a direct sum of copies of $\X^{DD}$ indexed by external I-states:
\[
M := \bigoplus_{\x_{\ext}} \X^{DD}
\]
where the direct sum runs over all external I-states $\x_{\ext}$. When we want to give explicit names to the summands, we will write $M$ as 
\[
M := \bigoplus_{\x_{\ext}} \X_{\x_{\ext}}^{DD}.
\]
\end{definition}

First we define gradings on $\X^{DD}_i$, i.e. on $M$.
\begin{definition}\label{def:GradingsOnGlobalDD}
Let $m$ be a basis element of $M$. The data of $m$ consist of an external I-state $\x_{\ext}$ and a basis element $m_{\loc}$ of the local $DD$ bimodule $\X^{DD}$, and we can write $m = (\x_{\ext}, m_{\loc})$.
\begin{itemize}
\item The homological grading on $M$ is a $\Z$ grading defined by $\m(\x_{\ext}, m_{\loc}) = \m(m_{\loc})$, where $\m(m_{\loc})$ is the homological degree from Definition~\ref{def:HomGrLocalDD}.
\item The first unrefined grading on $M$ is a $\Z^{2n}$ grading defined by $\deg^{\un}_1(\x_{\ext}, m_{\loc}) = \iota(\deg^{\un}_1(m_{\loc}))$, where $\deg^{\un}_1(m_{\loc})$ is the first unrefined degree from Definition~\ref{def:Ref1LocalDD} and $\iota\co \Z^4 \to \Z^{2n}$ is the inclusion of the summand spanned by the subset $\{\tau_i, \tau_{i+1}, \beta_i, \beta_{i+1}\}$ of $\Z^{2n} = \Z\langle \tau_1, \beta_1, \ldots, \tau_n, \beta_n\rangle$.
\item The second unrefined grading on $M$ is a $(\frac{1}{4}\Z)^{2n}$ grading defined as in the previous item, but using the second unrefined grading $\deg^{\un}_2$ from Definition~\ref{def:Ref2LocalDD}.
\item The refined grading on $M$ is a $(\frac{1}{4}\Z)^n$ grading obtained by applying the homomorphism of grading groups $\eta\co (\frac{1}{4}\Z)^{2n} \to (\frac{1}{8}\Z)^n$ to the second unrefined grading, where $e_1,\ldots,e_n$ are the standard generators of $\Z^n$ and $\eta$ sends $\tau_i$ and $\beta_i$ to $\frac{e_i}{2}$ (note that $\eta$ sends the degree of any basis element of $M$ to an element of $(\frac{1}{4}\Z)^n \subset (\frac{1}{8}\Z)^n$).
\item The single Alexander grading on $M$ is a $\frac{1}{2}\Z$ grading obtained by applying the sum map $(\frac{1}{4}\Z)^n \to \frac{1}{4}\Z$ to the refined grading (note that this sum map sends the degree of any basis element of $M$ to an element of $\frac{1}{2}\Z \subset \frac{1}{4}\Z$).
\end{itemize}
\end{definition}

Next we define how the idempotent ring $\I$ acts on $M$. The basis element $m_{\loc}$ has a local idempotent $\Ib_{1,\local} \otimes \Ib_{2,\local}$, where $\Ib_{j,\local} = \Ib_{\x_{j,\local}}$ for some local I-state $\x_{j,\local} \subset \{0,1,2\}$. From this information, plus the external I-state $\x_{\ext}$, we construct a primitive idempotent $\Ib_{1,\glob} \otimes \Ib_{2,\glob}$ in $\I$ below.

\begin{definition}
Define $\x_{1,\glob} \subset \{0,\ldots,n\}$ by setting $\x_{1,\glob} \cap \{i-1,i,i+1\} := \{j + i \,|\, j \in \x_{1,\local}\}$ and $\x_{1,\glob} \cap (\{0,\ldots,i-2\} \cup \{i+2,\ldots,n\}) := \x_{\ext}$. Visually, for the $n+1$ regions between and outside $n$ parallel strands, $\x_{1,\glob}$ contains a region adjacent to strands $i$ and $i+1$ if and only if $\x_{1,\local}$ contains the corresponding local region. For the rest of the regions, $\x_{1,\glob}$ contains them if and only if $\x_{\ext}$ does.

Similar,y define $\x_{2,\glob} \subset \{0,\ldots,n\}$ by setting $\x_{2,\glob} \cap \{i-1,i,i+1\} := \{j + i \,|\, j \in \x_{2,\local}\}$ and $\x_{2,\glob} \cap (\{0,\ldots,i-2\} \cup \{i+2,\ldots,n\}) := \{0,\ldots,n\} \setminus \x_{\ext}$. Visually, $\x_{2,\glob}$ contains a region adjacent to strands $i$ and $i+1$ if and only if $\x_{2,\local}$ contains the corresponding local region. For the rest of the regions, $\x_{2,\glob}$ contains them if and only if $\x_{\ext}$ does not.

For $j = 1,2$, let $\Ib_{j,\glob} = \Ib_{\x_{j,\glob}}$. We call $\Ib_{1,\glob}$ the \emph{first idempotent} of $m = (\x_{\ext}, m_{\loc})$ and $\Ib_{2,\glob}$ the \emph{second idempotent} of $m$. 
\end{definition}

For a primitive idempotent $\Ib_1 \otimes \Ib_2$ of $\I$, we can define $(\Ib_1 \otimes \Ib_2) \cdot m$ to be $m$ if $\Ib_1 = \Ib_{1,\glob}$ and $\Ib_2 = \Ib_{2,\glob}$, and to be zero otherwise.

\subsection{Local component of the differential}
Now that we have defined $M$ as an $\I$-module, we want to define $\delta^1\co M \to \B \otimes_{\I} M$.

\begin{definition} The differential $\delta^1$ is defined to be $\delta^1_{\local} + \delta^1_{\external}$, where the $\I$-linear maps $\delta^1_{\local}$ and $\delta^1_{\external}$ will be defined below.
\end{definition}

We define the local component $\delta^1_{\local}$ first (obtained from, but strictly speaking different from, the truly local differential on $\X^{DD}$ that we will call $\delta^1_{\loc}$). Note that given $n \geq 2$ and $i \in \{1,\ldots,n-1\}$, there are (non-unital) dg algebra homomorphisms $\Phi\co \B(2) \to \B(n)$ and $\Phi^!\co \B^!(2) \to \B^!(n)$ sending vertices $\x \subset \{0,1,2\}$ of the quivers defining $\B(2)$ and $\B^!(2)$ to the vertices $i-1+ \x \subset \{0,\ldots,n\}$ of the quivers defining $\B(n)$ and $\B^!(n)$. These maps send edges labeled $R_j$, $L_j$, $U_j$, and $C_j$ to edges labeled $R_{i+j-1}$, $L_{i+j-1}$, $U_{i+j-1}$, and $C_{i+j-1}$ respectively. The relations defining $\B(2)$ and $\B^!(2)$ are satisfied in $\B(n)$ and $\B^!(n)$. 

\begin{definition} Given $n$ and $i$ as above, let $\x_{\ext}$ be an external I-state. Define a map $\iota_{\x_{\ext}}\co \B_{\loc} \rightarrow \B$ by $\iota_{\x_{\ext}} = \Phi \otimes \Phi^!$ (recall that $\B_{\loc} = \B(2) \otimes \B^!(2)$ and $\B = \B(n) \otimes \B^!(n)$). The map $\iota_{\x_{\ext}}$ respects multiplication and differentials but not the units of the algebras $\B_{\loc}$ and $\B$.

\end{definition}

\begin{definition}
Given $n$ and $i$, let $\x_{\ext}$ be an external I-state. The map 
\[
\delta^1_{\local, \x_{\ext}}\co \X_{\x_{\ext}}^{DD} \to \B \otimes_{\I} \X_{\x_{\ext}}^{DD}
\]
is defined to be $(\iota_{\x_{\ext}} \otimes \id) \circ \delta^1_{\loc}$, where $\delta^1_{\loc}\co \X^{DD} \to (\B(2) \otimes \B^!(2))_{\I(2) \otimes \I(2)} \otimes \X^{DD}$ is the local $DD$ bimodule operation $\delta^1$ from Definition~\ref{def:XDD} (identifying $\X^{DD}$ and the summand $\X_{\x_{\ext}}^{DD}$ of $\X_i^{DD}$).
\end{definition}

\begin{proposition} The map $\delta^1_{\local, \x_{\ext}}$ is $\I$-linear.
\end{proposition}

\begin{proof} This fact follows from the $\I(2) \otimes \I(2)$-linearity of $\delta^1_{\loc}$ and how the $\I$-action was defined on $M$, of which $\X_{\x_{\ext}}^{DD}$ is an $\I$-submodule.
\end{proof}

\begin{proposition} The map $\delta^1_{\local, \x_{\ext}}$ satisfies 
\[
(\partial_{\B} \otimes \id) \circ \delta^1_{\local, \x_{\ext}} + (\mu_{\B} \otimes \id) \circ (\id \otimes \delta^1_{\local, \x_{\ext}}) \circ \delta^1_{\local, \x_{\ext}} = 0.
\]
\end{proposition}

\begin{proof} This equation follows from the same structure relation for $\delta^1_{\loc}$.
\end{proof}

\begin{definition}
Define $\delta^1_{\local}\co M \to \B \otimes_{\I} M$ by
\[
\delta^1_{\local} := \oplus_{\x_{\ext}} \delta^1_{\local, \x_{\ext}},
\]
where the sum is over external I-states.
\end{definition}

\begin{proposition}\label{prop:HalfLocalDD} The map $\delta^1_{\local}$ is $\I$-linear and satisfies the DD relation
\[
(\partial_{\B} \otimes \id) \circ \delta^1_{\local} + (\mu_{\B} \otimes \id) \circ (\id \otimes \delta^1_{\local}) \circ \delta^1_{\local} = 0.
\]
\end{proposition}

\begin{proof} 
Both statements are true for $\delta^1_{\local}$ restricted to each summand $\X_{\x_{\ext}}^{DD}$ of $\X_i^{DD}$, so they are true in general.
\end{proof}

\subsection{External component of the differential}

Now that we have defined $\delta^1_{\local}$, we need to define $\delta^1_{\external}$. We can define $\delta^1_{\external}$ as the sum of two maps:
\[
\delta^1_{\external} := \delta^1_{\ext, \unmoving} + \delta^1_{\ext, \moving}.
\]

We first define $\delta^1_{\ext, \unmoving}$.
\begin{definition} For a strand $j \neq i, i+1$, the map $\delta^1_{\ext, \unmoving, j}\co M \to \B \otimes_{\I} M$ sends each basis element $m \in M$ to $(U_j \otimes C_j) \otimes m$. The total unmoving component $\delta^1_{\ext, \unmoving}$ is defined to be
\[
\delta^1_{\ext, \unmoving} := \sum_{j \neq i, i+1} \delta^1_{\ext, \unmoving, j}.
\]

\end{definition}

To finish defining the global $DD$ bimodule $\X^{DD}_i$, we just need to define $\delta^1_{\ext, \moving}$. This map will also be a sum over external strands of maps $\delta^1_{\ext,\moving,j}$. First, we partition the gemerators of $M$ into three sets $A_j$, $B_j$, and $C_j$.

\begin{definition} Let $j \neq i, i+1$ be an index of an external strand. The strand $j$ is adjacent to two regions, namely regions $j-1$ and $j$. Let $m \in M$ be a basis element with left idempotent $(\Ib_{\x_1} \otimes \Ib_{\x_2}) \in \I$.
\begin{itemize}
\item  If $j-1 \in \x_1$, $j-1 \notin \x_2$, $j \notin \x_1$, and $j \in \x_2$, then $m \in A_j$.
\item If $j \in \x_1$, $j \notin \x_2$, $j-1 \notin \x_1$, and $j-1 \in \x_2$, then $m \in B_j$.
\item Otherwise, $m \in C_j$.
\end{itemize}
\end{definition}

Note that each basis element $m = (\x_{\ext}, m_{\loc})$ in $A_j$ has a counterpart in $B_j$, obtained by modifying $\x_{\ext}$ appropriately.

\begin{definition} Let $j \neq i, i+1$ be an external strand. Define $\sigma_j\co A_j \to B_j$ to be the bijection sending a basis element to its counterpart. Define $\tau_j\co B_j \to A_j$ to be the inverse of $\sigma_j$.
\end{definition}

\begin{definition} Let $j \neq i, i+1$ be an external strand. For a basis element $m$ of $M$, define $\delta^1_{\ext,\moving,j}(m)$ to be 
\begin{itemize}
\item $(R_j \otimes L_j) \otimes \sigma_j(m)$ if $m \in A_j$;
\item $(L_j \otimes R_j) \otimes \tau_j(m)$ if $m \in B_j$;
\item zero otherwise.
\end{itemize}
\end{definition}

\begin{proposition} Each map $\delta^1_{\ext,\moving,j}$ is $\I$-linear.
\end{proposition}

\begin{proof} The effects of $(R_j \otimes L_j)$ and $(L_j \otimes R_j)$ on idempotents exactly offset the effects of the maps $\sigma_j$ and $\tau_j$.
\end{proof}

\begin{definition} 
The $\I$-linear map $\delta^1_{\ext,\moving}$ is defined to be $\sum_{j \neq i, i+1} \delta^1_{\ext,\moving,j}$.
\end{definition}

\section{Checking the \texorpdfstring{$DD$}{DD} relations for the global singular bimodule}\label{sec:GlobalDDWellDefined}

We want to show that
\[
(\partial_{\B} \otimes \id) \circ \delta^1 + (\mu_{\B} \otimes \id) \circ (\id \otimes \delta^1) \circ \delta^1 = 0.
\]
Writing out $\delta^1$ as $\delta^1_{\local} + \delta^1_{\external}$, and using the fact from Proposition~\ref{prop:HalfLocalDD} that $\delta^1_{\local}$ satisfies the $DD$ relations, the desired equation has two parts. The first says that $\delta^1_{\external}$ also satisfies the $DD$ relations, and the second is a compatibility between $\delta^1_{\local}$ and $\delta^1_{\external}$. We will deal with these two assertions one at a time, and then combine them to obtain the global $DD$ relations.

\subsection{Proof that \texorpdfstring{$\delta^1_{\external}$}{delta\_external} satisfies the \texorpdfstring{$DD$}{DD} relations}

\begin{lemma}\label{lem:ZeroInc} Let $j \neq i, i+1$ be an external strand, and let $m$ be a basis element of $M$. If $m \in C_j$, then $(U_j \otimes U_j) \otimes m = 0$.
\end{lemma}

\begin{proof} Let $m$ be a basis element of $M$ with idempotent $\Ib_{\x_1} \otimes \Ib_{\x_2}$. If $j$ is neither $i-1$ nor $i+2$, then the only way $m$ can be in $C_j$ is to have the two regions $\{j-1,j\}$ adjacent to strand $j$ either both in $\x_1$ or both in $\x_2$. In either case, $(U_j \otimes U_j) \otimes m = 0$.

If $j = i+2$, we would be worried about basis elements $m$ which are in $C_{i+2}$ but such that both $\x_1$ and $\x_2$ contain the region $j-1 = i+1$ immediately to the left of strand $j$ (the rightmost local region). However, by inspection, these basis elements do not exist. A similar analysis applies to $j = i-1$.
\end{proof}

\begin{lemma}\label{lem:ExternalPartialOfDelta} Let $m$ be a basis element of $M$. Then 
\[
((\partial_{\B} \otimes \id) \circ \delta^1_{\external}) (m) = \sum_{j \neq i, i+1} (U_j \otimes U_j) \otimes m.
\]
\end{lemma}

\begin{proof} The only algebra elements with nonzero differential are divisible by $C$ generators, and the only $C$ generators produced by $\delta^1_{\external}$ come from $\delta^1_{\ext, \unmoving}$. This map is a sum of multiplications by $(U_j \otimes C_j)$, for the external strands $j$, and the differential of $(U_j \otimes C_j)$ is $(U_j \otimes U_j)$.
\end{proof}

It will be helpful to introduce the notation 
\[
\delta^1_{\external,j} := \delta^1_{\ext,\unmoving,j} + \delta^1_{\ext,\moving,j}
\]
for an external strand $j$; we have $\delta^1_{\external} = \sum_{j \neq i,i+1} \delta^1_{\external,j}$.

\begin{lemma}\label{lem:SameExternalStrandDelta} If $j \neq i,i+1$ is an external strand, then 
\[
((\mu_{\B} \otimes \id) \circ (\id \otimes \delta^1_{\external,j}) \circ \delta^1_{\external,j}) (m) = (U_j \otimes U_j) \otimes m.
\]
\end{lemma}

\begin{proof} Consider the cases $m \in A_j, B_j, C_j$. If $m \in C_j$, then the left side of the equation we want is zero (the only contribution to $\delta^1_{\external,j}$ is the unmoving component, and applying the unmoving component twice picks up $C_j^2 = 0$). The right side of the equation is also zero by Lemma~\ref{lem:ZeroInc}.

If $m \in A_j$ or $B_j$, then applying the unmoving component of $\delta^1_{\external,j}$ twice still gives zero, as does the sum of the two ways of applying the moving component once and the unmoving component once. Applying the moving component twice gives $(U_j \otimes U_j) \otimes m$ as desired.
\end{proof}

\begin{lemma}\label{lem:TwoExternalStrandsDelta} We have
\[
\sum_{j,k \notin \{i,i+1\}; j \neq k} (\mu_{\B} \otimes \id) \circ (\id \otimes \delta^1_{\external,k}) \circ \delta^1_{\external,j} = 0.
\] 
\end{lemma}

\begin{proof} Let $m$ be a basis element of $M$. The term $(m_{\B} \otimes \id) \circ (\id \otimes \delta^1_{\ext,\unmoving,k}) \circ \delta^1_{\ext,\unmoving,j}$ cancels with the corresponding term with $j$ and $k$ reversed. Similarly, terms $(m_{\B} \otimes \id) \circ (\id \otimes \delta^1_{\ext,\moving,k}) \circ \delta^1_{\ext,\unmoving,j}$ cancel with terms $(m_{\B} \otimes \id) \circ (\id \otimes \delta^1_{\ext,\unmoving,k}) \circ \delta^1_{\ext,\moving,j}$ when $j$ and $k$ are reversed.

The remaining terms are $(m_{\B} \otimes \id) \circ (\id \otimes \delta^1_{\ext,\moving,k}) \circ \delta^1_{\ext,\moving,j}$. First, suppose $|j-k| > 1$. If we have $m \in A_k$, $B_k$, or $C_k$ as well as $m \in A_j$ (resp. $B_j)$, then $\sigma_j(m)$ (resp. $\tau_j(m)$) is also in $A_k$, $B_k$, or $C_k$ respectively. The same is true with $j$ and $k$ reversed. Thus, the contribution from $(j,k)$ is canceled by the one from $(k,j)$.

On the other hand, if $|j-k| = 1$ and $(m_{\B} \otimes \id) \circ (\id \otimes \delta^1_{\ext,\moving,k}) \circ \delta^1_{\ext,\moving,j}(m)$ is nonzero, then the output algebra element must have either $R_2 R_1$ or $L_1 L_2$ on either the right or left side of the $\otimes$ symbol. But then it has $L_2 L_1$ or $R_1 R_2$ on the other side, and these algebra elements are zero.
\end{proof}

\begin{proposition}\label{prop:DDRelsForExternal}
We have
\[
(\partial_{\B} \otimes \id) \circ \delta^1_{\external} + (\mu_{\B} \otimes \id) \circ (\id \otimes \delta^1_{\external}) \circ \delta^1_{\external} = 0.
\]
\end{proposition}

\begin{proof} 
Let $m$ be a basis element of $M$. By Lemma~\ref{lem:TwoExternalStrandsDelta}, we have
\[
(\mu_{\B} \otimes \id) \circ (\id \otimes \delta^1_{\external}) \circ \delta^1_{\external} (m) = \sum_{j \neq i,i+1} (\mu_{\B} \otimes \id) \circ (\id \otimes \delta^1_{\external,j}) \circ \delta^1_{\external,j} (m).
\]
By Lemma~\ref{lem:SameExternalStrandDelta}, this sum is equal to $\sum_{j \neq i,i+1} (U_j \otimes U_j) \otimes m$, which in turn equals $((\partial_{\B} \otimes \id) \circ \delta^1_{\external}) (m)$ by Lemma~\ref{lem:ExternalPartialOfDelta}.
\end{proof}

\subsection{Proof that \texorpdfstring{$\delta^1_{\local}$}{delta\_local} and \texorpdfstring{$\delta^1_{\external}$}{delta\_external} are compatible}

\begin{lemma}\label{lem:InternalExternalCompatibleUnmoving}
We have
\[
(\mu_{\B} \otimes \id) ((\id \otimes \delta^1_{\local}) \circ \delta^1_{\ext,\unmoving} + (\id \otimes \delta^1_{\ext, \unmoving}) \circ \delta^1_{\local}) = 0.
\]
\end{lemma}

\begin{proof} This equation holds since $\delta^1_{\ext,\unmoving}$ can be viewed as a sum of maps which just multiply each basis element of $M$ by a fixed element $\sum_{\Ib} \Ib (U_j \otimes C_j) \Ib$ in the center of the algebra $\B$.
\end{proof}

\begin{lemma}\label{lem:InternalExternalCompatibleFarStrands}
For external strands $j$ which are neither $i-1$ nor $i+2$, we have
\[
(\mu_{\B} \otimes \id) ((\id \otimes \delta^1_{\local}) \circ \delta^1_{\ext,\moving,j} + (\id \otimes \delta^1_{\ext,\moving,j}) \circ \delta^1_{\local}) = 0.
\]
\end{lemma}

\begin{proof} 
Let $m \in M$ be a basis element. If $m \in A_j$, $B_j$, or $C_j$, then the same is true for the $M$-basis element $m'$ of any term $a \otimes m'$ of $\delta^1_{\local}(m)$. Thus, $\delta^1_{\ext,\moving,j}$ acts the same before and after applying $\delta^1_{\local}$. Since the algebra generators $R_j$ and $L_j$ are far enough away from all local algebra generators to commute with them, the equation holds.
\end{proof}

\begin{lemma}\label{lem:InternalExternalCompatibleRHS}
We have
\begin{equation}\label{eq:RHSCompatibility}
(\mu_{\B} \otimes \id) ((\id \otimes \delta^1_{\local}) \circ \delta^1_{\ext,\moving,i+2} + (\id \otimes \delta^1_{\ext,\moving,i+2}) \circ \delta^1_{\local}) = 0.
\end{equation}
\end{lemma}

\begin{proof} 
We extend the notation at the beginning of Section~\ref{sec:LocalDDDefs} from $\{i-1,i,i+1\}$ to $\{i-1,i,i+1,i+2\}$ using the letter $D$, so that subsets of $\{i-1,i,i+1,i+2\}$ are written as subwords of $ABCD$.

In this notation, there are ten types of basis elements of $M$ which lie in $A_{i+2}$, namely 
\begin{align*}
&{_{BC}}W^{BD}_{t}, {_{BC}}W^{BD}_{b}, {_{AC}}(S^-_t)^{BD}, {_{AC}}(S^+_t)^{BD}, {_{C}}(S^-_t)^{ABD}, \\
&{_{C}}(S^+_t)^{ABD}, {_{C}}(S^-_b)^{ABD}, {_{C}}(S^+_b)^{ABD} {_{AC}}(S^-_b)^{BD}, {_{AC}}(S^+_b)^{BD}.
\end{align*}
Label these basis elements $(1)$--$(10)$ in order.

Similarly, ten types of basis elements of $M$ lie in $B_{i+2}$. These are 
\begin{align*}
&{_{BD}}W_t^{BC}, {_{BD}}W_b^{BC}, {_{AD}}(S^-_t)^{BC}, {_{AD}}(S^+_t)^{BC}, {_{D}}(S^-_t)^{ABC}, \\
&{_{D}}(S^+_t)^{ABC}, {_{D}}(S^-_b)^{ABC}, {_{D}}(S^+_b)^{ABC}, {_{AD}}(S^-_b)^{BC}, {_{AD}}(S^+_b)^{BC}.
\end{align*}
Label these basis elements $(1')$--$(10')$ in order.

The twenty-eight remaining types of basis elements of $M$ lie in $C_{i+2}$. The map $\delta^1_{\ext,\moving,i+2}$ is zero on basis elements in $C_{i+2}$, so the same is true for $(\mu_{\B} \otimes \id) ((\id \otimes \delta^1_{\local}) \circ \delta^1_{\ext,\moving,i+2}$. It is straightforward, although a bit tedious, to check that $(\id \otimes \delta^1_{\ext,\moving,i+2}) \circ \delta^1_{\local})$ vanishes on basis elements in $C_{i+2}$; one need only consider terms of $\delta^1_{\local}$ which map basis elements in $C_{i+2}$ to basis elements in $A_{i+2}$ or $B_{i+2}$.

We verify that equation~\eqref{eq:RHSCompatibility} holds for basis elements in $A_{i+2}$ and $B_{i+2}$. For each index $j$ between $1$ and $10$, we have
\[
\delta^1_{\ext,\moving,i+2} (j) = (R_{i+1} \otimes L_{i+1}) \otimes (j')
\]
and
\[
\delta^1_{\ext,\moving,i+2} (j') = (L_{i+1} \otimes R_{i+1}) \otimes (j).
\]

First we will list nonzero terms of $(\mu_{\B} \otimes \id) \circ (\id \otimes \delta^1_{\local}) \circ \delta^1_{\ext,\moving,i+2}(j)$, for $j = 1, \ldots, 10$.

\begin{itemize}
\item For $j = 1$, $(R_{i+2} L_i \otimes L_{i+2}) \otimes (3')$ and $(R_{i+2} L_i \otimes L_{i+2}U_i C_{i+1}) \otimes (4')$.
\item For $j = 2$, $(R_{i+2} U_{i+1} \otimes L_{i+2}) \otimes (1')$ and $(R_{i+2}L_i \otimes L_{i+2}) \otimes (9')$.
\item For $j = 3$, $(R_{i+2} R_i \otimes L_{i+2} U_i C_{i+1}) \otimes (2')$ and $(R_{i+2}U_i \otimes L_{i+2} U_i C_{i+1}) \otimes (10')$.
\item For $j = 4$, $(R_{i+2} R_i \otimes L_{i+2}) \otimes (2')$ and $(R_{i+2}U_i \otimes L_{i+2}) \otimes (10')$.
\item For $j = 5$, all terms are zero.
\item For $j = 6$, all terms are zero.
\item For $j = 7$, all terms are zero.
\item For $j = 8$, $(R_{i+2} \otimes L_{i+2}) \otimes (7')$.
\item For $j = 9$, all terms are zero.
\item For $j = 10$, $(R_{i+2} \otimes L_{i+2}) \otimes (9')$.
\end{itemize}

One can check that the nonzero terms of $(\mu_{\B} \otimes \id) \circ (\id \otimes \delta^1_{\ext,\moving,i+2}) \circ \delta^1_{\local}(j)$, for $j = 1, \ldots, 10$, agree with the above list, showing that equation~\eqref{eq:RHSCompatibility} holds on basis elements in $A_{i+2}$.

Below we list nonzero terms of $(\mu_{\B} \otimes \id) \circ (\id \otimes \delta^1_{\local}) \circ \delta^1_{\ext,\moving,i+2}(j')$, for $j = 1, \ldots, 10$:

\begin{itemize}
\item For $j = 1$, $(L_{i+2} L_i \otimes R_{i+2}) \otimes (3)$ and $(L_{i+2} L_i \otimes R_{i+2} U_i C_{i+1}) \otimes (4)$.
\item For $j = 2$, $(L_{i+2} U_{i+1} \otimes R_{i+2}) \otimes (1)$ and $(L_{i+2}L_i \otimes R_{i+2}) \otimes (9)$.
\item For $j = 3$, $(L_{i+2} R_i \otimes R_{i+2} U_i C_{i+1}) \otimes (2)$ and $(L_{i+2}U_i \otimes R_{i+2} U_i C_{i+2}) \otimes (9)$.
\item For $j = 4$, $(L_{i+2} R_i \otimes R_{i+2}) \otimes (2)$ and $(L_{i+2} U_i \otimes R_{i+2}) \otimes (10)$.
\item For $j = 5$, all terms are zero.
\item For $j = 6$, all terms are zero.
\item For $j = 7$, all terms are zero.
\item For $j = 8$, $(L_{i+2} \otimes R_{i+2}) \otimes (7)$.
\item For $j = 9$, all terms are zero.
\item For $j = 10$, $(L_{i+2} \otimes R_{i+2}) \otimes (9)$. 
\end{itemize}

One can check that these terms are also the nonzero terms of $(\mu_{\B} \otimes \id) \circ (\id \otimes \delta^1_{\ext,\moving,i+2}) \circ \delta^1_{\local}(j')$, for $j = 1, \ldots, 10$. Thus, equation~\eqref{eq:RHSCompatibility} holds on basis elements in $B_{i+2}$, so it holds in general.
\end{proof}

\begin{lemma}\label{lem:InternalExternalCompatibleLHS}
We have
\begin{equation}\label{eq:LHSCompatibility}
(\mu_{\B} \otimes \id) ((\id \otimes \delta^1_{\local}) \circ \delta^1_{\ext,\moving,i-1} + (\id \otimes \delta^1_{\ext,\moving,i-1}) \circ \delta^1_{\local}) = 0.
\end{equation}
\end{lemma}

\begin{proof} 
The proof is similar to that of Lemma~\ref{lem:InternalExternalCompatibleRHS}, but it is not exactly symmetric (note that the unsimplified local bimodule $\X^{DD}$ does not have the symmetry $\Rc$). 

We extend the notation at the beginning of Section~\ref{sec:LocalDDDefs} from $\{i-1,i,i+1\}$ to $\{i-2,i-1,i,i+1,i+2\}$ using the letter G, so that subsets of $\{i-2,i-1,i,i+1\}$ are written as subwords of $GABC$ (we use the C major scale to find a letter that can reasonably be said to precede A). In this notation, there are ten types of basis elements of $M$ which lie in $A_{i-1}$, namely
\begin{align*}
& {_{GB}}E_t^{AB}, {_{GB}}E_b^{AB}, {_{GC}}(S^-_t)^{AB}, {_{GC}}(S^+_t)^{AB}, {_{G}}(S^-_t)^{ABC}, \\
& {_{G}}(S^+_t)^{ABC}, {_{G}}(S^-_b)^{ABC}, {_{G}}(S^+_b)^{ABC}, {_{GC}}(S^-_b)^{AB}, {_{GC}}(S^+_b)^{AB}. 
\end{align*}
Label these basis elements $(1)$--$(10)$ in order.

There are also ten types of basis elements in $B_{i-1}$, namely
\begin{align*}
& {_{AB}}E_t^{GB}, {_{AB}}E_b^{GB}, {_{AC}}(S^-_t)^{GB}, {_{AC}}(S^+_t)^{GB}, {_{A}}(S^-_t)^{GBC}, \\
& {_{A}}(S^+_t)^{GBC}, {_{A}}(S^-_b)^{GBC}, {_{A}}(S^+_b)^{GBC}, {_{AC}}(S^-_b)^{GB}, {_{AC}}(S^+_b)^{GB}.
\end{align*}
Label these basis elements $(1')--(10')$ in order.

The twenty-eight remaining types of basis elements of $M$ are in $C_{i-1}$. As in Lemma~\ref{lem:InternalExternalCompatibleRHS}, one can check that equation~\eqref{eq:LHSCompatibility} holds for these basis elements.

The results of applying $(\mu_{\B} \otimes \id) \circ (\id \otimes \delta^1_{\local}) \circ \delta^1_{\ext,\moving,i-1}(j)$ to the basis elements of type $(j)$ for $j \in 1, \ldots, n$ are given as follows:
\begin{itemize}
\item For $j = 1$, $(R_{i-1} U_i \otimes L_{i-1}) \otimes (2')$.
\item For $j = 2$, $(R_{i-1} R_{i+1} \otimes L_{i-1}) \otimes (3')$ and $(R_{i-1} R_{i+1} \otimes C_i U_{i+1} L_{i-1}) \otimes (4')$.
\item For $j = 3$, $(R_{i-1} L_{i+1} \otimes C_i U_{i+1} L_{i-1}) \otimes (1')$.
\item For $j = 4$, $(R_{i-1} L_{i+1} \otimes L_{i-1}) \otimes (1')$.
\item For $j = 5$, all terms are zero.
\item For $j = 6$, all terms are zero.
\item For $j = 7$, all terms are zero.
\item For $j = 8$, $(R_{i-1} \otimes L_{i-1}) \otimes (7')$.
\item For $j = 9$, $(R_{i-1}U_{i+1} \otimes L_{i-1}) \otimes (3')$ and $(R_{i-1} U_{i+1} \otimes L_{i-1} C_i U_{i+1}) \otimes (4')$.
\item For $j = 10$, $(R_{i-1} L_{i+1} \otimes L_{i-1}) \otimes (2')$ and $(R_{i-1} \otimes L_{i-1}) \otimes (9')$.
\end{itemize}

These terms are also the results of applying $(\mu_{\B} \otimes \id) \circ (\id \otimes \delta^1_{\ext,\moving,i-1}) \circ \delta^1_{\local}$ to the basis elements of type $(j)$, so equation~\eqref{eq:LHSCompatibility} holds for the basis elements in $A_{i-1}$.

For the basis elements in $B_{i-1}$, namely $(j')$ for $j \in 1, \ldots n$, the results of applying $(\mu_{\B} \otimes \id) \circ (\id \otimes \delta^1_{\local}) \circ \delta^1_{\ext,\moving,i-1}(j)$ are given as follows.
\begin{itemize}
\item For $j = 1$, $(L_{i-1} U_i \otimes R_{i-1}) \otimes (2)$.
\item For $j = 2$, $(L_{i-1} R_{i+1} \otimes R_{i-1}) \otimes (3)$ and $(L_{i-1}R_{i+1} \otimes R_{i-1} C_i U_{i+1}) \otimes (4)$.
\item For $j = 3$, $(L_{i-1} L_{i+1} \otimes R_{i-1} C_i U_{i+1}) \otimes (1)$.
\item For $j = 4$, $(L_{i-1} L_{i+1} \otimes R_{i-1}) \otimes (1)$.
\item For $j = 5$, all terms are zero.
\item For $j = 6$, all terms are zero.
\item For $j = 7$, all terms are zero.
\item For $j = 8$, $(L_{i-1} \otimes R_{i-1}) \otimes (7)$.
\item For $j = 9$, $(L_{i-1} U_{i+1} \otimes R_{i-1}) \otimes (3)$ and $(L_{i-1} U_{i+1} \otimes R_{i-1} C_i U_{i+1}) \otimes (4)$.
\item For $j = 10$, $(L_{i-1} L_{i+1} \otimes R_{i-1}) \otimes (2)$ and $(L_{i-1} \otimes R_{i-1}) \otimes (9)$.
\end{itemize}

These terms are also the terms of $(\mu_{\B} \otimes \id) \circ (\id \otimes \delta^1_{\ext,\moving,i-1}) \circ \delta^1_{\local}$ applied to the basis elements in $B_{i-1}$. Thus, equation~\ref{eq:LHSCompatibility} holds.
\end{proof}

\begin{proposition}\label{prop:InternalExternalCompatibility}
We have
\[
(\mu_{\B} \otimes \id) ((\id \otimes \delta^1_{\local}) \circ \delta^1_{\external} + (\id \otimes \delta^1_{\external}) \circ \delta^1_{\local}) = 0.
\]
\end{proposition}

\begin{proof} 
This result follows from combining Lemmas \ref{lem:InternalExternalCompatibleUnmoving}, \ref{lem:InternalExternalCompatibleFarStrands}, \ref{lem:InternalExternalCompatibleRHS}, and \ref{lem:InternalExternalCompatibleLHS}.
\end{proof}

\begin{corollary}
The $DD$ operation $\delta^1$ on $\X_i^{DD}$ satisfies the $DD$ bimodule relation.
\end{corollary}

\begin{proof}
The $DD$ bimodule relation is a consequence of Propositions \ref{prop:HalfLocalDD}, \ref{prop:DDRelsForExternal}, and \ref{prop:InternalExternalCompatibility}.
\end{proof}

\bibliographystyle{fouralpha}
\bibliography{biblio}

\end{document}